\documentclass{amsart}
\usepackage{marginnote}     % marginnotes; remove in final version
\usepackage[english]{babel}
\usepackage{graphicx}
\usepackage{amssymb}
\usepackage{array}
\usepackage{subfigure}
\usepackage{amsfonts}
\usepackage{amsmath,amsthm}
\usepackage{color}
\usepackage{float}
\newtheorem{thm}{Theorem}[section]
\newtheorem{lem}{Lemma}[section]
\newtheorem{df}{Definition}[section]
\newtheorem{prop}{Proposition}[section]
\newtheorem{rem}{Remark}[section]

\catcode`\@=11

\@addtoreset{equation}{section}

\usepackage{enumitem}        % generalizes list enviroments
\setenumerate[0]{label=(\roman*)}
\usepackage{stackengine}
\stackMath

\newcommand{\R}{\mathbb{R}}
\newcommand{\N}{\mathbb{N}}
\newcommand{\DD}{\mathcal{D}}
\newcommand{\LL}{\mathcal{L}}
\newcommand{\FF}{\mathcal{F}}

\newcommand{\e}{\varepsilon}

\begin{document}
\bibliographystyle{amsplain}%{abbrv}%{unsrt}
\title[Non-local KdV-Burgers undercompressive waves]{Existence of undercompressive travelling waves of a non-local generalised Korteweg-de Vries-Burgers equation
}

\author{F. Achleitner}
\address{Vienna University of Technology, 
Institute for Analysis and Scientific Computing, Wiedner Hauptstrasse 8-10, 1040 Wien, Austria.}
\email{franz.achleitner@tuwien.ac.at}

\author{C.~M. Cuesta}
\address{University of the Basque Country (UPV/EHU), Faculty of Science and Technology, Department of Mathematics, Aptdo. 644, 48080 Bilbao, Spain.}
\email[corresponding author]{carlotamaria.cuesta@ehu.eus}

\author{X. Diez-Izagirre}
\address{University of the Basque Country (UPV/EHU), Faculty of Economics and Business, Department of Applied Economics, Oñati plaza 1, 20018 Donostia-San Sebastian, Spain.}
\email{xuban.diez@ehu.eus}

\thanks{F. Achleitner was supported by the Austrian Science Fund (FWF) via the FWF-funded SFB \# F65. C. M. Cuesta and X. Diez-Izagirre thank the support by the grant PID2021-126813NB-I00 funded by MICIU/AEI/10.13039/501100011033 and by ``ERDF A way of making Europe'' and the financial support of the Basque Government through the Research Group IT1615-22. X. Diez-Izagirre was also supported by the Basque Government through the pre-doctoral grant PRE-2018-2-0013.
}

%\date{\today}
%\maketitle
\begin{abstract}
We study travelling wave solutions of a generalised Korteweg-de Vries-Burgers equation with a non-local diffusion term and a concave-convex flux. This model equation arises in the analysis of a shallow water flow by performing formal asymptotic expansions associated to the triple-deck regularisation (which is an extension of classical boundary layer theory). The resulting non-local operator is a fractional type derivative with order between $1$ and $2$. Travelling wave solutions are typically analysed in relation to shock formation in the full shallow water problem. We show rigorously the existence of travelling waves that, formally, in the limit of vanishing diffusion and dispersion would give rise to non-classical shocks, that is, shocks that violate the Lax entropy condition. The proof is based on arguments that are typical in dynamical systems. The nature of the non-local operator makes this possible, since the resulting travelling wave equation can be seen as a delayed integro-differential equation. Thus, linearisation around critical points, continuity with respect to parameters and a shooting argument, are the main steps that we have proved and adapted for solving this problem.
\end{abstract}

\maketitle

\emph{Keywords.} non-local evolution equation, fractional derivative, travelling waves, non-classical shocks

\emph{Math.Subj.Class.} 47J35, 26A33, 35C07

\tableofcontents

\section{Introduction}
In this paper we study existence of undercompressive travelling waves of the following one-dimensional evolution equation:
\begin{equation}\label{EE}
  \partial_t u + \partial_x u^3 = \partial_x \DD^{\alpha}[ u] +\tau \partial_x^3 u\,, \quad x\in \R\,, \ t\geq 0
\end{equation}
with $\tau >0$ and $\DD^\alpha[\cdot]$ is the non-local operator here applied to a real valued function $g$,
\begin{equation}\label{FD}
\DD^\alpha [g](x)=d_\alpha\int_{-\infty}^x \frac{g'(y)}{(x-y)^{\alpha}}\,dy \,,
\quad \mbox{with} \quad 
0<\alpha<1\,, \quad d_\alpha = \frac{1}{\Gamma(1-\alpha)}>0 \,.
\end{equation}
Observe that this operator only acts on the variable $x$. Here $\Gamma$ denotes the Gamma function.

Equation (\ref{EE}) with $\alpha=1/3$ and a flux function ($u^3$ in (\ref{EE})) which is either a quadratic or a cubic polynomial, has been derived from one (quadratic flux) and two (cubic flux) layer shallow water flows by performing formal asymptotic expansions in the framework of the triple-deck boundary 
layer theory (see, e.g. \cite{Kluwick2018}, \cite{Kluwick2010} and \cite{NV}). In \cite{NV} numerical simulations suggest the existence of travelling waves that violate the entropy condition in the two-layer model. Such solutions resemble the inner structure, in such a particular limit, of small amplitude shock waves for the original shallow water problem, that are not entropic or non-classical (see e.g. \cite{lefloch1}). In this manuscript we aim to study rigorously the existence of such travelling wave solutions for (\ref{EE}), which has a cubic flux and thus corresponds to the two-layer model when $\alpha=1/3$. 

We recall, that existence and stability of travelling waves for (\ref{EE}) with a quadratic flux function have been established in \cite{AHS, AHS2} ($\tau=0$) and in \cite{ACH} ($\tau>0$), see also \cite{CA}. The results of \cite{AHS} and \cite{ACH} form part of building blocks in the proof of our main theorem, as we shall describe below.

We recall that hyperbolic conservation laws exhibit discontinuous solutions, whose discontinuities or shocks may travel with constant speed. These solutions belong to the class of weak solutions. An extra condition is necessary to select a unique solution for the Cauchy problem. The most common way to derive uniqueness conditions is to use vanishing diffusion arguments (see e.g. \cite{Smoller}). In particular, for scalar conservation laws admissible shocks result from constructing travelling wave solutions of the regularised parabolic equation. There are, however, other types of natural regularisations of hyperbolic conservation laws. Moreover, different regularisations of hyperbolic conservation laws might lead to different weak solutions. Examples describing this situation are scalar conservation laws with a non-genuinely nonlinear flux (neither convex nor concave), such as a cubic flux. In this case, shocks violating the classical Lax entropy condition, see \cite{JMcS1995}, can be constructed by introducing a diffusive-dispersive regularisation. This kind of solutions is defined by Hayes and LeFloch \cite{lefloch1} as non-classical shock waves, others refer to them as undercompressive, see \cite{ElHoeferSheafer2017} and the references therein. Our aim is thus to show the existence of such solutions for the non-local version of these regularisations (\ref{EE}). For a complementary study of non-classical shock waves for a scalar conservation law with local diffusion and non-local dispersion see \cite{Rohde2005}.

We notice that the parameter $\tau$ results from a choice in the rescaling. Analogous to \cite{JMcS1995} we can consider the equation in the following form 
\[
\partial_t u + \partial_x u^3 = \e\partial_x \DD^{\alpha}[ u] +\delta \partial_x^3 u\,, \quad x\in \R\,, \ t\geq 0
\]
where $\e$ and $\delta$ are positive constants that act as control parameters.  This means that depending on their relative size either diffusion ($\delta\ll \e^{2/\alpha}$) or dispersion ($\delta\gg \e^{2/\alpha}$) dominates in the limit of both $\e$ and $\delta \to 0$. The parameter $\tau$ results from the scaling $(x,t)\to(\e^{1/\alpha} x, \e^{1/\alpha} t)$ so that $\tau= \delta/\e^{2/\alpha}$. It is when this parameter is of order one when we expect to get solutions that violate the entropy condition. 

We introduce the travelling wave variable $\xi=x-ct$ with wave speed $c$ and 
look for solutions $u(x,t)=\phi(\xi)$ of (\ref{EE}) which connect 
two different far-field real values $\phi_-$ and $\phi_+$. A straightforward 
calculation shows that if $\phi$ depends on $x$ and $t$ only through the 
travelling wave variable, then so does $\DD^{\alpha}[\phi]$, and the travelling wave problem becomes
\begin{equation}\label{tw:row}
  -c\phi' + (\phi^3)' = (\DD^{\alpha} [\phi])'+\tau\phi''' .
\end{equation}
Here $'$ denotes differentiation with respect to $\xi$. We can then integrate 
(\ref{tw:row}) with respect to $\xi$ to arrive at the following travelling wave equation:
\begin{equation}\label{TWP}
 \tau\phi''+ \DD^\alpha[ \phi]=h(\phi)\,,
\quad \mbox{where} \quad h(\phi) := -c(\phi-\phi_-)+\phi^3-\phi_-^3 \,,
\end{equation}
where we have used
\begin{equation}\label{far-fieldL}
\lim_{\xi\to-\infty}\phi(\xi) = \phi_-
\end{equation}
and 
\begin{equation}\label{far-fieldR}
 \lim_{\xi\to \infty}\phi(\xi) = \phi_+ \,.
\end{equation}
If $\phi'$ decays to zero fast enough as $\xi\to\pm \infty$, then we obtain the Rankine-Hugoniot condition
\begin{equation}\label{RHC}
   c= \frac{\phi_+^3 - \phi_-^3}{\phi_+ - \phi_-}=\phi_+^2 +\phi_-^2+\phi_-\phi_+\,,
\end{equation}
that we assume throughout this manuscript.

One expects that travelling wave solutions correspond to classical shock waves in the limit of the diffusive and dispersive terms tending to zero (in the right order -dispersion followed by diffusion tending to zero- or at the right asymptotic rate) if the Lax entropy condition \cite[Chapter~II.1]{LeFlochBook} is satisfied, which for the current non-linear flux reads:
\begin{equation}\label{entropy:cond}
3\phi_+^2<  c< 3\phi_-^2\,.
\end{equation}
In this manuscript, however, we investigate the existence of travelling wave solutions that do not satisfy (\ref{entropy:cond}). In particular, we shall look for solutions that satisfy
\begin{equation}\label{no:entropy:cond}
  c< 3\min\{\phi_-^2,\phi_+^2\}\,.
\end{equation}

We assume without loss of generality that $\phi_+ < \phi_-$ (observe that the equation is invariant under the change $\phi\to -\phi$). 

The roots of $h(\phi)$ are $\phi_+$, $\phi_-$ and $\phi_c:=-(\phi_-+\phi_+)$. We require that 
\begin{equation}\label{lin:ass}
\phi_+<\phi_c <\phi_-
\end{equation}
so that
\begin{equation}\label{sign:h:prime}
h'(\phi_-)>0\,,\quad h'(\phi_c)<0 \quad \mbox{and} \quad h'(\phi_+)>0\,,
\end{equation}
which is equivalent to the condition (\ref{no:entropy:cond}).
We observe that this in particular implies that $\phi_->0$ and $\phi_+<0$.

Under these assumptions, travelling wave solutions of (\ref{TWP}) with (\ref{far-fieldL}) and (\ref{far-fieldR}) correspond to {\it non-classical shocks}
in the sense described earlier. On the other hand, solutions that satisfy (\ref{TWP}) with (\ref{far-fieldL}) and
\begin{equation}\label{far-fieldC}
\lim_{\xi\to \infty}\phi(\xi) = \phi_c 
\end{equation}
correspond to classical shocks (with the same wave speed). We expect, as in the local case $\alpha=1$ (see e.g. \cite{JMcS1995}), the possibility of solutions satisfying (\ref{far-fieldR}) to be a distinguished limit in the sense that there is a unique value of $\tau$ that allows such connection, whereas there is an open set of values of $\tau$ that allows solutions to satisfy (\ref{far-fieldC}). We recall that this last possibility corresponds to the classical shock admissibility condition (\ref{entropy:cond}) if we replace $\phi_+$ by $\phi_c$ in the notation. Moreover, if the only zeros of $h$ where $\phi_-$ and $\phi_c$ with (\ref{sign:h:prime}), this would be the only possible travelling wave solutions that can be constructed in both the local and the non-local case (see \cite{ACH}).

There is a further necessary condition for existence of (\ref{TWP})-(\ref{far-fieldR}) on the values $\phi_-$ and $\phi_+$, namely that $\phi_+ + \phi_->0$, which implies that $c>0$ (see (\ref{RHC})), in particular. We shall show this below (see Lemma~\ref{nec:cond:farfield}), and this is a consequence of the results of \cite{ACH}. 

Let us state our main theorem:
\begin{thm}[Existence of undercompressive travelling waves]\label{TW:main:theorem}
  Let $\phi_-$ and $\phi_+\in\R$ such that (\ref{lin:ass}) with $\phi_c=-( \phi_+ + \phi_-)$ holds and such that
  \begin{equation}\label{nec:cond0}
  \phi_+ + \phi_->0.
  \end{equation}
  Then, there exists $\tau>0$ such that (\ref{TWP})-(\ref{far-fieldR}) has a unique solution (up to a shift in $\xi$) in $C_b^3(\R)$.
\end{thm}

The uniqueness of $\tau$ is an open question.

The paper is organised as follows. In Section~\ref{sec:pre:results}, we give some preliminary results. First, in Section~\ref{sec:pre:D:estimates}, we give some lemmas with estimates related to the non-local operator. In Section~\ref{sec:existence}, we establish the existence of solutions of (\ref{TWP}) that satisfy (\ref{far-fieldL}) and prove that there are only three possible behaviours that such trajectories will have as $\xi$ increases: they tend to $\phi_c$ or to $\phi_+$ as $\xi\to\infty$, or they tend to $-\infty$ as $\xi\to(\xi^*)^-$ for some finite value $\xi^*$. This allows us to approach the problem via a shooting argument, with shooting parameter $\tau$. With this in mind, in Section~\ref{sec:main} we define three disjoint sets of $\tau>0$: $\Sigma_c$ (the profiles tend to $\phi_c$, thus for such $\tau$'s the waves are {\it classical}), $\Sigma_+$ (the profiles tend to $\phi_+$ and the waves are {\it non-classical}) and $\Sigma_u$ (the profiles tend to $-\infty$, and in fact `blow up' at a finite value of $\xi$). Then we prove Theorem~\ref{TW:main:theorem} by first showing that $\Sigma_u$ is non-empty and open and that $\Sigma_c$ is non-empty. Then, we argue by contradiction assuming that $\Sigma_+$ is empty, which in particular implies that $\Sigma_c$ must be closed. We arrive to a contradiction by applying continuity with respect to $\tau$, and hence $\Sigma_+$ must be non-empty, which proves the behaviour of the desired solutions and their existence, and thus our main theorem.

  We note that the proof of $\Sigma_c$ being non-empty is very involved: the idea is to use that for the quadratic case (see \cite{ACH}) where only classical solutions exist, it was conjectured that for small enough values of $\tau$ the travelling waves are monotone decreasing. This is still a conjecture, but if this was true, we can look at the problem for a modification of our non-linearity $h$ such that it is $C^2$ smooth, but that it has only the two zeros $\phi_-$ and $\phi_c$ and coincides with $h$ in an interval of $\phi$ containing these values.  
  This means that for $\tau$ small enough such that the solutions are monotone, both problems coincide (recall (\ref{lin:ass})), and hence solutions of the original problem for such small values of $\tau$ must tend to $\phi_c$ as $\xi\to\infty$, and therefore $\Sigma_c$ is non-empty. We do not prove monotonicity for the modified problem, but it is enough to prove that for small values of $\tau$ the profile solutions remain in the interval where both non-linearities coincide, we show this in Section~\ref{sec:control}. Additionally, in this section, for the sake of completeness, we show that if the profiles of the modified problem are monotone decreasing (and, by similar arguments, for the problem with a quadratic non-linearity), they decay to $\phi_c$ like $\xi^{-\alpha}$ as $\xi\to\infty$, as for the case $\tau=0$ (see \cite{DC}).

In Section~\ref{sec:numerics}, we give numerical evidence to illustrate the role of $\tau$ in the behaviour of solutions of (\ref{TWP}) satisfying (\ref{far-fieldL}). Indeed, we adapt the value of $\tau$ recursively between what it appears to be the two generic behaviours $\phi\to\phi_c$ and $\phi\to -\infty$, to capture (as long as machine accuracy allows) the distinguished behaviour $\phi\to\phi_+$. We do this for two values of $\alpha$ for illustrative purposes. 

In Appendix~\ref{appendix:B}, we prove continuous dependence with respect to $\tau$ on finite intervals of existence. In fact, we only have to formulate our equation as a functional differential equation and apply the results of \cite{HVL}.

Finally, we point out that for the proof we need some results on the linearised equations around $\phi_-$ (some of them are already given in \cite{ACH}, see also \cite{CA}) and most importantly around $\phi_c$. In Appendix~\ref{roots} we give pertinent information about the roots of the characteristic equations associated to the linearised problems around $\phi_-$ and $\phi_c$.

In particular, we will use repeatedly the following formulation for solutions that decay to $\phi_c$, or that are close to them in some sense, 
\[%begin{equation}\label{linzed:eq:phic}
\tau \phi'' + \DD^\alpha_{\xi_0}[\phi] - h'(\phi_c) \phi = h(\phi) - h'(\phi_c) \phi - d_\alpha \int_{-\infty}^{\xi_{0}}\frac{\phi'(y)}{(\xi-y)^\alpha} \, dy\,.
\]%end{equation}
where we have split the non-local term by integrating separately on $(-\infty,\xi_0)$ and on $(\xi_0,\xi)$ for some value $\xi_0<\xi$. In particular, we use the notation,
\[
\DD^\alpha_{\xi_0}[\phi]:= d_\alpha \int_{\xi_{0}}^\xi\frac{\phi'(y)}{(\xi-y)^\alpha} \, dy\,,
\]
which is a Caputo derivative of order $\alpha\in(0,1)$ (see e.g. \cite{BK} and \cite{GM2}). 

In this way and thanks to (\ref{sign:h:prime}), or $h'(\phi_c)<0$, we can treat the right-hand side as given and apply a variation of constants type of formulation (see e.g. \cite{BK} and \cite{GM2}) to obtain the solution implicitly. The study of the corresponding linear equations (the inhomogeneous and the homogeneous one) is done in the Appendix~\ref{appendix:v}. In particular, we analyse the fundamental solutions of the homogeneous equation: we give its behaviour near the initial conditions and in the far-field, and study its behaviour in several ranges of its domain as $\tau\to 0^+$ (concluding monotonicity for small values of $\tau$, in particular). These estimates are crucial to show that $\Sigma_c$ is non-empty.

%---------------------------------------------------------------

\section{Preliminary results}\label{sec:pre:results}		

\subsection{The non-local operator and some elementary lemmas}\label{sec:pre:D:estimates}
Let us first recall some basic properties of the fractional differential operator 
$\DD^{\alpha}[\cdot]$. Since it can be written as the convolution of $g'$ with 
$\theta(x)x^{-\alpha}/\Gamma(1-\alpha)$ (where $\theta$ is the Heaviside function), $\DD^{\alpha}[\cdot]$ 
is a pseudo-differential operator with symbol
\begin{equation}\label{symbol}
\frac{ik\sqrt{2\pi}}{\Gamma(1-\alpha)}  \FF\left( \frac{\theta(x)}{x^\alpha}\right)(k)
 = \left(b_\alpha+i a_\alpha \, \mbox{sgn}(k) \right)|k|^\alpha \,,
\end{equation}
i.e. 
$\FF(\DD^{\alpha}[u])(k)=\left(b_\alpha+i a_\alpha \, sgn (k) \right)|k|^\alpha \hat{u}(k)$ 
where $\FF$ denotes the Fourier transform
\[
  \FF(\varphi)(k) = \hat{\varphi}(k) = \frac{1}{\sqrt{2\pi}}\int_\R e^{-i k x}\varphi(x) \,dx \,,
\]			
and the coefficients $ a_\alpha$ and $b_\alpha$ are given by
\begin{equation*}%\label{a:b}
  a_\alpha = \sin\left(\frac{\alpha\pi}2\right)>0 \,,\qquad
  b_\alpha = \cos\left(\frac{\alpha\pi}2\right)>0 \,, \quad \mbox{for}\quad 0<\alpha<1, 
\end{equation*}
(we refer to \cite{AA} for the details of the computation to obtain (\ref{symbol})). 
The operator on the right-hand side of (\ref{EE}) is then a pseudo-differential operator with symbol
\begin{equation}\label{FT}
\FF(\partial_x\DD^\alpha) = -\left(a_\alpha-i b_\alpha \, sgn (k) \right)|k|^{\alpha+1} \,,
\end{equation}
which is dissipative in the sense that the real part of (\ref{FT}) is negative.

For $s\geq 0$ we shall adopt the following notation for the Sobolev space of square integrable functions,
\[
H^s := \{ u \in L^2(\R):\,\|u\|_{H^s} <\infty \} \,,\qquad \|u\|_{H^s} := \|(1+|k|^2)^{s/2}\hat{u}\|_{L^2(\R)} \,,
\]
and the corresponding homogeneous norm
\[
\|u\|_{\dot H^s} := \| |k|^s\hat{u} \|_{L^2(\R)} \,.
\]
Using that $(a_\alpha^2+b_\alpha^2)=1$ it is easy to see that 
$\|\DD^\alpha u\|_{\dot H^s} = \|u\|_{\dot H^{s+\alpha}}$, and this suggests that one can interpret 
$\DD^\alpha[\cdot]$ as a derivative of order $\alpha$. We also observe 
that $\DD^\alpha[\cdot]$ is a bounded linear operator from $H^s$ to $H^{s-\alpha}$, for all $s\geq 1$.

For $m\in \N_{\geq 0}$, let $C^m_b(\R)$ denote the set of functions, whose derivatives up to order $m$ are continuous and bounded. Then one can also infer that $\DD^{\alpha}[\cdot]$ is a bounded linear operator from $C_b^1(\R)$ to $C_b(\R)$. As explained in \cite{AHS}, this can be easily seen by splitting the domain of integration in (\ref{FD}) into $(-\infty,x-M]$ and $[x-M, x]$ for some positive $M>0$. Then integration by parts in the first integral shows the boundedness of $\DD^{\alpha}[\cdot]$. Moreover, we will need the following improved estimate: 

%\section{Bounds on $|D^\alpha[\phi](\xi)$}

\begin{lem}\label{bound:Da} For $\alpha \in (0,1)$, let $x\in\R$ and $g\in C_b^1(-\infty,x)$, then for every $z\in\R$ with $z\leq x$
\[
|\DD^{\alpha}[g](z)| \leq C_\alpha \left(\sup_{y\in(-\infty,z]}|g(y)|\right)^{1-\alpha}\left(\sup_{y\in(-\infty,z]}|g'(y)|\right)^\alpha
\]
where 
\[
C_\alpha = d_\alpha\left( \frac{2(2\alpha)^{-\alpha}}{1-\alpha}\right)\,. %= d_\alpha\left( \frac{(2\alpha)^{1-\alpha}}{1-\alpha}+2(2\alpha)^{-\alpha}\right)
\]
In particular, if $g \in C_b^1(\R)$, then $\DD^{\alpha}[g]\in C_b(\R)$ with
\[
\|\DD^{\alpha}[g]\|_\infty\leq  C_\alpha \, \| g'\|_\infty^{\alpha}\,\|g\|_\infty^{1-\alpha}\,.
\]
\end{lem}
\begin{proof}
The proof is similar to the estimate for Riesz-Feller operators, see e.g. \cite[Proposition~2.4]{AK2}. Let $z\leq x$, and let us denote, for simplicity, 
\[
A=\left(\sup_{y\in(-\infty,z]}|g(y)|\right) \quad\mbox{and}\quad A'=\left(\sup_{y\in(-\infty,z]}|g'(y)|\right)\,.
\]
Then
\begin{equation}\label{boundDa:1}
\begin{split}
| \DD^\alpha[g](z)| =& d_\alpha \left|  \int_0^\infty \frac{g'(z-s)}{s^\alpha} \,ds \right| \\ 
\leq& d_\alpha \left|\int_0^M  \frac{g'(z-s)}{s^\alpha} \,ds\right| + d_\alpha \left|\int_M^\infty \frac{g'(z-s)}{s^\alpha} \,ds \right|\,.
\end{split}
\end{equation}
We estimate the first integral by taking the supremum in $g'$ and computing the remaining integral, thus
\[
\left|\int_0^M  \frac{g'(z-s)}{s^\alpha}\,ds\right|\leq  A' \,
\int_0^M \frac{ds}{s^\alpha} =  \frac{A'}{1-\alpha}M^{1-\alpha}.
\]
In the second integral, we first integrate by parts and pull out the supremum of $g$ to deduce
\[
\left|\int_M^\infty  \frac{g'(z-s)}{s^\alpha} \,ds\right| \leq \alpha \left|\int_M^\infty  \frac{g(z-s)}{s^{\alpha+1}} \,ds\right| + A M^{-\alpha} \leq 2 A M^{-\alpha}.  % A \,   \int_0^M \frac{ds}{s^\alpha} = A \frac{M^{1-\alpha}}{1-\alpha},
\]
Using these estimates in (\ref{boundDa:1}) gives
\begin{equation}\label{boundDa:2}
| \DD^\alpha[g](z)| \leq d_\alpha \left(   \frac{A'}{1-\alpha} M^{1-\alpha}+  2 A M^{-\alpha} \right)\,.
\end{equation}
An easy computation shows that the minimum of the right-hand side of (\ref{boundDa:2}) is attained at $M=2\alpha A/A'$ and this implies the first statement. The second is a consequence of the first by taking the supremum over all values in $\R$.
\end{proof}

In some instances we shall need to split the integral operator (\ref{FD}) as follows
\begin{equation}\label{FD:split}
\DD^\alpha [g](x)=d_\alpha\int_{-\infty}^{x_0}\frac{g'(y)}{(x-y)^{\alpha}} \,dy 
+ d_\alpha \int_{x_0}^x\frac{g'(y)}{(x-y)^{\alpha}}\,dy \,,
\quad \mbox{for some} \quad x_0<x\,, 
\end{equation}
and treat the first term as a known function, whereas the second one can be viewed as a 
left-sided Caputo derivative, see e.g. \cite{Kilbas}, which we denote by $\DD_{x_0}^\alpha[\cdot]$, indicating that the integration is from a finite value $x_0$, i.e. $g \in C^1_b([x_0,\infty))$ and $\alpha\in(0,1]$.

Notice that the first term on the right-hand side of (\ref{FD:split}), which is a function of $x$, is not equal to $\DD^\alpha [g](x_0)$, which is a number for fixed $x_0$.

We shall use the following technical lemma several times:
\begin{lem}\label{h:H:bounds}
  For all $\phi\leq -\phi_-$($<0$), and defining  $H(\phi)=\int_0^\phi h(y)dy$, we have
\[
 2 \phi^3 \leq  h(\phi) < C_h \phi^3 (<0)\quad \mbox{and}\quad (0<)H( \phi)-H(\phi_-)< C_H \phi^4
\]
where 
\[
0<C_h \leq -\frac{2(\phi_-+\phi_+)\phi_+}{(\phi_-)^2}(<1)\quad \mbox{and}\quad C_H\leq 2 \,. 
\]
\end{lem}
The proof is an elementary calculus exercise. It is important to recall that constants here only depend on $\phi_+$ and $\phi_-$. %(SEE VERSION NOV 2018).

%---------------------------------------------------------------
\subsection{Existence of trajectories that satisfy (\ref{far-fieldL}) and the derivation of (\ref{nec:cond0})}\label{sec:existence}

In this section we prove the existence of solutions of (\ref{TWP}) that satisfy (\ref{far-fieldL}) as $\xi\to -\infty$. Next, we prove that there are three possible behaviours of such trajectories as $\xi$ becomes large. The existence of these trajectories follows directly from the results of \cite{ACH}, this means that we shall not need to prove some steps, although we recall their proofs for completeness.

Namely, one obtains the following theorem by a direct application of the previous results and a soft argument for the unbounded case.

\begin{thm}\label{existence} Given $\tau>0$, $\phi_-$ and $\phi_+\in\R$ such that (\ref{lin:ass}) with $\phi_c=-( \phi_+ + \phi_-)<0$ holds. Then,
\begin{enumerate}
\item  There exists a solution $\phi\in C^3(-\infty,0)$ of (\ref{tw:row}) satisfying
\[
\lim_{\xi\to-\infty} \phi(\xi)=\phi_-
\]
and $\phi^\prime(\xi)<0$, for all $\xi\in(-\infty,0)$, that is unique (up to a shift in $\xi$) among all $\phi\in \phi_-+ H^2(-\infty,0) \cap C_b^3(-\infty,0)$.

\item Such solutions satisfy $\phi(\xi)<\phi_-$ for all $\xi$ in the interval of existence.

\item If such solutions are uniformly bounded, they exist for all $\xi\in\R$ and $\lim_{\xi\to\infty} \phi(\xi)\in\{\phi_+,\phi_c\}$. Otherwise, there exists a finite $\xi^*\in \R$ such that $\lim_{\xi\to(\xi^*)^-}\phi(\xi)=-\infty$.
\end{enumerate}
\end{thm}

The proofs of (i), (ii) and the first part of (iii) are a consequence of the results in \cite{ACH}. We recall some of the steps of the proof for these parts below. Thus, it remains to prove the second statement of (iii): that unbounded solutions cannot be extended to the whole $\R$. In the latter case we first exclude the oscillatory behaviour (Lemma~\ref{noosci} below), i.e. we show that the limit is $-\infty$, and then we prove that this limit is reached at a finite value of $\xi$ (Lemma~\ref{blow-down}). 

Let us first summarise the implications of the results from \cite{ACH} in the proof of Theorem~\ref{existence}. For (i), one shows a `local' existence result \cite[Lemma~2]{ACH} on $(-\infty,\tilde{\xi}]$ with $\tilde{\xi}<0$ and $|\tilde{\xi}|$ sufficiently large, that is based on linearisation about $\phi=\phi_-$ as $\xi=-\infty$. For $\tau\geq 0$, all solutions of the linearised equation 
\begin{equation} \label{LTW}
 \tau v''+\DD^\alpha [v]= h'(\phi_-)v\,,
\end{equation}
in the spaces $H^s(-\infty,\tilde\xi)$ with $s\geq 2$ are of the form $v(\xi) = b e^{\lambda\xi}$, $b\in\R$, where $\lambda$ is the only real and positive root of
\begin{equation}\label{pol:left}
P(z)=\tau z^2+z^\alpha-h'(\phi_-)\,,
\end{equation}
(see Appendix~\ref{appendix:roots}, Lemma~\ref{roots}).
The statement can be proved as in \cite{CA}, where a genuinely non-linear flux function has been considered. The requirement in this proof is only to have $h'(\phi_-)>0$, which is guaranteed by (\ref{sign:h:prime}). Then we can construct solutions of the non-linear problem (\ref{TWP}) on $(-\infty,\tilde{\xi}]$ as small perturbations of the exponential solutions of (\ref{LTW}) as in \cite{ACH}. The next step is to extend for increasing $\xi$ these solutions by a continuation principle \cite[Lemma~3]{ACH} and show that there is uniqueness up to translation in $\xi$ \cite[Lemma~5]{ACH}.
  
Statement (ii) follows by the same arguments as in the proof of \cite[Lemma~4]{ACH}. Then the first statement of (iii) is a direct consequence of \cite[Lemma~6]{ACH} that guarantees that under the given assumption, the value of the limit must be a zero of $h$ different from $\phi_-$. We recall that for the quadratic case, $h$ has only two zeros, $\phi_-$ and $\phi_+$ (that would correspond to $\phi_c$ given the condition (\ref{sign:h:prime})), but in the current case $h$ has three zeros. The argument in the proof of \cite[Lemma~6]{ACH} is by contradiction, assuming that the constant value of the limit of $\phi$ is not a zero of $h$, hence we obtain the conclusion in Theorem~\ref{existence} allowing the third possibility, $\phi_+$.

As a final remark, let us mention a crucial difference for the cubic flux. For the quadratic flux one can show that solutions remain bounded and that this implies the existence of a limit value as $\xi \to \infty$ and therefore the only possible connection is to the constant value $\phi_+$, and this implies the existence (and uniqueness up to translation in $\xi$). However, in the cubic case we cannot show that solutions of (\ref{TWP}) subject to (\ref{far-fieldL}) remain bounded from below. The main difference in the arguments comes from the functional
\begin{equation}\label{energy}
H(\phi):=\int_{0}^\phi h(y) \,dy = -c\frac{\phi^2}{2} + \frac{\phi^4}{4}+A\phi
\,,\quad \mbox{with}\quad  A= c \phi_--\phi_-^3 \,.
\end{equation}
The difference $H(\phi)-H(\phi_-)$ being non-negative is a necessary condition for existence (see Lemma~\ref{good:sign} below). In the quadratic case (and, more generally, for a genuinely non-linear flux) the corresponding primitive has a zero $\bar\phi<\phi_+$ that gives a lower bound of the solutions, because in the interval $(\phi_+,\phi_-)$, $H(\phi)-H(\phi_-)> 0$. In the current case this is no longer true; $H(\phi)-H(\phi_-)> 0$ is satisfied in $(-\infty,\phi_-)$, then, the existence of a lower bound is not guaranteed by this argument.

In the next lemmas we summarise the conditions on $H$:
\begin{lem}\label{good:sign} Let $\phi$ be a solution of (\ref{TWP}) that satisfies (\ref{far-fieldL}) and let $(-\infty,\xi_{exist})$ be its interval of existence, where $\xi_{exist}\in \R\cup\{+\infty\}$. Then,
\[
\int_{-\infty}^\xi \phi^\prime(y) \DD^\alpha[\phi](y)\, dy\geq 0\, , \quad \forall \xi\in (-\infty, \xi_{exist}). 
\]
Moreover, the integral vanishes if and only if $\phi\equiv\phi_-$.

As a consequence, for all $\xi\in (-\infty, \xi_{exist})$, the following holds:
\begin{equation}\label{energy:form}
0 \leq \frac{\tau}{2}(\phi'(\xi))^2 + \int_{-\infty}^\xi \phi^\prime(y) \DD^\alpha[\phi](y) \,dy = H(\phi(\xi))-H(\phi_-)\,.
%\quad \mbox{for all}\quad \xi\in (-\infty, \xi_{exist}). 
\end{equation}

\end{lem}
The first statement appears in the proof of \cite[Lemma~4]{ACH} (see also \cite{CA} for a similar result) and adapts the arguments of \cite[Chapter~9]{LiebLoss}. The second part of the lemma follows by multiplying (\ref{TWP}) by $\phi'$ and integrating with respect to $\xi$. The condition (\ref{energy:form}) is on the primitive of the corresponding non-linear function of the equation, here (\ref{energy}).

%This Lemma is also used in the proof of Theorem~\ref{existence}, for instance, it is crutial when showing that solutions remain uniformly bounded. 

%MAYBE THESE SHOULD BE SAID LATER

Next, we show (\ref{nec:cond0}), which is a necessary condition on the far-field values for a solution of (\ref{TWP})-(\ref{far-fieldR}) to exist.
\begin{lem}\label{nec:cond:farfield}
Let $\phi_-$ and $\phi_+$ satisfy (\ref{lin:ass}). Then, the inequality
\begin{equation}\label{nec:cond}
H(\phi_+) -H(\phi_-) > 0
\end{equation}
is a necessary condition to obtain a global solution $\phi$ of (\ref{TWP}) that satisfies both (\ref{far-fieldL}) and (\ref{far-fieldR}). Moreover, (\ref{nec:cond}) is equivalent to:
\begin{equation}\label{nec:cond:equiv}
c < \phi_-^2 \quad \mbox{which implies} \quad \phi_-+\phi_+ > 0\,.
\end{equation}
\end{lem}

\begin{proof}
Suppose that $\phi$ is a global solution of (\ref{TWP}) satisfying (\ref{far-fieldL}) and (\ref{far-fieldR}). Then (\ref{energy:form}) holds, and taking the limit $\xi\to\infty$ yields 
\[
H(\phi_+) - H( \phi_-) = \int_{\R} \phi^\prime(\xi) \DD^\alpha[\phi](\xi) \,d\xi \geq 0.
\]  
We observe that $H(\phi) > H(\phi_-)$ for all $\phi\neq \phi_-$, and in particular (\ref{nec:cond}) holds by (\ref{lin:ass}). The assertion (\ref{nec:cond:equiv}) follows from (\ref{nec:cond}) by elementary computations.
\end{proof}

%THE RESULT OF NO OSCILLATIONS GOES FIRST, BECAUSE IN THE PROOF OF BLOW DOWN WE ASSUME THAT u and v go to -infty!!! 
Next we show that if a trajectory $\phi$ that satisfies (\ref{far-fieldL}) becomes unbounded, then it cannot oscillate below a certain value:
\begin{lem}[Non-oscillatory behaviour]\label{noosci}
 Let $\phi \in C^3_b(-\infty,0)$ be a solution as constructed in Theorem~\ref{existence} (i)-(ii). If the continuation of $\phi$ becomes unbounded, then there exists $\xi^* \in \R\cup \{+\infty\}$ such that $\lim_{\xi \to (\xi^*)^-}\phi(\xi)=-\infty$.
\end{lem}

\begin{proof} %HERE IN PRINCIPLE $xi^*$ MAY OR MAY NOT BE FINITE, IT DOES NOT AFFERCT THE PROOF.
  
  Since $\phi(\xi)<\phi_-$ for all $\xi\in \R$ and, by assumption, $\phi$ is unbounded, there must exists $\xi^* \in \R\cup \{+\infty\}$ such that
  $\liminf\limits_{\xi \to (\xi^*)^-}{ \phi(\xi)} = -\infty$. We have to prove that $\lim\limits_{\xi \to (\xi^*)^-}{ \phi(\xi)} = -\infty$.

  We argue by contradiction, and assume that $\displaystyle\lim_{\xi \to (\xi^*)^-}{\phi(\xi)}\neq -\infty$ and does not exist. Then, by regularity, $\phi$ becomes unbounded in an oscillatory fashion as $\xi$ increases. This means that there exists a sequence of local minima as follows: $\{\xi^n_{min}\}_{n\geq 0}$ is increasing and satisfies $\xi^n_{min} \to (\xi^*)^-$, $\phi'(\xi^n_{min})=0$, $\phi''(\xi^n_{min}) > 0$ and $\phi(\xi^n_{min}) < -\phi_-$ for all $n \geq 0$ and $\{\phi(\xi^n_{min})\}_{n\geq0}$ is a monotone decreasing sequence with $\lim_{n\to \infty}{\phi(\xi^n_{min})} = -\infty$.

Observe that then also $h(\phi(\xi_{min}^n))<0$ for all $n\geq 0$, and this gives 
\begin{equation}\label{TWES}
\DD^\alpha[\phi](\xi_{min}^n) = h(\phi(\xi_{min}^n)) - \tau \phi''(\xi_{min}^n) < 0 \quad \text{for all} \quad n\in \N.
\end{equation}

Lemma~\ref{bound:Da} gives a bound for the fractional derivative in terms of $\phi$ and its first derivative. Namely, there exists $C_\alpha >0$ (independent of $\tau$) such that
\[
\left|\DD^\alpha[\phi](\xi)\right| \leq C_\alpha \|\phi\|_{L^\infty(-\infty,\xi)}^{1-\alpha} \|\phi'\|_{L^\infty(-\infty,\xi)}^{\alpha} \ \ \text{for all} \ \xi\in (-\infty, \xi^*)
\]
and, in particular, for each $\xi=\xi_{min}^n$ we get the lower bound
\begin{equation}\label{Damin:bound:below}
0>\DD^\alpha[\phi](\xi_{min}^n) \geq -C_\alpha \|\phi\|_{L^\infty(-\infty,\xi_{min}^n)}^{1-\alpha} \|\phi'\|_{L^\infty(-\infty,\xi_{min}^n)}^{\alpha}.
\end{equation}
%Using $\phi''(\xi_{min}^n)>0$ and (\ref{TWES}) yields that
%$$\DD^\alpha[\phi](\xi_{min}^{n}) < h(\phi(\xi_{min}^{n})).$$
On the other hand, considering Lemma~\ref{h:H:bounds} and that $\phi''(\xi_{min}^n)>0$, we get the upper bound
 \begin{equation}\label{Damin:bound:above}
 \DD^\alpha[\phi](\xi_{min}^{n}) < h(\phi(\xi_{min}^{n})) < -C_h |\phi(\xi_{min}^{n})|^3 = -C_h  \|\phi\|_{L^\infty(-\infty,\xi_{min}^n)}^{3}.
 \end{equation}
Now, combining (\ref{Damin:bound:below}) and (\ref{Damin:bound:above}), we obtain
\[
 -C_\alpha \|\phi\|_{L^\infty(-\infty,\xi_{min}^n)}^{1-\alpha} \|\phi'\|_{L^\infty(-\infty,\xi_{min}^n)}^{\alpha} \leq - C_h \|\phi\|_{L^\infty(-\infty,\xi_{min}^n)}^{3}
 \]
 which is equivalent to
  \begin{equation}\label{norm1}
   % \frac{C_\alpha}{C_h} \|\phi'\|_{L^\infty(-\infty,\xi_{min}^n)}^{\alpha} \geq \|\phi\|_{L^\infty(-\infty,\xi_{min}^n)}^{\alpha+2}
    \|\phi\|_{L^\infty(-\infty,\xi_{min}^n)}^{\alpha+2} \leq \frac{C_\alpha}{C_h} \|\phi'\|_{L^\infty(-\infty,\xi_{min}^n)}^{\alpha}.
  \end{equation}

We obtain an upper bound on $\|\phi'\|_{L^\infty(-\infty,\xi_{min}^n)}$ using Lemma~\ref{good:sign} and Lemma~\ref{h:H:bounds}, as follows:
\[
 \frac{\tau}{2}\left(\phi'(\xi)\right)^2 \leq H(\phi(\xi)) - H(\phi_-) \leq C_H \phi^4(\xi) \leq 2 \left(\phi(\xi_{min}^n))\right)^4, \quad \forall \xi \in (-\infty,\xi_{min}^n)
\]
which implies
 \begin{equation}\label{norm2}
  \|\phi'\|_{L^\infty(-\infty,\xi_{min}^n)}^2 \leq  \frac{4}{\tau} \|\phi\|_{L^\infty(-\infty,\xi_{min}^n)}^4.
  \end{equation}
  Finally, combining (\ref{norm1}) and (\ref{norm2}) implies for all $n\in\N$ that
  \[
  \|\phi\|_{L^\infty(-\infty,\xi_{min}^n)}^{2-\alpha} %= (-\phi(\xi_{min}^n))^{2-\alpha}
  < \tau^{-\alpha/2}C,
  \]
  with $C=2^\alpha\frac{C_\alpha}{C_h}$. But this contradicts that $\lim_{n\to\infty} \phi(\xi_{min}^n) = -\infty$ and such sequences of local minima cannot exists. Thus it must be that $\lim_{\xi\to (\xi^*)^-} \phi(\xi)=-\infty$.
  
\end{proof}

Next we show that $\xi^*$ of the previous lemma is a finite value.
\begin{lem}\label{blow-down}
Let $\phi \in C^3_b(-\infty,0)$ be a solution as constructed in Theorem~\ref{existence} (i)-(ii). If the continuation of $\phi$ becomes unbounded, then there exists a finite $\xi^* \in \R$ such that% $\lim_{\xi \to (\xi^*)^-}\phi(\xi)=-\infty$.   Let $\phi$ be an unbounded solution of (\ref{TWP}). Then there exists $\xi^* \in \R$ such that
\begin{equation}\label{phi:blow-down}
\lim_{\xi\to(\xi^*)^-}\phi(\xi) = -\infty,
\end{equation}
and, therefore, $\phi$ cannot be extended to $\R$. Moreover, the asymptotic behaviour of $\phi$ is given by,
\begin{equation}\label{blow-down:bhv}
\lim_{\xi \to (\xi^*)^-} {\left| \phi(\xi) \right| (\xi^* - \xi)}  = \sqrt{\tau} C
\end{equation}
where $C>0$ is a constant independent of $\tau$. 
\end{lem}

\begin{proof}
Since $\phi$ is unbounded, by Lemma~\ref{noosci} there exist $\xi^* \in \R\cup\{+\infty\}$ and $\xi_1\in \R$ such that 
\begin{equation}\label{noosci:step}
\displaystyle \lim_{\xi \to (\xi^*)^-} \phi(\xi)  = -\infty \quad \text{and} \quad \phi'(\xi)<0, \ \forall \xi \in (\xi_1,\xi^*). 
\end{equation}

It is now convenient to rewrite the equation (\ref{TWP}), as a first order system by making the change of variables
$u(\xi)=\phi(\xi)$ and $v(\xi)=\phi'(\xi)$ 
%(\ref{far-fieldL}) as a first order system making the change $u=\phi$ and $v=\phi'$,
for $\xi \in (-\infty, \xi^*)$. This gives:
\begin{equation}\label{TWEsystem}
%\begin{cases}
%u'(\xi) =& v(\xi), \\ 
%\displaystyle{
%v'(\xi) =& \frac{1}{\tau}\left( h(u(\xi)) - \DD^\alpha[u](\xi)\right).
%}
%\end{cases}
\begin{cases}
\displaystyle {u' = v,} \\ 
\displaystyle{
v' = \frac{1}{\tau}\left( h(u) - \DD^\alpha[u]\right).
}
\end{cases}
\end{equation}
We first notice that there exist some $\xi_0 \in  [\xi_1,\xi^*)$ and $C_v>0$ such that
\begin{equation}\label{u:bounds}
-\infty < u(\xi)< -\phi_-\quad \forall \xi \in [\xi_0, \xi^*),\quad  u(\xi_0)\leq u(\xi)<\phi_- \quad \forall \xi \leq \xi_0
\end{equation}
and
\begin{equation}\label{v:bounds}
  v(\xi)< -C_v, \quad \forall \xi \in (\xi_0, \xi^*).
\end{equation}
The bounds (\ref{u:bounds}) hold by (\ref{noosci:step}). Let us show (\ref{v:bounds}): If $v$ becomes unbounded there is nothing to prove, again by (\ref{noosci:step}). On the other hand, if $v$ is bounded, we get an upper bound of $v'(\xi)$, which diverges to $-\infty$ as $\xi \to (\xi^*)^-$. Indeed, applying Lemma~\ref{bound:Da}, Lemma~\ref{h:H:bounds} and an estimate like (\ref{norm2}) yield
\[
\begin{split}
v'(\xi) = \frac{1}{\tau} h(u(\xi)) - \frac{1}{\tau} \DD^\alpha[u](\xi) &\leq \frac{C_h}{\tau} u(\xi)^3 + \frac{2^\alpha C_\alpha}{\tau^{1+\alpha/2}} |u(\xi)|^{1+\alpha} \\
&= u(\xi)^3 \left(\frac{C_h}{\tau}  - \frac{2^\alpha C_\alpha}{\tau^{1+\alpha/2}} \frac{1}{|u(\xi)|^{2-\alpha}}\right) .
\end{split}
\]
The right-hand side of this inequality tends to $-\infty$ as $\xi\to(\xi^*)^-$, therefore, $\lim_{\xi\to(\xi^*)^-}v'(\xi)=-\infty$. This implies $(\ref{v:bounds})$. We can adjust the value of $\xi_0$ by taking it closer to $\xi^*$ as necessary so that both bounds hold in the same interval.

Once (\ref{u:bounds}) and (\ref{v:bounds}) are established for $\xi \in (\xi_0,\xi^*)$, we introduce the variables
\[
z:=\frac{1}{u}<0 \quad \mbox{and}\quad  w:=-\frac{v}{u^2}\geq0
\]
in such interval.

We also change the independent variable, in order to absorb $z$ in the derivative, as follows:
\begin{equation}\label{def:s}
s=s_0 - \int_{\xi_0}^{\xi}\frac{dy}{z(y)}, \quad s_0>0.
\end{equation}
Notice that $z<0$, thus $s$ is strictly increasing with respect to $\xi$.

The system (\ref{TWEsystem}) then reads:
\begin{equation}\label{system:z:w}
\begin{cases}
\displaystyle\frac{dz}{ds} = -zw,  \\ 
\displaystyle\frac{dw}{ds} = -2w^2
                        + \frac{1}{\tau}\left( h\left(\frac{1}{z}\right) - \DD^\alpha\left[\frac{1}{z} \right]\right)z^3.
\end{cases}
\end{equation}

{\bf CASE I:} We first analyse the possibilities of `extinction' and of `blow-up' for $w$ at a finite $s$.
Let us assume the former: there exists a finite $\bar{s}>s_0$ such that $\lim_{s\to \bar{s}} w(s) = 0$.
Then, at $ \bar{\xi}\leq \xi^*$, given by $\bar{s}= s_0 - \int_{\xi_0}^{\bar{\xi}} u(y)dy$, either $v(\bar{\xi})=0$ or $\lim_{\xi\to(\bar{\xi})^-} u(\xi)=-\infty$. The former contradicts (\ref{v:bounds}). The latter case implies that either $\bar{s}$ is infinite, which gives a contradiction, or that $u$ is integrable in $(\xi_0,\bar{\xi})$. In that case, we have that $\bar{\xi}=\xi^*<\infty$, but also that  for $\xi\in(\xi_0,\xi^*)$ there are very small constants $\e_1>0$ and $\e_2>0$, such that 
\begin{equation}\label{lim:w:0}
\e_1< -\frac{u'(\xi)}{u(\xi)^2}< \e_2
\end{equation}
and integrating gives
\[
 \frac{1}{\e_2 (\xi-\xi_0)+1/u(\xi_0)}< u(\xi) <\frac{1}{\e_1 (\xi-\xi_0)+1/u(\xi_0)}
\]
so that $\xi^*\leq \xi_0 - \frac{\e_1}{u(\xi_0)}$. The constant $\e_1$ tends to $0$ as $\xi\to(\xi^*)^-$ (because we have only used the definition of $w$ and the limit $w\to0$ in the estimate (\ref{lim:w:0})), implying that $\xi^*=\xi_0$, but this gives a contradiction, with the choice of $\xi_0$. This implies that $w$ cannot tend to $0$ at a finite value. 

%\and this finishes the proof.

Let us assume now that $w$ exhibits `blow-up'. Then there exists a finite $s^*>s_0$ such that
$\lim_{s\to s^*} w(s) =\infty$.  Thus, for $\xi\in(\xi_0,\xi^*)$ sufficiently close to $\xi^*$,
\[
-\frac{u'(\xi)}{u^2(\xi)} \gg 1 
\]
and this, integrating over $(\xi_0,\xi)$, gives
\[
u(\xi)< \frac{1}{\xi - \xi_0 + 1/u(\xi_0)}, 
\]
but, this implies that $s^*$ cannot be finite (using the above bound in the integral of (\ref{def:s}) gives $\infty$ as lower bound) and $w$ does not blow up.

{\bf CASE II:} Let us now assume that $w$ is defined for all $s\in\R$. We shall show that $\exists\lim_{s\to\infty}w(s)<\infty$, and then apply the definition of the new variables and integrate to get the result. First, we obtain estimates from (\ref{system:z:w}).

Integrating the first equation in (\ref{system:z:w}), we get the following:

\begin{equation}\label{zto0}
 z(s)= z_0 e^{-\int_{s_0}^s w(s) \, ds} \rightarrow 0^-\quad  \mbox{as}\quad s\to\infty.
\end{equation} This limit is clear if $\int_{s_0}^\infty w(s) \, ds=+\infty$. If $\int_{s_0}^\infty w(s) \, ds<+\infty$, in particular $\lim_{s\to\infty} w(s)=0$, then an argument as for the extinction of $w$ gives that either $\lim_{s\to\infty} v(s)=0$, which is in contradiction with (\ref{v:bounds}), or that $\lim_{s\to\infty} u(s)=-\infty$, which also implies that $z(s)\to 0$ as $s\to\infty$.

Now, we observe that $\DD^\alpha[u]<0$ for all $\xi \in (\xi_0,\xi^*)$. To prove this we first split $\DD^\alpha[u](\xi)$ as follows,
\[
\frac{1}{d_\alpha} \DD^\alpha[u](\xi) =
\int_{-\infty}^{\xi_0} \frac{v(y)}{(\xi-y)^\alpha}\, dy  + \int_{\xi_0}^\xi\frac{v(y)}{(\xi-y)^\alpha} \,dy.
\]
The second integral is negative due to (\ref{v:bounds}). And the first is also negative, since integrating by parts and using (\ref{u:bounds}) we get:
\[
\begin{split}
\int_{-\infty}^{\xi_0}\frac{v(y)}{(\xi-y)^\alpha}\, dy =& -\alpha \int_{-\infty}^{\xi_0}\frac{u(y)}{(\xi-y)^{\alpha+1}}\, dy + \frac{u(\xi_0)}{(\xi-\xi_0)^\alpha} \\
<& \ \alpha |u(\xi_0)|  \int_{-\infty}^{\xi_0}\frac{dy}{(\xi-y)^{\alpha+1}} + \frac{u(\xi_0)}{(\xi-\xi_0)^\alpha} =0.
\end{split}
\]

Let us now get a lower bound for $\DD^\alpha[1/z] = \DD^\alpha[u]$ by rewriting Lemma~\ref{bound:Da} in terms of $z$ and $w$. We have two cases for all $s>s_1$ for some $s_1\geq s_0$ large enough:
\begin{equation}\label{fracBoundzw}
  0<
  -\DD^\alpha\left[\frac{1}{z}\right] \leq D(z,w):=\begin{cases} C'_\alpha \left|\frac{1}{z}\right|^{1-\alpha}, \quad&\mbox{if}\ |v| \ \text{stays bounded},  \\ C_\alpha \left|\frac{1}{z}\right|^{1+\alpha} (\sup_{s>s_1} w)^\alpha,\quad&\mbox{if} \ |v| \ \text{becomes unbounded},
  \end{cases}
\end{equation}
for some $C'_\alpha$, $C_\alpha >0$. 

We now get bounds for the non-linear term using (\ref{zto0}):
Let $s_2\geq s_0$ large enough such that for all $s>s_2$
\[
|z(s)|<\min\left\{\frac{c}{-(c\phi_- - \phi_-^3)}, \frac{1}{\sqrt{c}}\right\},
\]
then
\begin{equation}\label{h:z:small:bounds}
0<\frac{1}{\tau}-\frac{c}{\tau} z^2 \leq \frac{1}{\tau}z^3 h\left(\frac{1}{z}\right) = \frac{1}{\tau} - \frac{1}{\tau}(c - (c\phi_- - \phi_-^3)z)z^2 \leq \frac{1}{\tau} \ \mbox{for all}\ s>s_2.
\end{equation}

With (\ref{fracBoundzw}) and (\ref{h:z:small:bounds}) we have the following bounds on $dw/ds$ using the second equation in (\ref{system:z:w}):
\begin{equation}\label{wBounds}
\frac{1}{\tau} - \frac{c}{\tau} z^2 -2w^2 - \frac{1}{\tau}(-z)^3 D(z,w) \leq  \frac{dw}{ds} \leq \frac{1}{\tau} - 2w^2,  \quad \forall s>\max\{s_1,s_2\}.
\end{equation}

In order to prove that $w\to C>0$ as $s \to \infty$, we now argue by contradiction. First, we assume that $w$ becomes unbounded, then (\ref{wBounds}) implies that $w$ is decreasing as long as $w > 1/\sqrt{2\tau}$, but this contradicts that $0<w(s)$ becomes unbounded. Now, we assume that the limit is finite with $C=0$, since also $\lim_{s\to \infty} z(s)=0$, these implies that $dw(s)/ds>0$ for all $s$ large enough. This contradicts that $C=0$ and $w(s)>0$.

The last possibility that we have to exclude is that $w$ oscillates without limit. Let $M=\sup_{s>\max\{s_1,s_2\}} w(s)<\infty$, then using (\ref{wBounds}) we get there exists $C>0$ and $s_3\geq \max\{s_1,s_2\}$ such that
\begin{equation}\label{wgoodBounds}
\frac{1}{\tau} -2w(s)^2 - C |z(s)|^{2-\alpha} \leq \frac{dw(s)}{ds} \leq \frac{1}{\tau} -2w(s)^2 \quad \mbox{for all}\quad s>s_3.
\end{equation}
Here we have used (\ref{fracBoundzw}) and noticed that for $|z|$ small enough we have $|z|^{2+\alpha}$, $|z|^2<|z|^{2-\alpha}$.

%Here we distinguish two cases:
Observe that if $M<1/\sqrt{2\tau}$, no oscillations are possible in the limit,
because there exists $s_4\geq s_3$  such that
\[
0<\frac{1}{\tau}  -2M^2 - C |z(s)|^{2-\alpha} \leq \frac{dw(s)}{ds} \quad \text{for all} \quad s>s_4.
\]
Then, in this case, we obtain the desired result: $\lim_{s\to\infty} w(s)= M>0$ and is finite.
 
Now, if $M=1/\sqrt{2\tau}$, there exists $s \in \R$ such that $w(s)<1/\sqrt{2\tau}$, because we are assuming that the limit of $w$ does not exist. If $M>1/\sqrt{2\tau}$, then by (\ref{wgoodBounds}), with $s>s_3$,
\[
\frac{dw(s)}{ds} \leq \frac{1}{\tau} - 2w(s)^2 \leq 0, \quad \text{as long as} \quad \frac{1}{\sqrt{2\tau}}\leq w(s) \leq M.
\] 
This means that on the intervals of $s$ for which $\frac{1}{\sqrt{2\tau}}\leq w(s) \leq M$, $w$ is not increasing, so $w$ cannot oscillate in this range. This implies that $w(s) \in (0,1/\sqrt{2\tau}]$, for all $s>s_4$ with $s_4>s_3$ large enough, oscillating without limit.%, then it does so in this interval. 

Then, there exists also an increasing sequence $\{\overline{s}_n\}_{n\geq 0}>s_4$, where local minima of $w$ are attained, with $dw(\overline{s}_n)/ds = 0$, $0<w(\overline{s}_n)\leq 1/\sqrt{2\tau}$ and $\overline{s}_n \to\infty$ as $n\to \infty$. Let the sequence $\{\delta_n\}_{n\geq 0}$ be defined by evaluating the right-hand side of (\ref{wgoodBounds}) at each $s=\overline{s}_n$,
\begin{equation*}%\label{def:delta:n}
0<\delta_n:= \frac{1}{\tau}-2w^2(\overline{s}_n)<\frac{1}{\tau}.
\end{equation*}
Again there are two possibilities: either $\{\delta_n\}_{n\geq 0}$ is bounded from below by a positive constant $K$, or $\delta_n \to 0$ as $n\to \infty$. In the former case, we get from (\ref{wgoodBounds}) that for all $n$,
\[
 K - C|z(\overline{s}_n)|^{2-\alpha} \leq \frac{dw}{ds}(\overline{s}_n)=0.
\]
Applying (\ref{zto0}) as $n\to \infty$, we deduce that there exists $n_0\geq 0$ such that the left hand side is strictly positive
%f$|z(\overline{s}_n)|<\left(\frac{K}{C}\right)^{\frac{1}{2-\alpha}}$
for all $n\geq n_0$, a contradiction.

If $\delta_n\to 0$ as $n\to \infty$, then $w(\overline{s}_n) \to 1/\sqrt{2\tau}$ as $n\to \infty$. This means that this sequence of local minima is converging to the supremum of $w$ on $s>s_4$, but this contradicts that $w$ oscillates without a limit and, moreover, this implies that $\lim_{s\to\infty} w(s)=1/\sqrt{2\tau}$.

With the limit of $w$ and taking into account its  regularity, we get that there exists positive constants $C_1$ and $C_2$ proportional to $1/\sqrt{\tau}$, such that for all $\xi>\xi_0$
\[
 C_1<- \frac{u'(\xi)}{u^2(\xi)} < C_2.
\]
%where $C_1=\min\left\{\displaystyle\min_{s>s_0} \{w(s)\},\frac{1}{\sqrt{2\tau}}\right\}>0$ and $C_2=M>0$.
Then, integrating over the interval $(\xi_0,\xi)$ for $\xi<\xi^*$, gives
%\[
%C_1 (\xi - \xi_0)  < \frac{1}{u(\xi)} - \frac{1}{u(\xi_0)} < C_2 (\xi - \xi_0),
%\]
%which is equivalent to
\[
\frac{1}{C_2 (\xi - \xi_0) + 1/u(\xi_0)} < u(\xi)< \frac{1}{C_1 (\xi - \xi_0) + 1/u(\xi_0)}.
\]
These bounds imply (\ref{blow-down:bhv}), since $C_1$ and $C_2$ are proportional to $1/\sqrt{\tau}$, and with $\xi^* \leq \xi_0 - \frac{\sqrt{\tau}}{C u(\xi_0)}$, for some positive constant $C$.
%\]
\end{proof} 

% ------------------------------------------------
Gathering the results of the previous lemmas we finish this section with:
\begin{proof}[Proof of (iii) of Theorem~\ref{existence}]
As explained earlier the proofs of (i), (ii) and the first part of (iii) follow from the results of \cite{ACH}. The proof of the last part of (iii) follows applying the previous lemmas. First, Lemma~\ref{noosci} rules out the oscillatory behaviour of $\phi$ and then Lemma~\ref{blow-down} ensures that (\ref{phi:blow-down}) is satisfied.
\end{proof}
% ------------------------------------------------

\section{Existence of undercompressive waves: Proof of Theorem~\ref{TW:main:theorem}}\label{sec:main}
In the forthcoming, for every value of $\tau>0$ we will let $\phi_\tau(\xi)$ denote a solution of equation (\ref{TWP}) satisfying (\ref{far-fieldL}) as constructed in Theorem~\ref{existence}. According to the three possible behaviours of such trajectories, established in Theorem~\ref{existence} (iii), we define the following sets of $\tau$'s: 
\begin{df} For every $\tau> 0$ let $\phi_\tau$ be a solution as constructed in Theorem~\ref{existence}. Then we define the sets
  \begin{eqnarray*}
    \Sigma_u & := & \left\{ \tau> 0: \quad \lim_{\xi\to(\xi^*)^-} \phi_\tau(\xi)=-\infty \ \ \mbox{for some} \ \ \xi^*\in \R \right\}\,,
    \\
    \Sigma_c & := & \left\{ \tau> 0: \quad \lim_{\xi\to\infty}\phi_\tau(\xi)=\phi_c \right\}\,,
    \\
    \Sigma_+ & := & \left\{ \tau> 0: \quad \lim_{\xi\to\infty}\phi_\tau(\xi)=\phi_+ \right\}\,.
    \end{eqnarray*}
  \end{df}
By definition and uniqueness up to translation, these sets are disjoint. If the set $\Sigma_+$ is non-empty, then there exists a solution $\phi_\tau$ of the problem (\ref{TWP})-(\ref{far-fieldL}), such that $\lim_{\xi\to\infty}\phi_{\tau}(\xi)= \phi_+$, thus showing that will finish the proof of Theorem~\ref{TW:main:theorem}.

We prove this by a shooting argument, where $\tau$ is the shooting parameter. We divide this section into two. In the first part we show that $\Sigma_u$ is non-empty and open and in the second we show that $\Sigma_c$ is non-empty and then also that $\Sigma_+$ is non-empty. In this final part of the proof, we argue by contradiction; we assume that $\Sigma_+$ is empty, this means that $\Sigma_c$ is closed, and using continuity with respect to $\tau>0$, where we invoke Appendix~\ref{appendix:B}, we get to a contradiction.

We remark that we use Theorem~\ref{CDTh} of Appendix~\ref{appendix:B} for our problem, that is rewritten as (\ref{gen:delay:eq}) with (\ref{our:F1})-(\ref{our:F3}). This is valid on finite intervals, but, as we shall see, we can get continuity with respect to $\tau>0$ on intervals $(-\infty,\xi)$ using the results of \cite{ACH}.
 
We notice that as a main step to show that $\Sigma_c$ is non-empty we need to consider a modified problem where the non-linearity has only the zeros $\phi_-$ and $\phi_c$ and coincides with $h$ in that range. Then we can apply the monotonicity results of \cite{ACH} and a new result, Theorem~\ref{mono:tau:small:tail}, that guarantees that trajectories stay in a range where both non-linearities coincide.

%the results from \cite{ACH} applied to a modified problem where the non-linearity has only the zeros $\phi_-$ and $\phi_c$.

%---------------------------------
\subsection{The set $\Sigma_u$}\label{sigma:u}
We first show that $\Sigma_u$ is non-empty:

\begin{lem}\label{unbounded}
Consider $\phi_-$ and $\phi_+$ satisfying \eqref{lin:ass} and (\ref{nec:cond0}). Let $\phi_\tau$ denote the unique (up to shifts in $\xi$) solution of (\ref{TWP}) satisfying (\ref{far-fieldL}) as constructed in Theorem~\ref{existence}.
Then, there exists $\tau_m>0$ such that for all $\tau>\tau_m$ there exists $\xi^*_\tau \in\R$ such that $\lim_{\xi\to(\xi^*_\tau)^-} \phi_\tau(\xi) =-\infty$.
\end{lem}

\begin{proof}
Let us argue by contradiction. Assume that for all $\tau_0 > 0$ there exists at least one $\tau>\tau_0$ such that $\phi_\tau(\xi)$ is defined for all $\xi\in\R$. 
By Theorem~\ref{existence}, we know that $\phi_\tau$ is smooth, $\phi_\tau(\xi) <\phi_-$ for all $\xi\in\R$, and $\lim_{\xi\to\infty}\phi_\tau(\xi)=\phi^\ast\in\{\phi_c,\phi_+\}$. Moreover,
\begin{equation}\label{lin:ass:2}
  -\phi_- <\phi_+ <\phi_c <0
\end{equation}
due to (\ref{nec:cond0}), see also Lemma~\ref{nec:cond:farfield}.

  First, we prove that $\|\phi_\tau\|_\infty =\phi_-$ and deduce estimates on $\|\phi'_\tau\|_\infty$ and $\|\DD^\alpha[\phi_\tau]\|_\infty$.

At this point, one has just a lower bound for $\|\phi_\tau\|_\infty\geq\phi_-$.
Then, we distinguish two cases:
Either $\inf_{\xi\in\R} \phi_\tau(\xi)=\phi^*$ or there exists a value $\xi_{min}\in \R$ such that $\phi_\tau(\xi_{min})=\min_{\xi\in\mathbb R} \phi_\tau(\xi) =:\phi_{min}$. 

First, if $\inf_{\xi\in\R} \phi_\tau(\xi)=\phi^*$ then $\|\phi_\tau\|_{\infty}\leq \max\{|\phi^*|,\phi_-\}=\phi_-$ due to (\ref{lin:ass:2}), and this implies $\|\phi_\tau\|_\infty =\phi_-$ in this case.

In the other case, let $\xi_{min}\in\R$ be the point at which the minimum of $\phi_\tau$ is attained. It is enough to prove that $\phi_\tau(\xi_{min})=:\phi_{min}\in (-\phi_-,\phi_c)$, since this implies that $\|\phi_\tau\|_\infty\leq\phi_-$, due to (\ref{lin:ass:2}).

We argue by contradiction and assume to the contrary that $\phi_{min}<-\phi_-$. 
%(recall that $\phi_+>-\phi_-$ by (\ref{nec:cond:equiv}) of Proposition~\ref{nec:cond:farfield}).
Now, using that $\phi_\tau''(\xi_{min}) \geq 0$ in (\ref{TWP}) yields  
\begin{equation}\label{TWE:xi:min}
0> h(\phi_{min})=\tau \phi_{\tau}''(\xi_{min})+\DD^\alpha [\phi_\tau](\xi_{min}) \geq \DD^\alpha [\phi_\tau] (\xi_{min}) ,
\end{equation}
and, by Lemma~\ref{bound:Da}, there exists $C_\alpha>0$ (independent of $\tau$) such that
\begin{equation}\label{bdd:Da:phi:phiprime}
|\DD^\alpha [\phi_\tau](\xi)|\leq C_\alpha \| \phi_\tau\|_\infty^{1-\alpha}\,\|\phi_\tau'\|_\infty^{\alpha} \qquad \text{for all } \xi\in\R.
\end{equation}
Then, combining (\ref{TWE:xi:min}) and (\ref{bdd:Da:phi:phiprime}), we conclude that
\begin{equation}\label{bddDalpha}
0 > \DD^\alpha[\phi_\tau](\xi_{min}) \geq  -  C_\alpha \| \phi_\tau\|_\infty^{1-\alpha}\,\|\phi'_\tau\|_\infty^{\alpha}. 
\end{equation}
Now, Lemma~\ref{h:H:bounds} implies that there exists a constant $C_h>0$, depending only on $\phi_-$ and $\phi_+$, such that 
\begin{equation}\label{bddh}
  \DD^\alpha[\phi_\tau](\xi_{min})\leq h(\phi_{min}) < C_h \phi_{min}^3 = -C_h\|\phi_\tau\|_\infty^3 <0.
\end{equation}
Combining (\ref{bddDalpha}) and (\ref{bddh}), then gives 
\begin{equation}\label{phi:bdd:by:phiprime}
\|\phi_\tau\|_\infty^{2+\alpha } \leq \frac{C_\alpha}{C_h}\,\|\phi'_\tau\|_\infty^{\alpha}
\end{equation}
where the constants $C_\alpha$ and $C_h$ depend on $\alpha$, $\phi_-$ and $\phi_+$ but are independent of $\tau$.

On the other hand, Lemmas~\ref{good:sign} and~\ref{h:H:bounds} imply that
\[
\frac{\tau}{2} (\phi_\tau'(\xi))^2
 \leq H(\phi_\tau(\xi))-H(\phi_-)
 \leq 2 \|\phi_\tau\|_\infty^4
 \qquad \text{for all } \xi\in\R,
\]
and taking the supremum with respect to $\xi\in\R$ yields
\begin{equation}\label{phiprime:bdd:by:phi}
\frac{\tau}{2} \|\phi'_\tau\|_\infty^2 \leq  2\|\phi_\tau\|_\infty^4.
\end{equation}

Finally, with (\ref{phi:bdd:by:phiprime}) and (\ref{phiprime:bdd:by:phi}), we obtain
\begin{equation}\label{phi_tau:upper:lower:bounds}
\phi_-^{2-\alpha} < \|\phi_\tau\|_\infty^{2-\alpha } < \tau^{-\alpha/2} C,
\end{equation}
with $C =2^\alpha {C_\alpha}/{C_h}>0$, which is independent of $\tau$.
Our assumption $\phi(\xi_{min}) < -\phi_-$ implies that the inequalities in (\ref{phi_tau:upper:lower:bounds}) are strict, then, necessarily $\tau < C^{2/\alpha} \phi_-^{2-4/\alpha}$, if this holds. That means that for $\tau > \tau_\alpha := C^{2/\alpha} \phi_-^{2-4/\alpha}$ the bounded solution~$\phi_\tau$ satisfies $\|\phi_\tau\|_\infty = \phi_-$.

Observe that the estimates (\ref{bdd:Da:phi:phiprime}) and (\ref{phiprime:bdd:by:phi}) are valid in both cases considered above, then substituting $\|\phi_\tau\|_\infty = \phi_-$ in them we obtain:
\begin{equation}\label{bounds:phiprime:Da}
\|\phi_\tau'\|_\infty\leq \tau^{-\frac{1}{2}} 2\phi_-^2, \qquad\text{and}\quad \|\DD^\alpha[\phi_\tau]\|_\infty \leq \tau^{-\frac{\alpha}{2}} C_\alpha 2^{\alpha} \phi_-^{\alpha+1}.
\end{equation}
We shall use these estimates below.

In order to finish the proof, we have to get a contradiction with the assumption that there are such bounded solutions if $\tau$ is large enough. For the argument, we rescale the variables as follows
\[
\xi= \sqrt{\tau} X\quad \mbox{and} \quad \phi_\tau(\xi) = \Psi_\tau(X)
\]
such that (\ref{TWP}) reads
\begin{equation}\label{resc:TWe}
\frac{d^2}{dX^2}\Psi_\tau + \tau^{-\frac{\alpha}{2}} \DD_X^\alpha[\Psi_\tau] = h(\Psi_\tau). 
\end{equation}
Then the estimates~\eqref{bounds:phiprime:Da} induce the uniform bounds:
\begin{equation}\label{UB:Phi}
 \exists C>0 \text{ (independent of $\tau$)}: \quad \|\DD_X^\alpha[\Psi_\tau]\|_\infty <C, \quad \|\Psi'_\tau\|_\infty <C.
\end{equation}
Due to Theorem~\ref{existence}(i) and its proof, for sufficiently small $\e>0$ there exists $X_0\in \R$ such that $\Psi_\tau(X_0) =\phi_- -\e$ and $\Psi'_\tau(X_0)<0$ with $\phi_- -\e>\phi_m$ where $\phi_m\in(\phi_c,\phi_-)$ such that $h'(\phi_m)=0$. Let $X_1\in\R$ such that $\Psi_\tau(X_1)\in(\phi_c,\phi_m)$ and $h(\Psi_\tau(X_1)) =h(\Psi_\tau(X_0))$. Choosing $\e$ even smaller, we can ensure that $\phi_c<\Psi(X_1)<0<\phi_m$. Then, integrating (\ref{resc:TWe}) over the interval $(X_0,X)$ and using~\eqref{UB:Phi} yields
\[
\Psi'_\tau(X)
 \leq \Psi'_\tau(X_0) + \tau^{-\alpha/2} C(X-X_0) + \int_{X_0}^{X}{h(\Psi_\tau(Y)) \,dY}.
\]
For all $X\in(X_0,X_1)$, we deduce $h(\Psi_\tau(X)) \leq h(\Psi_\tau(X_0)) =h(\phi_- -\e) <0$ and 
\[
\Psi'_\tau(X)
 \leq \Psi'_\tau(X_0) + \big\{\tau^{-\alpha/2} C +h(\phi_- -\e)\big\} (X-X_0).
\]
Choosing $\tau_0>0$ sufficiently large, such that the associated $\tau>\tau_0$ satisfies $\tau^{-\alpha/2} < \frac{|h(\phi_- -\e)|}{2C}$, implies that
$\Psi'_\tau(X)<0$ for all $X\in(X_0,X_1)$. And therefore, also,  $\Psi_\tau$ decreases monotonically for all $X\in(-\infty,X_1)$ with $\phi_c<\Psi_\tau(X_1)<0$.

If $\Psi_\tau$ is not monotone for all $X\in\R$ then it attains its first local minimum at some $X_{min}>X_1$.

We now evaluate the energy estimate (\ref{energy:form}), rescaled as for (\ref{resc:TWe}), at $X_{min}$ and using the bound (\ref{UB:Phi}) for $\DD^\alpha_X[\cdot]$ yields:
\[
\begin{split}
0\leq H(\Psi_\tau(X_{min})) - H(\phi_-)
 &= \tau^{-\alpha/2} \int_{-\infty}^{X_{min}}\Psi'_\tau(Y)\DD^\alpha_Y[\Psi_\tau] \,dY \\
 &< \tau^{-\alpha/2} C \int_{-\infty}^{X_{min}}|\Psi'_\tau(Y)| \,dY.
\end{split}
\]
Using that $\Psi_\tau$ is decreasing in $(-\infty,X_{min})$ and that $\|\Psi_\tau\|_\infty =\phi_-$, implies
\begin{equation}\label{energy:tau:large}
\begin{split}
0\leq H(\Psi_\tau(X_{min})) - H(\phi_-)
 &< \tau^{-\alpha/2} C \int_{-\infty}^{X_{min}}|\Psi'_\tau(Y)| \,dY \\
 &= \tau^{-\alpha/2} C(\phi_- - \Psi_\tau(X_{min})) < 2 \tau^{-\alpha/2} C \phi_-.
\end{split}
\end{equation}
Observe that, $H(\Psi_\tau(X_{min})) - H(\phi_-) \geq H(\phi_+) - H(\phi_-)  >0$, since $\Psi_\tau(X_{min}) \leq \Psi_\tau(X_1) <0$ and by Lemma~\ref{nec:cond:farfield} (indeed, $H(\phi)-H(\phi_-)$ has two local minima, one at $\phi_-$ which is zero, and the other at $\phi_+$, which is strictly positive; at $\phi_c$ it attains a local maximum). On the other hand, the upper bound in (\ref{energy:tau:large}) can be made arbitrarily small by choosing $\tau_0$ sufficiently large. This gives a contradiction, thus $\Psi_\tau$ does not attain a minimum and decreases for all $X\in\R$. 

We have thus concluded that the bounded solution $\Psi_\tau$ converges either to $\phi_+$ or $\phi_c$ in a monotonically decreasing way. We can use the previous argument again and take the limit $X_{min} \to \infty$ in the energy estimate, this gives
\[
0
< H(\phi^\ast)- H(\phi_-) 
< \tau^{-\alpha/2} C (\phi_- - \phi^\ast)
< 2 \tau^{-\alpha/2} C \phi_-.
\] 
However, $0 <H(\phi_+) -H(\phi_-)\leq H(\phi_c) -H(\phi_-)$ is a fixed positive number whereas the upper bound can be made arbitrarily small by choosing $\tau_0$ sufficiently large. This gives again the contradiction, and so there cannot exist such bounded solutions if $\tau$ is large enough.

\end{proof}
%HERE 

\begin{lem}\label{sigma_u_open}
$\Sigma_u$ is an open set.
\end{lem}

\begin{proof}
By Lemma~\ref{unbounded}, there exists a value $\tau_m>0$ such that $(\tau_m,+\infty) \subset \Sigma_u$, thus such points are inner points of $\Sigma_u$. Then, it remains to prove that if the intersection $(0,\tau_m]\cap\Sigma_u$ is non-empty then is an open set.

Suppose $\tau_0 \in (0,\tau_m]\cap \Sigma_u$, in particular $\lim_{\xi \to (\xi^*_{\tau_0})^-}{\phi_{\tau_0}}(\xi) =-\infty$ for some $\xi^*_{\tau_0}$. We next prove that there exists $\delta>0$ such that for all $\tau \in (\tau_0-\delta, \tau_0 +\delta)$, the solution $\phi_\tau$ of \eqref{TWP} and \eqref{far-fieldL} satisfies $\lim_{\xi \to (\xi^*_\tau)^-}{\phi_\tau}=-\infty$ for some $\xi^*_\tau\in\R$.

  We use continuous dependence on the parameter $\tau$ on finite intervals (see Appendix~\ref{appendix:B}). Given a bounded interval $I$ such that $\phi_{\tau_0}(\xi)< -\phi_-$ for all $\xi\in I$, then for all $\e>0$ there exists $\delta>0$ such that
\begin{equation}\label{continuous_dependence_tau}
\left| \phi_{\tau_0}(\xi) - \phi_\tau(\xi) \right| <\e, \quad \text{for } \ \xi \in I \ \text{ and } \ \tau \in (\tau_0-\delta, \tau_0+\delta). 
\end{equation}

Let $C=2^\alpha {C_\alpha}/{C_h}$, as in the proof of Lemma~\ref{unbounded}, and $\e>0$ fixed. We take
\[
\xi \geq \xi_{sup}:=\sup\left\{ \xi \in (-\infty, \xi^*_{\tau_0}): \ \phi_{\tau_0}(\xi) = -(\max\{ \phi_-, (\tau_0^{-\frac{\alpha}2} C)^{\frac{1}{2-\alpha}} \}+\e)  \right\}
\]
and choose a bounded interval $I$ by means of,
\[
\begin{split}
I:=\left\{ \xi \in (\xi_{sup}, \xi^*_{\tau_0}) : 
 -(2\max\{\phi_-, (\tau_0^{-\frac{\alpha}2} C)^{\frac{1}{2-\alpha}}\}+\e) <\phi_{\tau_0}(\xi) \right.\\
 \left.<-(\max\{ \phi_-, (\tau_0^{-\frac{\alpha}2} C)^{\frac{1}{2-\alpha}} \}+\e)  \right\}.
\end{split}
\]
Thus, there exists $\delta>0$ such that (\ref{continuous_dependence_tau}) holds. Then, there we take (a maybe smaller value than the previous) $\delta$ such that $\delta<(1-2^{-2/\alpha})\tau_0$, and we define a sub-interval $J\subseteq I$ such that 
\[
 \phi_{\tau_0}(\xi)
 <-(\max\{ \phi_-, ((\tau_0 -\delta)^{-\frac{\alpha}2} C)^{\frac{1}{2-\alpha}} \}+\e)
 \qquad 
 \text{for all $\xi\in J$}.
 \]
Then, since $|\phi_{\tau_0}(\xi)| - |\phi_{\tau}(\xi)| < \e$ by (\ref{continuous_dependence_tau}),
\begin{equation}\label{bound:phi:tau}
  |\phi_{\tau}(\xi)| > |\phi_{\tau_0}(\xi)| - \e >\max\{\phi_-,((\tau_0 -\delta)^{-\alpha/2} C)^{\frac{1}{2-\alpha}}\}
\end{equation}
for all $\xi \in J$ and $\tau \in (\tau_0-\delta, \tau_0+\delta)$ for $\delta$ sufficiently small. 

Let us now argue by contradiction and suppose that there exists some $\tau \in (\tau_0-\delta, \tau_0+\delta)$ such that $\tau \notin \Sigma_u$. Then this means that $\phi_\tau$ is bounded and one can apply the first part of the proof of Lemma~\ref{unbounded}, see (\ref{phi_tau:upper:lower:bounds}), and get,
\[
\|\phi_\tau\|_\infty \leq (\tau^{-\alpha/2} C)^{\frac{1}{2-\alpha}},
\]
where $C$ as above. We recall that to get to this inequality the starting assumption is that $\phi_\tau$ has its minimum in the interval $(-\infty,-\phi_-)$. Now, the inequality contradicts (\ref{bound:phi:tau}), since $\tau_0 -\delta<\tau$.
Therefore, $(\tau_0-\delta, \tau_0+\delta) \subset \Sigma_u$ and we have that $\Sigma_u$ is open with the usual topology in $(0,+\infty)$.
\end{proof}

%>>>>>>>>>>>>>>>>>>>>>>>>>>>>>>>>>>>>>>>>>>>>>>>>>>>>>>>>>>>>>>>>>>>>>>>>>>>>>>>>>>>>>>>>

\subsection{The sets $\Sigma_c$ and $\Sigma_+$: end of the proof}\label{sigma:c:plus}
In this section we finish the proof of Theorem~\ref{TW:main:theorem}.

%First we show that $\Sigma_c$ is non-empty for $\tau\geq 0$ sufficiently small. In order to show this, we use the results in~\cite{ACH, AHS, AHS2, CA} and the results on the next section, regarding the problem with a $C^2$ modification of $h$ that has only the zeros $\phi_-$ and $\phi_c$. The idea is to use that, for a genuinely non-linear flux, the travelling wave solutions are monotone for $\tau$ sufficiently small and have no other alternative than to converge to $\phi_c$ as $\xi\to\infty$.

Before we state the result that we need, which is proved in Section~\ref{sec:control}, let us introduce the following modified problem:
\begin{equation}\label{TWP:modified}
 \tau \phi_\tau''+ \DD^\alpha[ \phi_\tau]=\tilde{h}(\phi_\tau)\,,
\end{equation}
with
\begin{equation}\label{h:tilde}
  \tilde{h}(\phi) :=\begin{cases}
  h(\phi), & \mbox{if}\ \phi\geq -\sqrt{c/3}\,,
  \\
  P_c(\phi), & \mbox{if} \ \phi\leq -\sqrt{c/3}\,,
  \end{cases}
  %\tilde{h}(\phi) := \begin{cases} h(\phi) & \forall \phi\geq -\sqrt{c/3}\,, \\ h(-\sqrt{c/3}) & \forall \phi\leq -\sqrt{c/3}\,. \end{cases}
\end{equation}
where $P_c(\phi)$ is a function such that $P_c(-\sqrt{c/3})=h(-\sqrt{c/3})$, $P_c'(-\sqrt{c/3})=h'(-\sqrt{c/3})=0$ and $P_c''(-\sqrt{c/3})=h''(-\sqrt{c/3})$, and such that $P_c(\phi)>0$ for all $\phi\leq -\sqrt{c/3}$\footnote{ %order $4$ polynomial
For example, we can choose 
\[
P_c(\phi)=A\phi^4+B\phi^3+C\phi^2+D\phi+E 
\]
such that $A>0$ and the rest of coefficients are chosen such that, $P_c(-\sqrt{c/3})=h(-\sqrt{c/3})$, $P_c'(-\sqrt{c/3})=h'(-\sqrt{c/3})=0$ and $P_c''(-\sqrt{c/3})=h''(-\sqrt{c/3})<-6\sqrt{c/3}<0$ (these give $C$ and $D$ as a linear combination of $A$ and $B$), and such that the local minimum at some $\phi_{min}<-\sqrt{c/3}$ (this gives a linear relation for $A$ and $B$, and choosing $B$ very negative guarantees that $A$ is positive) has $P_c(\phi_{min})>0$ (this is achieved by taking $E>0$ as large as necessary). This last condition guarantees that $\tilde{h}(\phi)>0$ for all $\phi\leq -\sqrt{c/3}$. 
}. Observe that at $\phi=-\sqrt{c/3}$, $h$ attains its local maximum and that $-\sqrt{c/3}<\phi_c<0$, thus this modification of $h$ is $C^2$ at the point $-\sqrt{c/3}$.

  We observe that we also denote by $\phi_\tau$ a solution to this problem (for simplicity of notation), but it will be made clear to which problem such profile corresponds.

We define, analogously to (\ref{energy}), the primitive of $\tilde{h}$, $\widetilde{H}(\phi)=\int_0^\phi \tilde{h}(y)dy$. We observe that $\tilde h$ has only two zeros, namely, $\phi_c$ and $\phi_-$, and the necessary condition, see Lemma~\ref{good:sign} (\ref{energy:form}),
\begin{equation}\label{H:tilde:conditon}
0 \leq \widetilde H(\phi) -\widetilde H(\phi_-)
= \int_{\phi_-}^{\phi} \tilde h(y) \,dy
\end{equation}
holds only for $\phi\in[\bar\phi,\phi_-]$, where $\bar{\phi}$ is the zero of $\widetilde H(\phi) -\widetilde H(\phi_-)$ that satisfies $\bar\phi<\phi_c$.
One can then adapt the results of \cite{ACH} to this equation subject to the far-field behaviour%problem, 
 \begin{equation} \label{TW:m-c}
   \lim_{\xi\to-\infty} \phi_\tau(\xi) = \phi_- \quad\text{and}\quad 
   \lim_{\xi\to+\infty} \phi_\tau(\xi) = \phi_c.	
 \end{equation}
 In fact, solutions to this problem exist for all $\tau>0$ and lie in the interval $(\bar\phi,\phi_-)$. We also recall that $\phi_-$ and $\phi_c$ satisfy the Rankine-Hugoniot condition, giving the same wave speed $c$ as in (\ref{RHC}), because $\phi_c$ is also a root of $h$.
% \[
% \frac{\phi_c^3 -\phi_-^3}{\phi_c -\phi_-} = \phi_c^2 +\phi_c \phi_- +\phi_-^2 = \phi_+^2 +\phi_+ \phi_- +\phi_-^2 = c \,.
 % \]

We shall denote by $\phi_0$ the solutions of the equation with $\tau=0$:
%First we analyse \eqref{TWP} with $\tau=0$, that is we look for solutions of the equation
\begin{equation}\label{TWP:zero}
  \DD^\alpha[ \phi_0]=\tilde{h}(\phi_0)\,,
\end{equation}
where $\tilde{h}$ is given in (\ref{h:tilde}) such that
 \begin{equation} \label{TW:m-c0}
   \lim_{\xi\to-\infty} \phi_0(\xi) = \phi_- \quad\text{and}\quad 
   \lim_{\xi\to+\infty} \phi_0(\xi) = \phi_c.	
 \end{equation}
These solutions are constructed as in \cite{AHS} (see also \cite{AHS2}) and are monotone decreasing.

With regard to the monotonicity analysis, we recall the following result of \cite{ACH}, that concludes monotonicity for $\tau$ small enough in a large interval, by comparing the solutions of (\ref{TWP:modified})-(\ref{TW:m-c}) for $\tau>0$ small to the solution of (\ref{TWP:zero})-(\ref{TW:m-c0}):
\begin{thm}[{\cite[Theorem 9]{ACH}}]\label{mono:tau:small}
%Let $\phi_\tau$ be a solution of (\ref{}) $\tau$
 %and $\phi_0$ be a travelling wave solution of the problem with $\tau=0$.
%Let $\phi_\tau$ denote a solution constructed as in Theorem~\ref{existence}. 
  Let $\phi_\tau$ be a solution of (\ref{TWP:modified})-(\ref{TW:m-c}). If $\tau$ is sufficiently small, then there exists an interval $I_\tau=(-\infty,\xi_\tau]$ with $\xi_\tau=O(\tau^{-\frac{1}{2-\alpha}})$ as $\tau\to 0^+$,
 and a value $\xi=\xi_\tau^0<\xi_\tau$ such that $\phi_\tau(\xi_\tau^0)=\phi_0(\xi_\tau^0)$, moreover, $|\phi_\tau(\xi)-\phi_0(\xi)|\leq \tau C$
 and $|\phi_\tau^\prime(\xi)-\phi_0^\prime(\xi)|\leq \tau^{1/(2-\alpha)} C$ for all $\xi\in I_\tau$. Thus for $\tau$ sufficiently small, $\phi_{\tau}$ is also monotone decreasing in $I_\tau$.
\end{thm}

The proof of this result can be adapted to the travelling wave equation (\ref{TWP:modified}), without change, since it only requires the construction of solutions on intervals of the form $(-\infty,\xi]$ and a uniform bound on the solution.

 From this result we have conjectured monotonicity in the `tail' as well (see \cite{ACH}) for sufficiently small $\tau$. This is still an open problem, but the conjecture is supported by the numerical examples in Section~\ref{sec:numerics}. Instead of monotonicity, we prove the following theorem, that is sufficient for our purpose:% is then proved in  

\begin{thm}\label{mono:tau:small:tail}
 Let $\phi_\tau$ be a solution of (\ref{TWP:modified})-(\ref{TW:m-c}).
 If $\tau>0$ is sufficiently small, then there exists a constant $C_\tau=O(\tau^{\frac{\alpha}{2-\alpha}})$ as $\tau \to 0^+$, such that $-\sqrt{c/3}<\phi_c-C_\tau<\phi_-$ and $\phi_\tau(\xi)\in (\phi_c-C_\tau ,\phi_-)$ for all $\xi\in[\xi_\tau,\infty)$, where $\xi_\tau= O\left( \tau^{-\frac{1}{2-\alpha}}\right)$ as $\tau \to 0^+$ as in Theorem~\ref{mono:tau:small}. 
\end{thm}

As we have mentioned earlier, we postpone the proof of this result to Section~\ref{sec:control}. We assume for the rest of this section that the result is true. With that we can prove the following Lemma:
\begin{lem}\label{thm:profile:m-c:tau}
Let $(\phi_-,\phi_+;c)$ satisfy the Rankine-Hugoniot condition~(\ref{RHC}) and (\ref{lin:ass}) with $\phi_c=-\phi_- -\phi_+$. If $\tau>0$ is sufficiently small then there exists a solution $\phi_\tau\in C_b^3(\R)$ of the problem (\ref{TWP}) and (\ref{TW:m-c}). It is unique (up to a shift) among all $\phi_\tau\in \phi_- + H^2(-\infty,0)\cap C_b^3(\R)$. In particular, this means that the set $\Sigma_c$ is non-empty.
\end{lem}

\begin{proof}
%If $\tau=0$ then the result follows directly from~\cite[Theorem 2]{AHS}, and observe that \cite[Theorem 1.1]{CA} allows to remove an assumption on the kernel of the linearised problem. ?¿?¿?¿
If $\tau>0$, the existence of solutions to (\ref{TWP:modified})-(\ref{TW:m-c}) is shown as in \cite{ACH}. We now use Theorem~\ref{mono:tau:small} and Theorem~\ref{mono:tau:small:tail}, in particular, for such sufficiently small $\tau$, these solutions satisfy $-\sqrt{c/3}<\phi_\tau(\xi)<\phi_-$ for all $\xi\in\R$. In this range of $\phi$'s, $h$ and $\tilde{h}$ coincide, thus, these solutions are also solutions of the original equation (\ref{TWP}), since they satisfy (\ref{TW:m-c}), and this finishes the proof. 
\end{proof}

We are in the position to finish the proof of our main result.

\begin{proof}[Proof of Theorem~\ref{TW:main:theorem}]
  Throughout this proof, $C$ denotes a positive constant of order $1$ which is independent of $\tau$ and of any other parameter or small constant that may be used.

  We have to prove that $\Sigma_+$ is non-empty. We argue by contradiction; let us assume that $\Sigma_+ = \emptyset$. Now, by Theorem~\ref{existence}, we have that $\Sigma_+$, $\Sigma_u$ and $\Sigma_c$ are disjoint, and that $\Sigma_+ \cup \Sigma_u \cup \Sigma_c= (0,+\infty)$. From Lemmas~\ref{unbounded} ,~\ref{sigma_u_open} and ~\ref{thm:profile:m-c:tau}, we then deduce that the set $\Sigma_c$ must be closed in $(0,+\infty)$. This implies the existence of a $\tau_0 \in \Sigma_c\cap \partial \Sigma_c$ such that, for all $\delta>0$ the set $\Sigma_u \cap (\tau_0-\delta, \tau_0+\delta)$ is non-empty.

Since $\tau_0 \in \Sigma_c$, the far-field behaviour of $\phi_{\tau_0}$ is given by (\ref{TW:m-c}). We recall the construction of solutions (see \cite[Lemma~2]{ACH}) satisfying (\ref{far-fieldL}): For all $\tau>0$, let $\lambda_\tau$ be the positive root of $\tau z^2+z^\alpha-h'(\phi_-)=0$ and for any $\e>0$ let also $I_{\tau,\e}=(-\infty,\xi_{\tau,\e}]$ with $\xi_{\tau,\e}=\log\e/\lambda_{\tau}$. Then, there exists an order one constant $C>0$, such that:
   \[
   \phi_{\tau}\in \phi_- + H^2(I_{\tau,\e}), \quad
   \|\phi_{\tau}-\phi_--e^{\lambda_{\tau} \xi}\|_{H^2(I_\e)}\leq C \e^2\,.
   \]
We conclude, that for all $\e>0$ there exists $\delta_\e>0$ and $\xi^\e$, defined by,
\begin{equation*}%\label{xi_barra:def} 
\xi^\e = \inf_{\tau \in (\tau_0-\delta_\e,\tau_0+\delta_\e)}\left\{\frac{\log\e}{\lambda_\tau}\right\},
\end{equation*}
such that, for all $\tau \in (\tau_0 - \delta_\e, \tau_0 + \delta_\e)$,
\begin{equation*}%\label{ineq:phi_-:1}
   |\phi_{\tau_0}(\xi) - \phi_{\tau}(\xi) |, \ |\phi_{\tau_0}'(\xi) - \phi_{\tau}'(\xi) |\ < \e, \quad\forall \xi<\xi^\e.
 \end{equation*}
%for any $\tau \in (\tau_0 - \varepsilon', \tau_0 + \varepsilon')$.

At this stage, we can apply Theorem~\ref{CDTh}, as explained in Appendix~\ref{appendix:B}, about continuous dependence on finite intervals, to get possibly a smaller neighbourhood of $\tau_0$ for which solutions and their first derivative remain close, by an order $\e$ constant, to $\phi_{\tau_0}$ and its derivative, respectively, in a much larger interval $(-\infty,\xi_{\tau_0}^\e]$ with $\xi_{\tau_0}^\e>\xi^\e$. %(REVISE THE DELTA EPSILON NOTATION ABOVE..)

  Now, we follow the steps of Appendix~\ref{appendix:v} to solve the equation implicitly as in (\ref{var:consts1}) for $a= -h'(\phi_c)>0$. For convenience we also introduce the notation:
\begin{equation}\label{h:prime:phic}
  h_c':=h'(\phi_c)<0\,.
  \end{equation}

  We split the interval of integration on the non-local operator $\DD^\alpha[\cdot]$ at $\xi_{\tau_0}^\e$. Thus, the integral part for $\xi>\xi_{\tau_0}^\e$ can be considered as a classical Caputo derivative, while the remainder is treated as a known inhomogeneity. This gives:
\begin{equation}\label{linzed:eq:phic}
\tau_0 \phi''_{\tau_0} + \DD^\alpha_{\xi_{\tau_0}^\e}[\phi_{\tau_0}] - h_c' \phi_{\tau_0} = h(\phi_{\tau_0}) - h_c' \phi_{\tau_0} - d_\alpha \int_{-\infty}^{\xi_{\tau_0}^\e}\frac{\phi'_{\tau_0}(y)}{(\xi-y)^\alpha} \, dy\,.
\end{equation}
%where % written here with more generality, 
%\[
%\DD^\alpha_{\xi_0}[g] := d_\alpha \int_{\xi_0}^\xi\frac{g'(y)}{(\xi-y)^\alpha} \, dy.
%\]
For convenience, we can shift the independent variable as $\eta=\xi-\xi_{\tau_0}^\e$ and introduce the function
\[
\Phi_{\tau_0}(\eta)=\phi_{\tau_0}(\xi) -\phi_c.
\]
Applying these changes of variables to (\ref{linzed:eq:phic}) gives the equation:
\[
 \tau_0 \Phi''_{\tau_0}(\eta) + \DD^\alpha_0\left[\Phi_{\tau_0}(\eta)\right] - h_c' \Phi_{\tau_0}(\eta) = Q(\eta;\tau_0)
\]
where $Q$ is the term that contains the remainder terms of the new formulation, namely,
\[
Q(\eta;\tau_0):= h(\phi_{\tau_0}(\eta + \xi_{\tau_0})) - h_c' \Phi_{\tau_0}(\eta) - d_\alpha \int_{-\infty}^{0}\frac{\Phi'_{\tau_0}(z)}{(\eta-z)^\alpha}\, dz.
\]
We can rewrite this term, using the explicit formula (\ref{TWP}) of $h$, as follows (we shall use this notation below)
\begin{equation}\label{Q:tau0:v2}
  Q(\eta;\tau_0):= ( \Phi_{\tau_0}(\eta) +3\phi_c ) (\Phi_{\tau_0}(\eta))^2 + R(\eta;\tau_0)%- d_\alpha \int_{-\infty}^{0}\frac{\Phi'_{\tau_0}(z)}{(\eta-z)^\alpha}\, dz.
\end{equation}
with
\begin{equation}\label{R:tau0}
  R(\eta;\tau_0) = - d_\alpha \int_{-\infty}^{0}\frac{\Phi'_{\tau_0}(z)}{(\eta-z)^\alpha}\, dz.
\end{equation}

Once we reach this point, we can write the solution $\Phi_{\tau_0}$ implicitly for all $\eta>0$ as follows: %(for convenience we denote $h'_c:= h'(\phi_c)$) 
\begin{equation}\label{Phi:tau0:VOC}
\Phi_{\tau_0}(\eta) = \Phi_{\tau_0}(0^+)v(\eta;\tau_0) + \frac{\tau_0}{h_c'} \Phi'_{\tau_0}(0^+)v'(\eta;\tau_0) + \frac{1}{h_c'} \int_{0}^{\eta}{v'(r;\tau_0) Q(\eta-r;\tau_0)\, dr} 
\end{equation}
%(REVISE THE FOLLOWING: WHEN DO WE INTRODUCE THIS FUNCTION)
where $v(\eta;\tau_0)$ is the solution of the homogeneous equation
\[
\tau_0 v'' +\DD_0^\alpha[v]-h_c' v=0 \quad \mbox{with} \quad v(0;\tau_0)=1, \ v'(0;\tau_0)=0\, .
\]
The main properties of this function $v(\eta;\tau)$ are given in Appendix~\ref{appendix:v} (for any $\tau$). Indeed, due to Lemma~\ref{v:cont:behav}, $v(\eta;\tau)$ is continuous with respect to $\tau>0$. Thus, we obtain the same continuity with respect to $\tau$ for (\ref{Phi:tau0:VOC}) since, in particular, $v'(\eta;\tau)$ and $R(\eta;\tau)$ are continuous at $\tau_0$ as well. 

%THIS IS THE BOOTSTRAP (CHECK WITH THE PREVIOUS PRARAGRAPHS):
Now, as explained above, by continuity with respect to $\tau$ around the value $\tau_0$, %(again by application of Theorem~\ref{CDTh} in the Appendix~\ref{appendix:B}),
we have that for all $\e>0$ there exists $\delta_\e>0$ and an interval $(-\infty,\xi_{\tau_0}^\e)$ such that, for all $\tau\in (\tau_0-\delta_\e,\tau_0+\delta_\e)$ and $\xi\in (-\infty,\xi_{\tau_0}^\e)$, $|\phi_{\tau_0}(\xi)-\phi_\tau(\xi)|<\e$. Now, since $\phi_{\tau_0}(\xi) \to \phi_c$ as $\xi\to\infty$, by taking $\delta_\e$ smaller if necessary, we can take $\xi_{\tau_0}^\e$ large enough so that $\phi_{\tau_0}(\xi)$ is close to $\phi_c$. Moreover, again, by continuity and making $\delta_\e$ smaller, there exists $\overline\xi_{\tau_0}^\e\gg \xi_{\tau_0}^\e$, such that $|\phi_{\tau_0}(\xi)-\phi_c|<\e$ and $|\phi_{\tau_0}(\xi)-\phi_\tau(\xi)|<\e$ for all $\tau\in (\tau_0-\delta_\e,\tau_0+\delta_\e)$ and $\xi \in (\xi_{\tau_0}^\e,\overline\xi_{\tau_0}^\e)$. We can achieve
\begin{equation}\label{order1overeps}
  \overline\xi_{\tau_0}^\e- \xi_{\tau_0}^\e =O(1/\e) \quad\mbox{as}\quad \e\to0^+,
\end{equation}
   by taking $\delta_\e$ smaller if necessary.

By assumption, $\tau_0\in\partial \Sigma_c$, $\Sigma_u$ is an open set and $\Sigma_+$ is empty. That means, that for all $\e>0$, there exists $\overline\tau \in  (\tau_0-\delta_\e,\tau_0+\delta_\e)\cap \Sigma_u$ with 
\begin{equation}\label{key:point}
|\phi_{\tau_0}(\xi)-\phi_{\overline\tau}(\xi)|<2\e \quad \mbox{for all}\quad  \xi\in  (\xi_{\tau_0}^\e,\overline\xi_{\tau_0}^\e).
\end{equation}

This $\overline\tau$ is not unique, it depends on $\e$ (via $\delta_\e$), we skip this in the notation for simplicity, but we have to bear in mind that by taking other values of $\e$ there is always such a $\overline\tau$.

Due to $\bar\tau \in \Sigma_u$ and Lemma~\ref{blow-down}, there exists $\xi^*_{\bar\tau}$, hence $\eta^*:=\xi^*_{\bar\tau}-\xi_{\tau_0}^\e$, such that
\begin{equation}\label{Phi:overlinetau:*}
\Phi_{\overline{\tau}}(\eta) \sim -\frac{C\sqrt{\overline\tau}}{\eta^*-\eta}\, , \ \ \text{as} \ \ \eta \to (\eta^*)^-\, .
\end{equation}
Then, $\Phi_{\overline{\tau}}(\eta)=\phi_{\overline\tau}(\eta+\xi_{\tau_0}^\e)-\phi_c$, in the interval of existence, can be written as:
\begin{equation}\label{Phi:overlinetau:VOC}
  \Phi_{\overline{\tau}}(\eta) = \Phi_{\overline{\tau}}(0^+)v(\eta;\overline{\tau}) + \frac{\overline{\tau}}{h_c'} \Phi'_{\overline{\tau}}(0^+)v'(\eta;\overline{\tau}) + \frac{1}{h_c'} \int_{0}^{\eta}{v'(r;\overline{\tau}) Q(\eta-r;\overline{\tau})\, dr}\,,
\end{equation}
where necessarily $\eta\leq\overline\xi_{\tau_0}^\e-\xi_{\tau_0}^\e <\eta^*=\xi^*_{\bar\tau}-\xi_{\tau_0}^\e$, where $\xi^*_{\bar\tau}$ depends on $\overline\tau$.% is as in Lemma~\ref{blow-down}. %Moreover, this means that %there exists a finite value $\xi^* \in \R$  such that 

If we take into account this behaviour in (\ref{Phi:overlinetau:*}), in terms of $\eta$, we notice that the first two terms in (\ref{Phi:overlinetau:VOC}) are uniformly bounded by a constant proportional to $\e$: this is because $v(\eta;\overline{\tau})\sim \eta^{-\alpha}$ and $v'(\eta;\overline{\tau})\sim \eta^{-(1+\alpha)}$ as $\eta\to \infty$ (see Appendix~\ref{appendix:v} Lemma~\ref{v:cont:behav}), and by $|\Phi_{\overline\tau}(0)|<\e$ and $|\Phi_{\overline\tau}'(0)|<C$, where this holds by continuity in $\tau$ near $\tau_0$.

The integral term in (\ref{Phi:overlinetau:VOC}) with $R(\eta;\overline{\tau})$ i.e. the linear part of $Q(\eta;\overline{\tau})$, can be controlled again by continuity with respect to $\tau$, since, by construction, this term is close to the one for $\tau_0$ (observe that this term corresponds to the solution with $\xi<\xi_{\tau_0}^\e$ that is given and close to $\phi_{\tau_0}$ in $(-\infty,\xi_{\tau_0}^\e]$). On the other hand, the integral term with the first two terms of $Q(\eta;\overline{\tau})$ govern the asymptotic behaviour for large $\eta$ but with $\eta <\eta^*$. Let us make this more precise.

Let us assume that $\eta<\eta^*$ and let $m>0$ be a small number (that we relate to $\e$ below), and $M \in (m,\eta)$. Then we can split the integral of (\ref{Phi:overlinetau:VOC}) in the following four terms:
\begin{align}%\label{VOC:int}
  \int_0^\eta  v'(r;\overline\tau)Q(\eta-r;\overline\tau)dr =\int_0^m v'(r;\overline\tau)( \Phi_{\overline\tau}(\eta-r) +3\phi_c ) (\Phi_{\overline\tau}(\eta-r))^2dr \label{VOC:int}\\
  +  \int_m^M v'(r;\overline\tau) ( \Phi_{\overline\tau}(\eta-r) +3\phi_c ) (\Phi_{\overline\tau}(\eta-r))^2 dr\label{VOC:int2} \\
  +\int_M^\eta v'(r;\overline\tau) ( \Phi_{\overline\tau}(\eta-r) 
  + 3\phi_c ) (\Phi_{\overline\tau}(\eta-r))^2 dr \label{VOC:int3}\\
  + \int_0^\eta  v'(r;\overline\tau) R(\eta-r;\overline\tau)dr. \label{VOC:int4}
\end{align}

Take $m=\e^p$ with $p> 1/2$, and $M=O(1/\e)$, the latter is possible by (\ref{order1overeps}). This last condition allows to control the term (\ref{VOC:int2}) by an $O(\e)$ constant: since in this range $\Phi_{\overline\tau}$ satisfies (\ref{key:point}), the non-linear term contribution is quadratic. The last non-linear term (\ref{VOC:int3}) is controlled by the behaviour of $v'(\eta)$ for large $\eta$ ($v'(\eta)\sim \eta^{-(1+\alpha)}$ as $\eta\to \infty$ and $\eta>M=O(1/\e)$) and also by the fact that in this range $\Phi_{\overline\tau}$ is small and the non-linear contribution is quadratic. As mentioned above, the last integral (\ref{VOC:int4}) can be controlled by continuity with respect to $\tau$ in $v$ and $R$, since this is a perturbation of the similar term in equation (\ref{Phi:tau0:VOC}), but $\Phi_{\tau_0}(\eta) \to 0$ as $\eta\to\infty$, thus it must be smaller than a constant of $O(\e)$.

Thus, except for the first integral term, which is the right-hand side of (\ref{VOC:int}), all other terms stay below (in absolute value) a constant of order $\e$ (as a consequence, mainly of (\ref{key:point}) and the behaviour of $v$ for large $\eta$).

Suppose that $\eta^*=\eta + C_\e$ where $C_\e>0$, let us see how the behaviour of such distance to $\eta^*$ can be estimated in terms of $\e$. The right-hand side of (\ref{VOC:int}) can be estimated as, 
\[
\left|\int_0^m  v'(r;\overline\tau) ( \Phi_{\overline\tau}(\eta-r) +3\phi_c ) (\Phi_{\overline\tau}(\eta-r))^2dr \right| \leq \frac{C \, m^2 \sqrt{\overline\tau}}{(\eta^*-\eta)^3},
\]
where we have used (\ref{Phi:overlinetau:*}) (this is the worst case scenario, where the dominant term of the non-linear part is $(\Phi_{\overline\tau})^3$), and, see Appendix~\ref{appendix:v} Lemma~\ref{v:cont:behav}, that
\[
v'(r;\overline{\tau})= O\left(\frac{h_c'}{\overline{\tau}} r\right)  \quad\mbox{as} \ r\to 0^+,
\]
in particular $v'(r)<0$ for small, but positive, values of $r$. So as long as $(\eta^*-\eta)> C_\e$ with  $C_\e\propto \e^{\frac{2p-1}{3}}(\overline\tau)^{1/6}$, the first term in (\ref{VOC:int}) is smaller in absolute value than a constant of order $\e$. Then, putting together the estimates on all other terms of (\ref{Phi:overlinetau:VOC}), we get $|\Phi_{\overline\tau}(\eta)|<\e C$.
%, then in fact the behaviour of this term must be quadratic a long as $(\eta^*-\eta)> C_\e$.
%, this means that in fact (by a similar argument but with the power $2$ in the denominator) this term is smaller than a constant of the order $\e \e^{\frac{2p-1}{3}}$.

Then we can take $\eta$ larger, thus closer to $\eta^*$: Indeed we can get a smaller $m$, by taking $p$ larger, to guarantee that $|\Phi_{\overline\tau}(\eta)|< \e C$ by keeping $C_\e\propto \e^{\frac{2p-1}{3}}(\overline\tau)^{1/6}$, which is the distance to $\eta^*$. Thus the quadratic behaviour of the non-linear term is valid for as close as we want to $\eta-m$, since we can keep $|\Phi_{\overline\tau}(\eta)|< \e C$ , thus the estimates on the other terms of (\ref{Phi:overlinetau:VOC}) can be done as before for the new $\eta$ and $m$. We can repeat the process for $\eta$ closer and closer to $\eta^*$ by taking a larger $p$ for any $\e$. 

On the other hand, before blow-up, but close to it, we have that there is a constant $C$ proportional to $\sqrt{\overline\tau}$ but independent of $\e$ (see Lemma~\ref{blow-down}-(\ref{blow-down:bhv})), such that $|\phi_{\overline\tau}(\xi)|>C\sqrt{\overline\tau}/( \xi^*-\xi)$ near $\xi^*$, or in terms of the variable $\eta$
\[
|\Phi_{\overline\tau}(\eta)+\phi_c|>C\frac{\sqrt{\overline\tau}}{( \eta^*-\eta)}. 
\]
The previous argument allows also $|\Phi_{\overline\tau}(\eta)|<\e C$ near $\eta^*$, but then, this would imply that $\eta^*-\eta > C \sqrt{\overline\tau}$ which contradicts the previous estimates, unless $\overline\tau$ (hence $\tau_0$) is very small (by making the balance $\sqrt{\overline\tau}\sim \e^{\frac{2p-1}{3}}(\overline\tau)^{1/6}$ gives that $\overline\tau\sim  \e^{2p-1}$). If this is the case (taking $\e$ smaller) then $\overline\tau \in\Sigma_c$ by Lemma~\ref{thm:profile:m-c:tau}, which gives a contradiction. %(Prop that show $\Sigma_c$ is non-empty), a contradiction.
\end{proof}

%<<<<<<<<<<<<<<<<<<<<<<<<<<<<<<<<<<<<<<<<<<<<<<<<<<<<<<<<<<<<<<<<<<<<<<<<<<<<<<<<

\section{Control of the modified problem for small $\tau$: Proof of Theorem~\ref{mono:tau:small:tail}}\label{sec:control}
In this section we prove that solutions of (\ref{TWP:modified})-(\ref{TW:m-c}), provided that $\tau$ is sufficiently small, remain within the range where $h$ and $\tilde{h}$ coincide. As before, we use the notation $\phi_\tau$ for solutions to this problem which are constructed as in \cite{ACH}. Throughout we adopt the notation of Theorem~\ref{mono:tau:small} and assume that $\tau$ is small enough so that the conclusion of this theorem applies.

At this step, we introduce the implicit formulation for the equation (\ref{TWP:modified}) that is used in the proofs that follow. We write the equation as the linearised equation around $\phi_c$ together with the remainder terms as we did in the proof of Theorem~\ref{TW:main:theorem}. We notice that such formulation for equation (\ref{TWP:modified}) is similar, with the obvious changes; the linearised part being the same for $h$ and $\tilde{h}$. The difference is the choice where we take the shift on $\xi$ that defines the new variable $\eta$ (for which we use the same notation) in the current case. As before, we use the notation
\[%begin{equation}\label{h:prime:phic}
  h_c':=h'(\phi_c)=\tilde{h}'(\phi_c)<0\,.
  \]%end{equation}

For the proof of Theorem~\ref{mono:tau:small:tail} with $\tau$ sufficiently small, we shall let $\xi_0\in\R$ such that $\xi_0\ll \xi_\tau$ and, by Theorem~\ref{mono:tau:small}, such that %(which will be chosen appropriately in the proofs).
%Where the $\delta$ subscript refers to the value:
\begin{equation}\label{xi:0}%\label{delta:xi:def}
\phi_\tau(\xi_0) - \phi_c >0\,.
\end{equation}
%And that $\phi_\tau '(\xi)<0$ for all $\xi\leq xi_0 $.
We can take $\xi_0$ sufficiently away from $\xi_\tau$, to guarantee that $|\phi_\tau'(\xi_0)|\ll 1$ (by the exponential behaviour as $\xi\to-\infty$) although it is not necessary; it will be enough to have $|\phi_\tau'(\xi_0)|$ of order one for $\tau$ small. In particular, with such choice we know that $\phi_\tau '(\xi)<0$ for all $\xi\leq \xi_0$, but also for $\xi_0<\xi<\xi_\tau$.

%also on  lso

Let us split the interval of integration of the non-local operator at this $\xi_0$, %, so the integral part for $\xi>\xi_0$ can be considered as a classical Caputo derivative, and the other as a known inhomogeneity.
which gives, as before:
\begin{equation}\label{linzed:eq}
\tau \phi''_\tau + \DD^\alpha_{\xi_0}[\phi_\tau] - h_c' \phi_\tau = \tilde{h}(\phi_\tau) - h_c' \phi_\tau - d_\alpha \int_{-\infty}^{\xi_0}\frac{\phi'_\tau(y)}{(\xi-y)^\alpha} \, dy\,,
\end{equation}
where we have used again the notation, %as before,% written here with more generality, 
\[
\DD^\alpha_{\xi_0}[g] := d_\alpha \int_{\xi_0}^\xi\frac{g'(y)}{(\xi-y)^\alpha} \, dy.
\]
And the translation of the independent variable is done by means of $\eta=\xi-\xi_0$ and, as earlier, we introduce the dependent variable %function
\begin{equation}\label{Phi:phic:def}
\Phi_\tau(\eta)=\phi_\tau(\xi) -\phi_c.
\end{equation}
Applying these changes of variables in (\ref{linzed:eq}), give the equation:
\begin{equation}\label{TWP:eta}
 \tau \Phi''_\tau(\eta) + \DD^\alpha_0\left[\Phi_\tau(\eta)\right] - h_c' \Phi_\tau(\eta) = Q(\eta)
\end{equation}
where $Q$ is defined, similarly as before, by means of
\begin{equation}\label{lin:phic2}
Q(\eta):= \tilde{h}(\phi_\tau(\eta + \xi_0)) -h_c'\Phi_\tau(\eta) - d_\alpha \int_{-\infty}^{0}{\frac{\Phi'_\tau(z)}{(\eta-z)^\alpha}\, dz}.
\end{equation}

Again, we can implicitly write the solution to (\ref{TWP:eta})-(\ref{lin:phic2}), by Appendix~\ref{appendix:v} and, e.g., \cite{BK}. This gives:
%after the linearisation one can use the method of the Laplace transform and get the following identity for the solution $\Phi_\tau$, (recall $h'(\phi_c)<0$),
\begin{equation}\label{Phi:Cap}
\Phi_\tau(\eta) = \Phi_\tau(0^+)v(\eta) + \frac{\tau}{h_c'} \Phi'_\tau(0^+)v'(\eta) + \frac{1}{h_c'} \int_{0}^{\eta}{v'(y) Q(\eta-y)\, dy} 
\end{equation}
where $v$ is the solution to the homogeneous equation
%and is given by
\begin{equation}\label{linear:eq:sol}
\tau v'' +\DD_0^\alpha[v]-h_c' v=0 \quad \mbox{with} \quad v(0)=1, \ v'(0)=0.
\end{equation}
%CHECK POSSIBLE UNNECESSARY REPETTITIONS WHIT THE OTHER CASE

For simplicity of notation, we shall not write the dependency of $v$ and $Q$ on $\tau$ throughout this section, since essentially $\tau$ is fixed.

We recall that the properties and behaviour of $v$ and its derivatives are given in the Appendix~\ref{v:monotone}, Lemma~\ref{v:bhv:tau:small}. Observe that $v$ is the same for both linearisations of the associated problems (\ref{TWP}) and (\ref{TWP:modified}) at $\phi = \phi_c$.

%In what follows we prove that, for sufficiently small $\tau$'s, then the monotonicity also holds in the tail, that is on the interval $(\xi_\tau,\infty)$ for solutions of (\ref{TWP:modified}). Before that we need the following estimates:

We shall need the following lemma:
%THIS SEEMS TO BE ONLY NEEDED FOR EXTRA RESULT (algebraic decay?)
\begin{lem}\label{basic:ests}
  If $\Phi'_\tau(\eta)<0$ in the interval $(-\infty,0)$, then for $\eta>0$:
  \begin{equation}\label{eta:large}
  \int_{-\infty}^{0} \frac{|\Phi_\tau'(z)|}{(\eta-z)^\alpha} \, dz  \leq
  \frac{C}{\eta^{\alpha+1}} + \frac{C'\left|\Phi_\tau(\eta)\right|}{\eta^\alpha}
  \end{equation}
  and 
\begin{equation}\label{eta:small}
  \int_{-\infty}^{0}\frac{|\Phi_\tau'(z)|}{(\eta-z)^\alpha} \, dz
  \leq \frac{C}{1+ \eta^\alpha}.%\leq \frac{C}{ (\eta-y)^\alpha}
\end{equation}
Moreover, for the modified problem (\ref{TWP:modified})-(\ref{TW:m-c}), we have the following upper and lower bounds, there exists $C_h \geq 0$, such that
\begin{equation}\label{Q:upper:basic}
% \begin{split}
 Q(\eta) 
 \geq -\frac{C_h}{2} \Phi_\tau^2(\eta) +  d_\alpha \int_{-M}^{0} \frac{(-\Phi_\tau'(z))}{(\eta-z)^\alpha} \, dz 
\end{equation}
for any $0<M<\infty$, and
\begin{equation}\label{Q:lower:basic}
% \begin{split}
 Q(\eta)\leq  \frac{C_h}{2}  \ \Phi_\tau^2(\eta) + d_\alpha \int_{-\infty}^{0} \frac{(-\Phi_\tau'(z))}{(\eta-z)^\alpha} \, dz \leq \frac{C_h}{2}  \ \Phi_\tau^2(\eta) -\DD^\alpha[\Phi_\tau](0^+).   
% \end{split}
 \end{equation}

\end{lem}
\begin{proof}
  Observe that (\ref{eta:large}) and (\ref{eta:small}) are obtained as in
  \cite{DC}. The last estimate might be used for small values of $\eta$,
  and the first one for moderate or large values of $\eta$. The constants $C$
  in both estimates are at most of order one, but we cannot guarantee that they
  are small.

  The last two inequalities (\ref{Q:upper:basic}) and (\ref{Q:lower:basic}),
  simply follow by applying Taylor's theorem to $\tilde{h}(\phi)$ centred at
  $\phi_c$, since there exists, for each $\eta>0$, $\tilde{\phi}_\eta \in
  [\inf_{\xi\in\R}\phi_\tau,\phi_-)$, such that %(\bar{\phi},\phi_-)$,
\begin{equation}\label{h:taylor}
  \tilde{h}(\phi_\tau(\eta+\xi_\delta)) - h_c' \Phi_\tau(\eta)=
  \frac{\tilde{h}''(\tilde{\phi}_\eta)}{2} (\Phi_\tau(\eta))^2.
\end{equation}
We recall that solutions of (\ref{TWP:modified})-(\ref{TW:m-c}) lie in $(\bar\phi,\phi_-)$, where $\bar \phi<\phi_-$ is the other zero of (\ref{H:tilde:conditon}). Thus $\inf_{\xi\in\R}\phi_\tau(\xi) \geq \bar \phi$. Then $|\tilde{h}''(\tilde{\phi}_\eta)|\leq  \max_{\phi\in[\bar\phi,\phi_-]}|\tilde{h}''(\phi)|=:C_h$.
\end{proof}

\begin{proof}[Proof of Theorem~\ref{mono:tau:small:tail}]
 We recall that $\lim_{\xi\to\infty}\phi_\tau(\xi)=\phi_c$ and using the information from Theorem~\ref{mono:tau:small}, we can take $\xi_0<\xi_\tau$, where $\xi_\tau$ is given in this theorem, such that for $\tau>0$ sufficiently small, and since $\phi'<0$ in $(-\infty,\xi_\tau)$, $\phi(\xi_0)$, $\phi'(\xi_0)$ are of $O(1)$ if not smaller as $\tau\to 0^+$, and certainly we have%. The $\delta$ here is%  subscript refers to the value:
\begin{equation}\label{delta:def}
\phi_\tau(\xi_0) - \phi_c>0.
\end{equation}

We shall need to estimate $\Phi_\tau'$ in (\ref{Phi:phic:def}), for which we shall then use the equation
\begin{equation}\label{Phi:prime}
  \Phi'_\tau(\eta) = \Phi_\tau(0^+)v'(\eta)
  + \frac{\tau}{h_c'} \Phi'_\tau(0^+)v''(\eta)
  + \frac{1}{h_c'} \int_{0}^{\eta}v''(y) Q(\eta - y)\,dy,
\end{equation}
that results from differentiating (\ref{Phi:Cap}) and using that $v'(0)=0$ by (\ref{linear:eq:sol}). 

We will use the results of Lemma~\ref{v:bhv:tau:small}. In particular, we know that $v'(\eta)<0$ for all $\eta$ if $\tau$ is sufficiently small, that $v''(\eta)<0$, holds for $\tau$ sufficiently small and for $\eta\leq \eta_{inflex}$ with $\eta_{inflex}=O(\tau^{\frac{1}{2-\alpha}})$ and that $v''(\eta)$ is non-negative otherwise, see (\ref{v'':sign}).

{\bf STEP 1:}  Let $\eta:=\xi - \xi_0$, which we consider to be large. In the first step, we need to obtain a bound for $|\Phi_\tau'(\eta)\eta_{inflex}|$ for $\eta>\eta_{inflex}$ when $\tau$ is sufficiently small. Therefore, we consider (\ref{Phi:prime}) and using the estimates of Lemma~\ref{v:bhv:tau:small}, ((\ref{v'':sign}) in the integral part, and (\ref{bh:pass:inflex})-(\ref{bh:pass:inflex:vpp}) for the other two terms), and the boundedness of $Q$, which follows from (\ref{Q:lower:basic}). Then, we get for $\eta>\eta_{inflex}$ and $\tau$ sufficiently small
\[
\begin{split}
\left| \Phi_\tau'(\eta)\right| \leq& \left|\Phi_\tau(0^+)v'(\eta) + \frac{\tau}{h_c'} \Phi'_\tau(0^+)v''(\eta)\right| + \frac{1}{|h_c'|}\int_{0}^{\eta}|v''(y)| |Q(\eta - y)|\,dy \\
\leq& \frac{K(\tau)}{\eta^{\alpha+1}} + C\tau
%+\frac{C'}{|h_c'|}\int_{0}^{\eta}|v''(y)|\,dy = \frac{C}{\eta^{\alpha+1}}
+ \frac{C'}{|h_c'|}\left(\int_{0}^{\eta_{inflex}}(-v''(y))\,dy + \int_{\eta_{inflex}}^{\eta}v''(y)\,dy\right) \\
%=&\frac{C}{\eta^{\alpha+1}} + \frac{C'}{|h_c'|}(-2v'(\eta_{inflex}) + v'(\eta))
\lesssim &\frac{K(\tau)}{\eta^{\alpha+1}} + C\tau + \frac{2 C'}{|h_c'|}(-v'(\eta_{inflex}))
\lesssim \frac{\eta_{inflex}}{\tau} \sim \tau^{\frac{\alpha-1}{2-\alpha}}.
\end{split}
\]
Here we are applying the knowledge we have on the sign of $v''$ and we are interested in obtaining a bound of the derivative depending on $\tau$. Notice that this bound, despite having a negative exponent, is better than the one we obtain from the energy estimate (\ref{energy:form}) which is of order $\tau^{-1/2}$ for $\tau$ sufficiently small.

Consequently, we get the following estimate, which we will need below, for $\eta>\eta_{inflex}$
\begin{equation}\label{Phi:bound:tau}
|\Phi_\tau'(\eta)| \, \eta_{inflex} \lesssim \tau^{\frac{\alpha}{2-\alpha}} \quad \mbox{as} \quad \tau \to 0^+ \, .
\end{equation}

{\bf STEP 2:} In this step we take advantage of the behaviour of $v'$ near $\eta_{inflex}$ for $\tau$ sufficiently small. We  prove that $v'$ behaves as an approximation of the Dirac delta distribution at $\eta_{inflex}$ and then we can approximate $\int_{0}^{\eta} v'(y) Q(\eta-y) \, dy$ by $Q(\eta-\eta_{inflex})$ for $\tau$ sufficiently small. In fact, in this step we prove that for $\tau$ sufficiently small, there exist an order one constant $C>0$ such that
\begin{equation}\label{Q:approx}
  \left| \frac{1}{h_c'} \int_{0}^{\eta} v'(y) Q(\eta-y) \, dy- Q(\eta-\eta_{inflex})\right| < C \tau^{\frac{\alpha}{2-\alpha}}\, .
  \end{equation}

First, we compute the maximum of $(-v')$, which is attained at the inflection point $\eta_{inflex}$. Evaluating $(-v')$ at $\eta_{inflex}$, and using Lemma~\ref{v:bhv:tau:small} (\ref{v:prime:eta0:corr}), we get
\[
0<(-v'(\eta_{inflex})) \sim \frac{|h'_c|}{\tau} \eta_{inflex} - \frac{1}{(3-\alpha)(2-\alpha)}\frac{|h'_c|}{\tau^2} \eta_{inflex}^{3-\alpha} \sim \tau^{\frac{\alpha-1}{2-\alpha}}\,,
\]
as $\tau \to 0^+$. Notice that for $0<\alpha<1$, the previous exponent is negative, implying that this maximum tends to $+\infty$ as $\tau \to 0^+$ and the inflection point $\eta_{inflex}$ approaches $0$.

Next, we derive the rescaling by first translating the original variable $\eta$ by $\eta_{inflex}$ and considering the behaviour of $v'$ (\ref{v:prime:eta0:corr}) for small values of $\tau>0$, in particular, for $\tau\leq\eta\leq \eta_{inflex}$. Using the asymptotic behaviour of $v'$ as $\tau \to 0^+$, we have for very small values of $\tau$
\begin{equation}\label{rescale:vp}
\begin{split}
&(-v'(\eta))\sim \frac{|h'_c|}{\tau} \eta - \frac{1}{(3-\alpha)(2-\alpha)}\frac{|h'_c|}{\tau^2} \eta^{3-\alpha}\\% \sim \tilde{v}(\eta-\eta_{inflex}) \\ 
\sim& \frac{|h'_c|}{\tau} \eta_{inflex} - \frac{1}{(3-\alpha)(2-\alpha)}\frac{|h'_c|}{\tau^2} \eta_{inflex}^{3-\alpha} + \left( \frac{|h'_c|}{\tau} - \frac{1}{2-\alpha}\frac{|h'_c|}{\tau^2} \eta_{inflex}^{2-\alpha}  \right) (\eta - \eta_{inflex}) \\
\sim& C_1 \tau^{\frac{\alpha-1}{2-\alpha}} + C_2 \frac{1}{\tau} (\eta -\eta_{inflex}) = \tau^{\frac{\alpha-1}{2-\alpha}} \left(C_1  + C_2 \frac{\eta - \eta_{inflex}}{\tau^{\frac{1}{2-\alpha}}}\right), 
\end{split}
\end{equation}
such that $\eta \in (\eta_{inflex}-\varepsilon, \eta_{inflex}+\varepsilon)$ for some small $\varepsilon>0$.

Let us define, for convenience, the following function
\[
\tilde{v}(\eta-\eta_{inflex}):= (-v'(\eta)),
\]
then we observe that:
\begin{equation}\label{vp:int:1}
  1=\int_{0}^{\infty} (-v'(\eta)) \, d\eta = \int_{0}^{\infty} \tilde{v}(\eta-\eta_{inflex}) \, d\eta.
\end{equation}

Now we use (\ref{rescale:vp}) to motivate the following definition (or rescaling of $-v'$):%Observe the Consequently, we define
\[
(-v'(\eta)) = \frac{1}{\tau^{\frac{1-\alpha}{2-\alpha}}} \omega\left( \frac{\eta - \eta_{inflex}}{\tau^{\frac{1}{2-\alpha}}} \right)= \omega_\tau(\eta-\eta_{inflex})\, ,
\]
such that $\eta \in (\eta_{inflex}-\varepsilon, \eta_{inflex}+\varepsilon)$ for some small $\varepsilon>0$.

In order to analyse the mollification, we now consider the expression,
\[
\begin{split}
 &\mathcal{Q}(\eta) := \frac{1}{h'_c}\int_{0}^{\eta} v'(y)Q(\eta-y) \, dy =
  \frac{1}{|h'_c|}\int_{0}^{\eta} (-v'(y)) Q(\eta-y) \, dy \\
=& \frac{1}{|h'_c|}\int_0^\infty \tilde{v}(y- \eta_{inflex})  Q(\eta-y) \chi_{[0,\eta]}(y) \, dy =  \int_0^\infty \tilde{v}(y- \eta_{inflex})  \overline{Q}_\eta(\eta-y)   \, dy,
\end{split}
\]
where $\overline{Q}_\eta (z):=  \frac{1}{|h'_c|}Q(z) \chi_{[0,\eta]}(\eta-z)$.% Observe that $\overline{Q}_\eta(\eta)=\frac{1}{|h'_c|}Q(\eta)$. 

Thus, we integrate at least over $(0,\infty)$ to apply (\ref{vp:int:1})
\begin{equation}\label{prop:norm:1}
\overline{Q}_\eta(z) = \int_{0}^{\infty} \tilde{v} (y-\eta_{inflex}) \overline{Q}_\eta(z) \, dy \quad \mbox{for} \quad  z\in  \R.
\end{equation}
Since $Q$ and, hence, $\overline{Q}_\eta$ are locally integrable, Lebesgue's Differentiation Theorem implies that
\[
\lim_{r \to 0} \frac{1}{2r}  \int_{z-r}^{z + r}
  |\overline{Q}_\eta(y)- \overline{Q}_\eta(z)| \, dy = 0
\]
for almost every $z\in \R$. In particular, note that the previous identity holds for $Q$ as well.

Now, for a fixed $\eta>0$, $\tau>0$ small enough and $r \lesssim \frac{\tau^{\frac{1}{2-\alpha}}}{2}$, applying the change of variable $z=\eta - y$ and (\ref{prop:norm:1}), we get 
\[
\begin{split}
&|\mathcal{Q}(\eta)- \overline{Q}_\eta(\eta-\eta_{inflex})| = \left| \int_{-\infty}^{\eta}\tilde{v}(\eta-(z+\eta_{inflex})) (\overline{Q}_\eta(z) - \overline{Q}_\eta(\eta-\eta_{inflex})) \, dz\right| \\ 
=& \left| \int_{B(\eta-\eta_{inflex}, r)} \omega_\tau(\eta-(z+\eta_{inflex})) (\overline{Q}_\eta(z) - \overline{Q}_\eta(\eta-\eta_{inflex})) \, dz\right.  \\ 
&\left.+ \int_{B(\eta-\eta_{inflex}, r)^c}  \tilde{v}(\eta-(z+\eta_{inflex})) (\overline{Q}_\eta(z) - \overline{Q}_\eta(\eta-\eta_{inflex})) \, dz\right|\, .
\end{split}
\]
On the one hand, by applying the inequality $\tau^{\frac{-1}{2-\alpha}} r \lesssim 1/2$ and the Lebesgue's Differentiation Theorem, we obtain the following estimate for a constant $C>0$,
\begin{equation}\label{int:conv:b}
\begin{split}
& \left| \int_{B(\eta-\eta_{inflex}, r)} \omega_\tau(\eta-(z+\eta_{inflex})) (\overline{Q}_\eta(z) - \overline{Q}_\eta(\eta-\eta_{inflex})) \, dz\right| \\
\leq& \frac{1}{\tau^{\frac{1-\alpha}{2-\alpha}}}\int_{\eta-\eta_{inflex}-r}^{\eta-\eta_{inflex}+r} \omega\left(\frac{\eta-(z+\eta_{inflex})}{\tau^{\frac{1}{2-\alpha}}}\right) |\overline{Q}_\eta(z) - \overline{Q}_\eta(\eta- \eta_{inflex})| \, dz \\
\leq& C \tau^{\frac{\alpha-1}{2-\alpha}} r \left(
\frac{1}{2r} \int_{B(\eta-\eta_{inflex}, r)} |\overline{Q}_\eta(z)- \overline{Q}_\eta(\eta- \eta_{inflex})| \, dz \right)
  \\
  \leq& \frac{C}{2} \tau^{\frac{\alpha}{2-\alpha}}
 \left(
\frac{1}{2r} \int_{B(\eta-\eta_{inflex}, r)} |\overline{Q}_\eta(z)- \overline{Q}_\eta(\eta- \eta_{inflex})| \, dz \right)\rightarrow 0 \quad \mbox{as} \ r \to 0^+.
\end{split}
\end{equation}
%as $r \to 0^+$. 

On the other hand, we analyse the complementary part as follows. First, we determine the complementary of $B(\eta-\eta_{inflex}, r)$ and work on the original variable of integration, we do the change of variable $y=\eta- z$. Then, in the first inequality we apply the estimates on the behaviour of $v'$ (\ref{bh:pass:inflex}) and (\ref{v:prime:eta0:corr}). Finally, we can apply the uniform bound for $Q$ derived from the inequality (\ref{Q:lower:basic}), since $\Phi_\tau$ is uniformly bounded in this scenario, and yield the following estimate for some positive constants $C$ and $C'$: 

\begin{equation}\label{int:conv:c}
\begin{split}
& \left| \int_{B(\eta-\eta_{inflex}, r)^c} \tilde{v}(\eta-(z+\eta_{inflex})) (\overline{Q}_\eta(z) - \overline{Q}_\eta(\eta-\eta_{inflex})) \, dz\right| \\
=& \left| \int_{-\infty}^{\eta-\eta_{inflex} - r} + \int_{\eta-\eta_{inflex} +r}^\eta \tilde{v}(\eta-(z+\eta_{inflex})) (\overline{Q}_\eta(z) - \overline{Q}_\eta(\eta-\eta_{inflex})) \, dz \right| \\
=& \left| \int^{\infty}_{\eta_{inflex} + r} + \int^{\eta_{inflex}-r}_0 \tilde{v}(y -\eta_{inflex}) (\overline{Q}_\eta(\eta-y) - \overline{Q}_\eta(\eta-\eta_{inflex})) \, dy \right| \\
\leq& \int^{\infty}_{\eta_{inflex} + r}   \frac{K'(\tau)}{y^{\alpha+1}}\,  |\overline{Q}_\eta(\eta-y) - \overline{Q}_\eta(\eta-\eta_{inflex})| \, dy \\
&+ \tilde{C} \int^{\eta_{inflex} -r}_0  \frac{y}{\tau} \, |\overline{Q}_\eta(\eta-y) - \overline{Q}_\eta(\eta-\eta_{inflex})| \, dy \\
\leq& C_1 \frac{K'(\tau)}{(\eta_{inflex} +r)^\alpha}  + C_2  \frac{(\eta_{inflex}-r)^2}{\tau}%\\ 
\leq
%&
C \frac{\tau^{\frac{2\alpha}{2-\alpha}}}{\tau^{\frac{\alpha}{2-\alpha}}}  + C'  \frac{\tau^{\frac{2}{2-\alpha}}}{\tau} \sim \tau^{\frac{\alpha}{2-\alpha}}  \quad \mbox{as} \quad  r \to 0^+\, .
\end{split}
\end{equation}
Consequently, combining (\ref{int:conv:b}) and (\ref{int:conv:c}) gives (\ref{Q:approx}).

{\bf STEP 3:} In this step we use the previous steps to estimate (\ref{Phi:Cap}) from below and from above for $\eta\gtrsim \tau^{-1/(2-\alpha)}$ for $\tau$ sufficiently small.

Recall that $\eta=\xi - \xi_0$ and note that $\xi_0<\xi_{\tau}$ is not necessarily large. Consequently, at $\eta=0$, $\Phi_\tau$ and $\Phi_\tau'$ are not necessarily small. However, by Theorem~\ref{mono:tau:small}, $\Phi_\tau$ is decreasing up to an $\eta_\tau= \xi_\tau - \xi_0=O(\tau^{-1/(2-\alpha)})$ for $\tau>0$ sufficiently small.

Using equation (\ref{Q:approx}), we establish the bound  %  for some positive constant $C_1, C_2$ and $\tau$ small enough, we get
\begin{equation}\label{VOC:int:bound}
 Q(\eta - \eta_{inflex}) - C\tau^{\frac{\alpha}{2-\alpha}} \leq \frac{1}{h'_c}\int_{0}^{\eta} Q(\eta-y) v'(y) \, dy \leq  Q(\eta-\eta_{inflex}) + C\tau^{\frac{\alpha}{2-\alpha}} 
\end{equation}
for $\eta > \eta_{inflex}$.

Next, we consider the behaviour of $v$ and $v'$ for small values of $\tau$, (\ref{bh:pass:inflex}), to yield
\begin{equation}\label{f:bhv}
  f(\eta)=\Phi_\tau(0^+)v(\eta) + \frac{\tau}{h_c'} \Phi'_\tau(0^+)v'(\eta)\sim \frac{C_\tau}{\eta^\alpha}
\end{equation}
for $\eta>\eta_{inflex}$, where $C_\tau = O\left( \tau^{\frac{2\alpha}{2-\alpha}}\right)$ as $\tau \to 0^+$. Applying the previous estimates (\ref{VOC:int:bound})-(\ref{f:bhv}) to the variation of constants formulation (\ref{Phi:Cap}), we get
\[
f(\eta) + Q(\eta - \eta_{inflex}) - C\tau^{\frac{\alpha}{2-\alpha}} \leq \Phi_\tau (\eta) \leq f(\eta) + Q(\eta - \eta_{inflex}) + C\tau^{\frac{\alpha}{2-\alpha}}.
\]
Applying the Mean Value Theorem to $\Phi_\tau$ on the interval $[\eta-\eta_{inflex},\eta]$, there exists $\tilde{\eta} \in (\eta-\eta_{inflex},\eta)$ such that
\[
%\begin{split}
f(\eta) + Q(\eta - \eta_{inflex}) - C\tau^{\frac{\alpha}{2-\alpha}} \leq \Phi_\tau (\eta-\eta_{inflex}) - \Phi_\tau'(\tilde{\eta})\eta_{inflex} %\\ &
\leq f(\eta) + Q(\eta - \eta_{inflex}) + C\tau^{\frac{\alpha}{2-\alpha}}.
%\end{split}
\]
Using the bound (\ref{Phi:bound:tau}), we obtain
\begin{equation}\label{est:parabola:1}
f(\eta) + Q(\eta - \eta_{inflex}) - C\tau^{\frac{\alpha}{2-\alpha}} \leq \Phi_\tau (\eta-\eta_{inflex}) \leq f(\eta) + Q(\eta - \eta_{inflex}) + C\tau^{\frac{\alpha}{2-\alpha}},
\end{equation}
for $\tau$ sufficiently small.

Before we get bounds on the term $Q(\eta-\eta_{inflex})$, we recall its definition:
\[
\begin{split}
Q(\eta -\eta_{inflex}) = & \, \tilde{h}\left( \Phi_\tau(\eta-\eta_{inflex}) + \phi_c\right) - h_c' \Phi_\tau(\eta-\eta_{inflex}) \\ &+ d_\alpha \int_{-\infty}^{0} \frac{(-\Phi_\tau'(z))}{(\eta-\eta_{inflex} -z)^\alpha} \, dz \, ,
\end{split}
\]
here the integral term is positive since $\Phi_\tau$ is initially decreasing. For large $\eta$ and a fixed $M>0$, we use Lemma~\ref{basic:ests} to get the estimate
\begin{equation}\label{bound:frac}
0 < \frac{C'}{(\eta-\eta_{inflex}+M)^\alpha}\leq d_\alpha \int_{-\infty}^{0} \frac{(-\Phi_\tau'(z))}{(\eta-\eta_{inflex} -z)^\alpha} \, dz \leq \frac{C}{(\eta-\eta_{inflex})^\alpha}\, .
\end{equation}
where $C= (-\Phi_\tau(0^+) + \phi_-)$ and $C' = (-\Phi_\tau(0^+) +\Phi_\tau(M))$ are positive. These upper and lower bounds can be combined with the $f(\eta)$ term but the constants here are not necessarily small, however, note that we are considering $\eta$ large.

For the term involving $\tilde{h}$, there are two possibilities depending on its sign. Observe that $\tilde{h}$ is not convex. In fact, $h''(\phi)=0$ at $\phi=0$ and $\phi_c<0<\phi_0$. If $\tilde{h}\left( \Phi_\tau(\eta-\eta_{inflex}) + \phi_c\right) - h_c' \Phi_\tau(\eta-\eta_{inflex})\geq 0$, then $Q$ remains positive, but, since we know that $\Phi_\tau(\eta)\rightarrow 0$ as $\eta\to\infty$, there exists a larger $\eta$ such that $\tilde{h}\left( \Phi_\tau(\eta-\eta_{inflex}) + \phi_c\right) - h_c' \Phi_\tau(\eta-\eta_{inflex})\leq 0$. In particular, for $ \eta-\eta_{inflex}\sim \eta_\tau$ we get to this region: applying the inequality
\[
|\phi_\tau(\eta + \xi_0) - \phi_0(\eta + \xi_0)| \lesssim \tau \, , \quad \mbox{for} \quad \eta \lesssim \eta_\tau \, ,
\]
from Theorem~\ref{mono:tau:small}, we derive that for sufficiently small $\tau$ and $\eta$ large enough (but with $\eta\lesssim\eta_\tau =O(\tau^{-1/(2-\alpha)})$),
\[
|\Phi_\tau(\eta)| \leq |\phi_\tau(\eta + \xi_0) - \phi_0(\eta + \xi_0)| + |\phi_0(\eta + \xi_0) - \phi_c| \lesssim \tau + \frac{1}{(\eta + \xi_0)^\alpha} \, .
\]
Thus, for $\tau$ sufficiently small, $\phi_\tau$ can be made arbitrarily close to $\phi_c$ with $\phi_ \tau$ decreasing.

For such $\eta$'s in a neighbourhood of $\eta_\tau$, we approximate this term by the quadratic expression
\begin{equation}\label{bound:h}
\tilde{h}\left( \Phi_\tau(\eta-\eta_{inflex}) + \phi_c\right) - h_c' \Phi_\tau(\eta-\eta_{inflex}) = \tilde{h}''(\tilde\phi) \Phi_\tau^2(\eta-\eta_{inflex}) \, ,
\end{equation}
then $Q$ can change sign and, in this case, $\tilde{h}''(\tilde\phi)<0$ (depends on $\eta$). Since we know that $\Phi_\tau$ is bounded then
\begin{equation}\label{bound:h''}
  -|\tilde{h}_{max}''| \leq \tilde{h}''(\tilde\phi) \leq -|\tilde{h}_{min}''|<0.
\end{equation}

Subsequently, using the previous estimates (\ref{bound:frac})-(\ref{bound:h''}) in (\ref{est:parabola:1}), we derive,
\[
%\begin{split}
F_2(\eta) - |\tilde{h}''_{max}| \Phi_\tau^2(\eta-\eta_{inflex})  - C\tau^{\frac{\alpha}{2-\alpha}} \leq \Phi_\tau (\eta-\eta_{inflex}) %\\ &
\leq F_1(\eta) - |\tilde{h}''_{min}| \Phi_\tau^2(\eta-\eta_{inflex})+ C\tau^{\frac{\alpha}{2-\alpha}},
%\end{split}
\]
where 
\[
0<\frac{1}{\eta^\alpha}\sim F_2(\eta) \leq f(\eta) + d_\alpha \int_{-\infty}^{0} \frac{(-\Phi_\tau'(z))}{(\eta-\eta_{inflex} -z)^\alpha} \, dz \leq F_1(\eta)\sim \frac{1}{\eta^\alpha}.
\]
Hence, we conclude the following inequalities, rewriting as $X=\Phi_\tau (\eta-\eta_{inflex})$ and considering, by continuity in $\eta$, $\eta \gtrsim \tau^{-1/(2-\alpha)}$ (but not too away from $\eta_\tau$) to yield $\eta^{-\alpha} \lesssim \tau^{\alpha/(2-\alpha)}$,
\[
-\overline{C_2} \, \tau^{\frac{\alpha}{2-\alpha}} - C_2 X^2 \leq X \leq \overline{C_1} \, \tau^{\frac{\alpha}{2-\alpha}} - C_1 X^2 
\]
for $C_1, C_2, \overline{C_1}$ and $\overline{C_2}$ positive constants. Hence, if we analyse both inequalities and if $\tau$ is sufficiently small, so that the term $\overline{C_2} \, \tau^{\frac{\alpha}{2-\alpha}}$ is small, then both parabolas have real roots and since $X$ is positive initially, we obtain the following inequalities,
\begin{equation}\label{great:estimate}
\frac{-1 + \sqrt{1- 4 \overline{C_2} \, \tau^{\frac{\alpha}{2-\alpha}} C_2 }}{2 C_2} \leq \Phi_\tau(\eta-\eta_{inflex}) \leq \frac{-1 + \sqrt{1+ 4 \overline{C_1} \, \tau^{\frac{\alpha}{2-\alpha}} C_1 }}{2 C_1} %\quad \mbox{for} \quad \eta\grtsim\eta_\eta \,.
\end{equation}
for $\eta$ in a neighbourhood of $\eta_\tau$. We can extend the argument to all $\eta>\eta_\tau$ by continuity of $\eta$, because (\ref{great:estimate}) guarantees that we can extend the previous estimates: notice that the negative root tends to zero as $\tau \to 0^+$ and the same happens for the positive root, implying that $\Phi_\tau(\eta)$ remains close to $0$.

Therefore, this allows to control the possible oscillatory behaviour of $\Phi_\tau$ and, consequently, of $\phi_\tau$ for $\tau$ sufficiently small, ensuring that $\phi_\tau$ remains in the region where $h= \tilde{h}$. 

\end{proof}

%..............................................................

For completeness we give the following result on the asymptotic behaviour of solutions as $\eta\to\infty$ for monotone solutions:
\begin{prop}
  If $\tau>0$ is sufficiently small and $\phi_\tau$ is decreasing then
  \[
 \lim_{\xi \to \infty}{|\phi_\tau(\xi) - \phi_c | \, \xi^\alpha} < +\infty.
 \]
\end{prop}

\begin{rem}\label{remark}
  In the proof of this proposition we take $\xi_0$, the shift where we split the non-local term, very large. Then we can derive equation (\ref{TWP:eta})-(\ref{lin:phic2}) similarly and use the implicit formulation (\ref{Phi:Cap}). We shall use the same notation, but now since we are assuming that $\xi_0$ is very large and that $\phi_\tau$ decreases, then the corresponding initial value for $\Phi_\tau(0^+)$ is very small.
  
 Now, in this case, by Lemma~\ref{basic:ests} (\ref{Q:upper:basic}), as long as $\Phi_\tau$ decreases and stays positive ($\phi_\tau>\phi_c$) but close to $0$ (so $\tilde{h}(\phi_\tau)<0$, but small) then $Q$ is positive:
\begin{equation}\label{Q:positive}
Q(\eta) \geq -\frac{C_h}{2}  (\Phi_\tau(0^+))^2 +
d_\alpha\frac{ \Phi_\tau(-M)  -\Phi_\tau(0^+) }{(\eta+M)^\alpha}.
 \end{equation}
 The first term is dominated by the second if we take $M\geq \eta$ and large enough, but such that $M^\alpha\ll  (\Phi_\tau(0^+))^{-2}$, since we also have that $0<-\Phi_\tau(0^+) + \Phi_\tau(-M)=-\phi_\tau(\xi_0)+\phi_\tau(\xi_0-M)\leq \phi_- - \phi_c$. 
 \end{rem}

\begin{proof}
%Similar to \cite{DC}... 
We fix $\tau$ sufficiently small such that Lemma~\ref{v:bhv:tau:small} of the Appendix~\ref{v:monotone} holds, in particular $0<v(\eta)<1$ and $v'(\eta)<0$ for all $\eta>0$.

As said above in Remark~\ref{remark}, we use the implicit formulation (\ref{Phi:Cap}) (with the same notation as before, for simplicity). Since here we assume that $\phi_\tau'<0$ we take $\xi_0\gg 1$ (and possibly larger than $\xi_\tau$ of Theorem~\ref{mono:tau:small} above). The assumption implies also that $\Phi_\tau>0$ and that $\Phi_\tau'<0$. Then we follow an argument similar to that in \cite{DC}. Here we can take, if necessary, $\Phi_\tau(0^+)$ as small as we want, by assumption. This means that we can choose the shift $\xi_0$ {\it a posteriori} to get the result. For simplicity of notation we denote
\[
0<\delta:=\Phi_\tau(0^+).
\]

A lower bound is obtained by applying that $Q(\eta)>0$. Notice that this is possible for very small $\delta$ and large $\eta$ so that $M$ in (\ref{Q:positive}) can be taken $1\ll \eta \leq M$ and $M^\alpha\ll \delta^{-2}$. Then:
\begin{equation*}%\label{asymptotic:left}
  \Phi_\tau(\eta) \geq \delta v(\eta) + \frac{\tau \Phi_\tau'(0^+)}{h_c'} v'(\eta)
\end{equation*}
and this is valid for very large $\eta$ with $\eta \ll \delta^{-2/\alpha}$.
The second term is negative, but for all $\eta\geq \eta'$ such that
\[
\eta'\gg \frac{\tau}{\delta} \frac{\Phi_\tau'(0^+)}{h_c'}
\]
then there exists $C>0$ such that
\begin{equation}\label{Phi:asympt:lower}
  \Phi_\tau(\eta) \geq C\eta^{-\alpha}
  \quad \mbox{for} \quad 1\ll\eta\ll \delta^{-2/\alpha},
  \end{equation}
with $\delta\ll 1$ sufficiently small. Here we are using the behaviour of $v$ and its derivatives given in Lemma~\ref{v:bhv:tau:small} of the Appendix~\ref{v:monotone}. Notice that $\eta'$ is not necessarily large and both conditions on $\eta$ are compatible, because $\delta=\Phi_{\tau}(0^+)\geq C|\Phi_\tau'(0^+)|$ and then for $\tau$ small enough $|\Phi_\tau'(0^+)|\tau/|h_c'| \ll \delta^{2/\alpha-1}$. 
Let us obtain an upper bound. Since the second term in (\ref{Phi:Cap}) is negative, we have
\[
\Phi_\tau(\eta) \leq %\Phi_\tau(0^+)v(\eta) + \frac{1}{h_c'} \int_{0}^{\eta}{v'(r) Q(\eta-r)\, dr}=
\delta v(\eta) + \frac{1}{h_c'} \int_{0}^{\eta}{v'(r) Q(\eta-r)\, dr}\,.
\]
Then, we can apply the estimates (\ref{Q:upper:basic})-(\ref{Q:lower:basic}) of the Lemma~\ref{basic:ests} on $Q$. But we may split the integral into several parts. Before that, let us introduce the following notation:
\[
I_1:= \frac{1}{h_c'}\int_{0}^{\eta} v'(r) \left(\tilde{h}(\phi_\tau(\eta-r+\xi_0)) - h_c' \Phi_\tau(\eta-r)\right)\,dr %\leq 0
\]
and
\[
I_2 := \frac{d_\alpha}{h_c'}\int_{0}^{\eta} v'(r)  \int_{-\infty}^{0}\frac{(-\Phi'_\tau(z))}{(\eta-r-z)^\alpha} \, dz\, dr\,. %\geq 0.
\]
We observe that, by hypothesis, we can start for a $\xi_0$ large enough such that $\tilde{h}''(\phi(\xi_0))<0$ so that $I_1\leq 0$ is non-positive. The assumption on $\phi_\tau$ being decreasing up to $\xi_0$ also implies that $I_2\geq 0$. Then we have that
\begin{equation}\label{Phi:I2}
\Phi_\tau(\eta)\leq \delta v(\eta) + I_2.
\end{equation}

We get an upper bound for $I_2$ using (\ref{eta:large}) and (\ref{eta:small}) of Lemma~\ref{basic:ests} and splitting the interval of integration at some $R>0$:
\begin{equation}\label{I2:bound}
\begin{split}
  I_2 &\leq \frac{d_\alpha}{h_c'}\left(
  \int_{0}^{R}v'(r) \, \left(\frac{C}{(\eta-r)^{\alpha+1}} + \frac{C'\left|\Phi_\tau(\eta-r)\right|}{(\eta-r)^\alpha}\right) \,dr
  + \int_{R}^{\eta}v'(r) \, \left(\frac{C}{1+ (\eta-r)^\alpha}\right) \,dr
  \right) \\
  &\leq C_1\frac{R^2}{2\tau}\left(\frac{C}{(\eta-R)^{\alpha+1}}
  + \frac{C' \left|\Phi_\tau(\eta-R)\right|}{(\eta-R)^\alpha}\right)
  + \frac{C_2}{\alpha}\left(\frac{1}{R^\alpha} - \frac{1}{\eta^\alpha}\right).
\end{split}
\end{equation}
%Besides, the behaviour of $v'$, one has to also consider that the function $|\Phi_\tau(\eta)|$ is monotone decreasing.

Now, we take $R$ depending on $\eta$, once that $\tau$ and $\delta$ (taken as small as necessary) are fixed:
%\[
%R(\eta)= \frac{\gamma}{\sqrt{|\Phi_\tau(\eta)|}}, \quad \text{for} \quad \gamma\in (0,1) 
%\] 
%and 
\[
R(\eta)=(\sigma \eta)^{\alpha/2}, \quad \text{for} \quad \sigma\in (0,1).
\] 
We take $\sigma$ such that $R(\eta) \leq 1$. In particular for each $\eta$ we have
\begin{equation}\label{first:cond:gamma:sigma}
  \sigma < \eta^{-1}.
\end{equation}
Then we can say that
%This enables us to get the bound
\[
%\frac{1}{R^\alpha} \leq \left(\frac{1}{R^\alpha}\right)^{2/\alpha}, \quad
\frac{1}{R^\alpha} \leq \left(\frac{1}{R^\alpha}\right)^{2/\alpha}=\frac{1}{R^2}.
\] 
The previous estimate and inequality (\ref{I2:bound}) applied to (\ref{Phi:I2}) yield, for some order one constants,
%\begin{equation}
%\begin{split}
%  \Phi_\tau(\eta) \leq& \delta v(\eta) +
%  %\frac{C_1}{2\tau} (\Phi_\tau(\eta-R))^2 R^2 + \frac{C_2}{\alpha} \delta^2 \left(\frac{1}{R^\alpha}\right)^{2/\alpha} - \frac{C_2}{\alpha} \delta^2 \frac{1}{\eta^\alpha} \\
%&+ C_1 \frac{R^2}{2\tau} \frac{C}{(\eta-R)^{\alpha+1}} + C_2 \frac{R^2}{2\tau} \frac{C'|\Phi_\tau(\eta- R)|}{(\eta-R)^{\alpha}} \\
%&+ \frac{C_3}{\alpha}\left(\frac{1}{R^\alpha}\right)^{2/\alpha} - \frac{C_4}{\alpha}\frac{1}{\eta^\alpha}.
%\end{split}
%\label{asymptotic:right:1}
%\end{equation}
\begin{equation}
  \Phi_\tau(\eta) \leq \delta v(\eta) 
  +  \frac{R^2}{2\tau} \frac{C_1}{(\eta-R)^{\alpha+1}} + \frac{R^2}{2\tau}
  \frac{C_2|\Phi_\tau(\eta- R)|}{(\eta-R)^{\alpha}}
  + \frac{C_3}{\alpha}\frac{1}{R^2} - \frac{C_4}{\alpha}\frac{1}{\eta^\alpha}.
\label{asymptotic:right:1}
\end{equation}

Therefore, we can deduce the following upper bound from (\ref{asymptotic:right:1}) and using (\ref{first:cond:gamma:sigma}), where the worst case scenario is% are
\[
  \left( 1  - \frac{C_2}{2\tau}\sigma^\alpha \right) \Phi_\tau(\eta) \leq
  \frac{C_5}{\eta^\alpha},
\]
for some $C_5>0$ of order one, and, therefore, it is sufficient to take $\sigma$ small enough such that
\begin{equation}\label{cond:gamma:sigma}
  \sigma < \tau^{1/\alpha}\,. 
\end{equation}
Since, we can choose $\delta=\Phi_\tau(0^+)$ arbitrarily small once $\tau$ is fixed and we have (\ref{Phi:asympt:lower}), we can conclude that for $1\ll\eta\ll \delta^{-2/\alpha}$ (large enough but in this range), we can take $\sigma$ satisfying (\ref{first:cond:gamma:sigma}) and (\ref{cond:gamma:sigma}), then, there exists $C>0$ such that
\begin{equation}\label{asymptotic:right}
  %\frac{A_1}{\eta^\alpha}  \leq
  \Phi_\tau(\eta) \leq \frac{C}{\eta^\alpha} \quad \mbox{for} \quad 1\ll\eta\ll \delta^{-2/\alpha}. %\quad \text{as} \quad \eta \to \infty.
\end{equation}
Finally, taking the limit $\xi_0\to\infty$ implies that $\delta\to 0$ so we can increase the range of $\eta$ in the limit and we obtain the result.
\end{proof}

 %>>>>>>>>>>>>>>>>>>>>>>>>>>>>>>>>>>>>>>>>>>>>

\section{Numerical Computations}\label{sec:numerics}

In this section we show numerical simulations that confirm the existence of solutions of (\ref{TWP})-(\ref{far-fieldR}) for a value of $\tau>0$ under the assumptions (\ref{lin:ass}) and (\ref{nec:cond0}). Namely and 
for definiteness, in this section we consider
\begin{equation}\label{phis:num}
\phi_-=1\,, \quad \phi_+=-0.6 \quad (\phi_c=-0.4)\,,
\end{equation}
such that all conditions for a non-classical shock wave are satisfied.

First, we show numerical computations of (\ref{TWP}) performed with the method described and analysed in \cite{C}. Rewriting the travelling wave problem (\ref{TWP}) as a system making the change $\psi = \phi'$ gives
\begin{equation}\label{TWP:system}
\begin{cases}
\phi' = \psi, \\
\tau \psi' = h(\phi) - d_\alpha \int_{-\infty}^{\xi} \frac{\psi(y)}{(\xi-y)^\alpha} \, dy.
\end{cases}
\end{equation}
The singularity of the integral term $\DD^\alpha[\phi]$ is removed by using integration by parts and taking into account the regularity and far-field behaviour of $\phi$, which implies that
\begin{equation}\label{D:alpha:regular:kernel}
%\begin{split}
\int_{-\infty}^{\xi} \frac{\psi(y)}{(\xi - y)^\alpha} \, dy = \frac{1}{1-\alpha} \int_{-\infty}^{\xi} \psi'(y) (\xi-y)^{1-\alpha} \, dy. %\\
%+ \frac{1}{1-\alpha} \left( -\lim_{y \to \xi} \psi(y) (\xi - y)^{1-\alpha} + \lim_{y \to -\infty} \psi(y) (\xi - y)^{1-\alpha} \right),
%\end{split}
\end{equation}
The initial value problem (\ref{TWP:system})-(\ref{D:alpha:regular:kernel}) is solved by a scheme using the Heun's method (see e.g. \cite{Ascher}). For more information on the numerical scheme see \cite{C}.

Next, we proceed by shooting with respect to $\tau$ as follows. First, we identify two values of $\tau$, $\tau_c$ and $\tau_u$ such that $\tau_c\in \Sigma_c$ and $\tau_u\in \Sigma_u$. This is done by integrating the equations for a long enough interval, typically of length $500$, then if the solution approaches the value $\phi_c=-0.4$ in the tail, we assume that the corresponding $\tau$ is in $\Sigma_c$. If the solution decays to negative values beyond say $-10$, then we assume that the corresponding $\tau$ is in $\Sigma_u$. This operation allows to choose initial values for $\tau_c$ and $\tau_u$. Then, we start an iterative process, which consists of computing the solution for $\tau_m = (\tau_u+\tau_c)/2$, and apply the same criteria to either set $\tau_m=\tau_c$ or $\tau_m=\tau_u$. We repeat this process as long as $|\tau_c-\tau_m|<10e^{-15}$.

Figure~\ref{TWa09} shows solutions for $\alpha=0.9$, in this case the iteration stops at the value $\tau\approx 2.80018$.
\begin{figure}[H]
  \begin{center}
    \includegraphics[width=8cm,height=7cm]{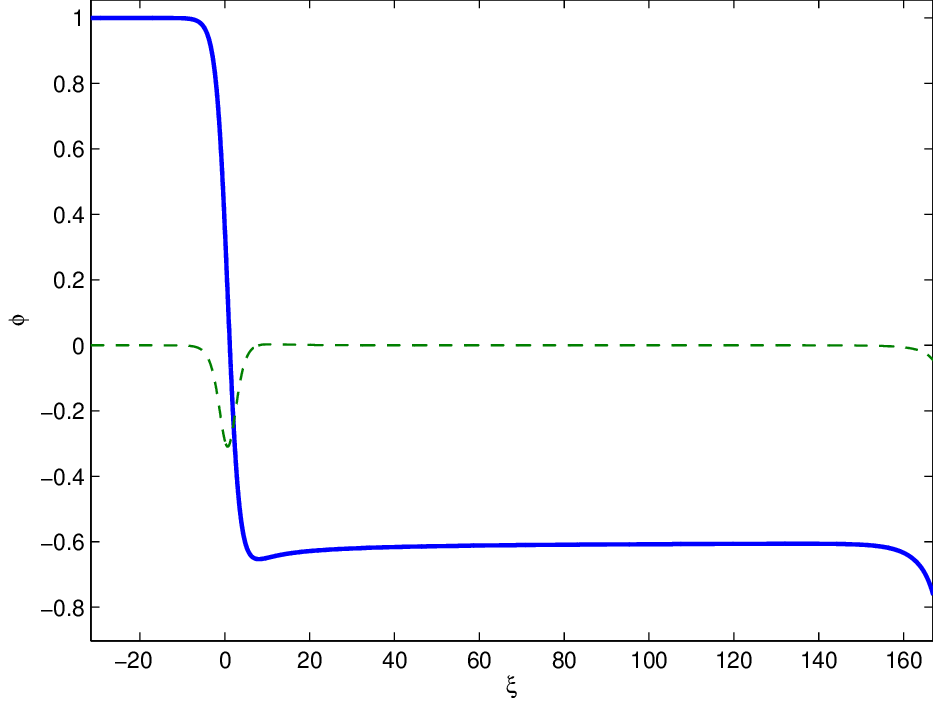}
  \end{center}
  \caption{$\alpha=0.9$ and $\tau\approx 2.80018$}
\label{TWa09}
\end{figure}

Figure~\ref{TWa05} shows solutions for $\alpha=0.5$, in this case the iteration stops at the value $\tau=72.821821443764975$.

\begin{figure}[H]
  \begin{center}
    \includegraphics[width=8cm,height=7cm]{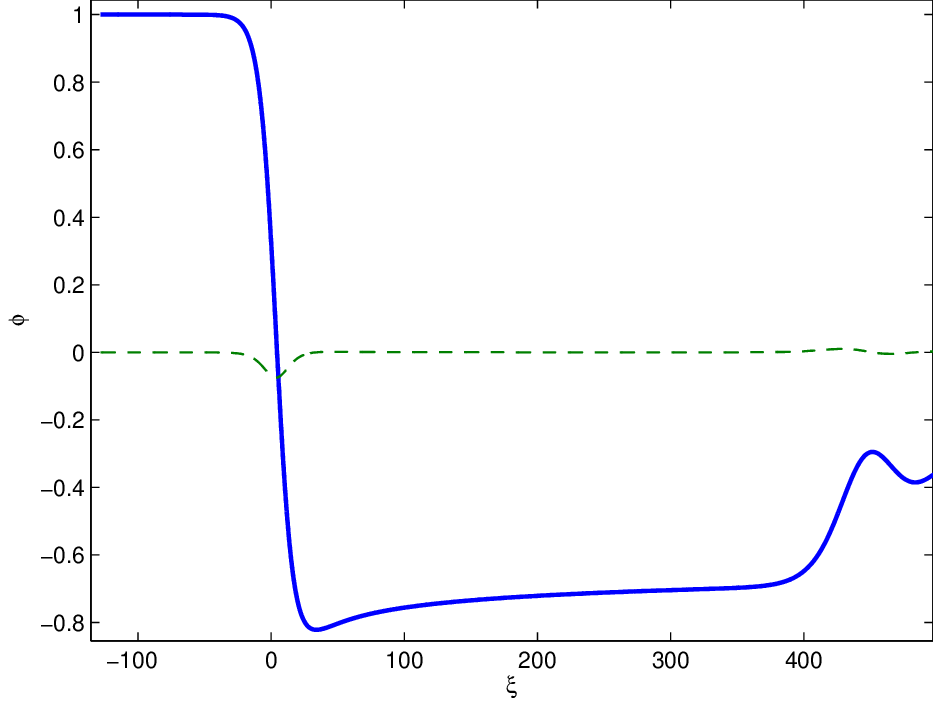}
  \end{center}
  \caption{$\alpha=0.5$ and $\tau\approx 72.82182$}%full number 72.821821443764975
\label{TWa05}
\end{figure}

%---------------------------------------------------

%\paragraph{Acknowledgements}
%F. Achleitner was supported by the Austrian Science Fund (FWF) via the FWF-funded SFB \# F65. C. M. Cuesta and X. Diez-Izagirre thank the support by the grant PID2021-126813NB-I00 funded by MICIU/AEI/10.13039/501100011033 and by ``ERDF A way of making Europe'' and the financial support of the Basque Government through the Research Group IT1615-22. X. Diez-Izagirre was also supported by the Basque Government through the pre-doctoral grant PRE-2018-2-0013.

%%%%%%%%%%% APPENDIX %%%%%%%%%%%%%%%%%%%%%%
\section*{Appendix~}\label{appendix}

\appendix

\section{Continuous dependence on $\tau$}\label{appendix:B}
In this section, we prove the continuous dependence on the parameter $\tau$ using the general theory for functional differential equations, e.g. see
\cite[\S 2]{HVL} and~\cite{SeHa00}. After rewriting the problem as a functional differential equation, we check that the necessary hypotheses are fulfilled in order to apply the auxiliary lemmas and the continuous dependence result from \cite[\S 2]{HVL}.
%NOTE: Continuity in intervals $(-\infty,\xi_\tau]$ with $\xi_\tau$ very negative holds as Lemma 10 of previous Paper. Then on bounded intervals using HVL. 

First, we rewrite \eqref{TWP} as a system of first-order delay functional differential equations %with infinite delay 
\[ 
\begin{cases}
  %\displaystyle{
    \phi' & = \psi,
  %}
  \\
%\displaystyle{
\psi' & = \frac{1}{\tau} h(\phi) -\frac{d_\alpha}{\tau}
\displaystyle{\int_{-\infty}^{\xi}\frac{\psi(y)}{(\xi-y)^\alpha} \,dy.
}
\end{cases}
\]
In order to study the continuous dependence of solutions on $\tau$, we add $\tau$ as an independent variable. 
However, it is easier to consider instead of $\tau>0$ its inverse $\nu :=1/\tau$ such that the augmented system of first-order differential equations reads
\begin{equation} \label{RFDE:infinite_delay}
\begin{cases} 
\phi' &=\psi, \\
\psi' &=\nu h(\phi) -\nu\ d_\alpha \int_{-\infty}^{\xi}\frac{\psi(y)}{(\xi-y)^\alpha} \,dy, \\
\nu'  &=0.
\end{cases}
\end{equation}
In order to frame this system as one of functional differential equations with finite delay, we split the integral term at $-r$ for some $r>0$:
\[
\int_{-\infty}^\xi\frac{\psi(y)}{(\xi-y)^\alpha} \, dy=\int_{-\infty}^{-r} \frac{\psi(y)}{(\xi-y)^\alpha}\, dy+ \int_{-r}^{\xi}\frac{\psi(y)}{(\xi-y)^\alpha}\, dy.
\]
The first one is treated as a known function of $\xi$ as long as $\xi \geq -r$, while the second denotes the fractional derivative from $-r$ and one needs the initial data to be given on $(-r,0)$. %then for $s=-r$ this implies that $\xi\leq r$. So the R should contain the integral term from -infty to -r and the fractional derivative is ok from -r, since this is initially the data given on (-r,0) and R is treated as a given function.
Without loss of generality, we consider the function $\phi(\xi - r)$ to be given for $\xi\in(-\infty,0)$ and split the fractional derivative in two parts as before and applying the change of variable $y'=y+r$ we get:
\begin{equation}\label{D:alpha:split:r}
\DD^\alpha[\phi] 
= d_\alpha \int_{-\infty}^{0}\frac{\phi'(y'-r)}{(\xi + r - y')^\alpha} \,dy' +
d_\alpha \int_{0}^{\xi+r}\frac{\phi'(y'-r)}{(\xi+r-y')^\alpha} \,dy'
= W(\xi) + \DD_{-r}^\alpha[\phi](\xi),
\end{equation}
where $W(\xi) := d_\alpha  \int_{-\infty}^{0}\frac{\phi'(y'-r)}{(\xi + r - y')^\alpha} \,dy' $ is a given function.
%To write $\DD_{\xi-r}^\alpha[\phi](\xi)$ as a term with finite delay, we choose $r>0$ such that, for all $\xi\leq r$,
%\[
%\DD_{\xi-r}^\alpha[\phi](\xi) 
%=d_\alpha \int_{\xi-r}^{\xi}\frac{\phi'(y)}{(\xi-y)^\alpha} \,dy 
%=d_\alpha \int_{-r}^{0}\frac{\phi'(\xi+s)}{|s|^\alpha} \,ds .
%\]
Using (\ref{D:alpha:split:r}), system~\eqref{RFDE:infinite_delay} can be written as
\begin{equation} \label{RFDE}
\begin{cases}      
\phi' &=\psi, \\
\psi' &=\nu h(\phi)
-\nu\ d_\alpha \int_{0}^{\xi+r}\frac{\psi(y'-r)}{(\xi+r-y')^\alpha} \,dy'
- \nu W(\xi), \\
\nu'  &=0,
\end{cases}
\end{equation}
for all $\xi\geq -r$.
The first and the third equations are ordinary differential equations, whereas the second one is an integro-differential equation with finite delay in the integral term. 
Following the notation of \cite[\S 2]{HVL}, equation (\ref{RFDE}) is a functional differential equation of the form
\begin{equation}\label{gen:delay:eq}
  \dot{x} =F(\xi,x_\xi)
  %\quad x=(\phi,\psi,\nu)
\end{equation}
such that $x_\xi(\theta) = x(\xi+\theta)$, for $-r\leq \theta \leq 0$.

In our particular case, $x=(\phi,\psi,\nu)$, and $F=(F_1,F_2,F_3)$ is identified as
\begin{align}
F_1\left(\xi, (\phi,\psi,\nu)\right) &= \psi, \label{our:F1}\\
F_2\left(\xi, (\phi,\psi,\nu)\right) &= \nu h(\phi)
-\nu\ d_\alpha \int_{0}^{\xi+r}\frac{\psi(y'-r)}{(\xi+r-y')^\alpha} \,dy' - \nu W(\xi), \label{our:F2}\\
F_3\left(\xi, (\phi,\psi,\nu)\right) &= 0.\label{our:F3}
\end{align}
Moreover, we consider the operator $F$ as $F: D\to\R^3$ with domain $D\subseteq \R\times C([-r,0],\R^3)$. 
We recall $C([-r,0],\R^3)$ is a Banach space with norm $\|\varphi\|_\infty = \sup_{-r\leq s \leq 0}{|\varphi(s)|}$ for functions $\varphi \in C([-r,0],\R^3)$. 
Finally, we consider the delay functional differential equation (\ref{RFDE}) for a starting time $\sigma=0$.
As it is mentioned previously, we only need the vector $(\phi,\psi,\nu)(\theta)$ for all $-r\leq\theta\leq 0$ as initial datum, the history of $\psi(\xi)$ for $\xi<-r$ is incorporated into $W(\xi)$. Besides, notice that the application of \cite[Lemma~3]{ACH} gives us the integrability of $\psi$ on $(-\infty,\xi_0)$ for $\xi_0< \xi_{exist}$ defined in Lemma~\ref{good:sign}. %Since $\xi=0$ is just an arbitrary splitting point, 
Moreover, this argument proves the finiteness of $W(\xi)$ for all $\xi>0$.

%Due to the continuity of $F$, solving the system \eqref{RFDE} and solving the following integral equation are equivalent,
%\begin{align*}
%\phi(\xi) &= \phi(0) + \displaystyle\int_{0}^{\xi}{\psi(\eta) \d[\eta]}, \\
%\psi(\xi) &= \psi(0) + \nu \int_{0}^{\xi}{h(\phi(\eta))\d[\eta]} - \nu\ d_\alpha \int_{0}^{\xi}\int_{-r}^{0}{\frac{\chi_{(-\eta,0)}(s)}{|s|^\alpha} \psi(\eta + s)\d[s] \d[\eta]} - \nu \int_{0}^{\xi}{R(\eta)\d[\eta]},\\
%\nu(\xi) &= \nu(0).
%\end{align*}

%........

We shall show below that we can apply the following theorem to (\ref{RFDE}), rewritten as in (\ref{gen:delay:eq}) with $F$ given by (\ref{our:F1})-(\ref{our:F3}).
\begin{thm}[Continuous dependence ({\cite[Theorem 2.2]{HVL}})]\label{CDTh}
Suppose $\Omega \subseteq \R\times  C([-r,0],\R^3)$ is open, $(\sigma^{0}, \gamma^{0})\in \Omega$, $F^{0} \in C(\Omega, \R^n)$, and $x^{0}$ is a solution of the problem (\ref{gen:delay:eq}) ($F^{0}$) with initial condition $(\sigma^{0},\gamma^{0})$ which exists and is unique on $[\sigma^{0}-r,b]$. 
Let $V^{0} \subseteq \Omega$ be the compact set defined by 
\[
V^{0}= \left\{(\xi,x_\xi^{0}): \xi \in [\sigma^{0},b]\right\}
\]
and let $U^{0}$ be a neighbourhood of $V^{0}$ on which $F^{0}$ is bounded. 
If $(\sigma^{k}, \gamma^{k}, F^{k})$, $k=1,2,\dots$ satisfies $\sigma^{k} \to \sigma^{0}$, $\gamma^{k} \to \gamma^{0}$ and $\left|F^{k} - F^{0}\right|_{U^{0}} \to 0$ as $k\to \infty$, then there is a $k^*$ such that the problem (\ref{gen:delay:eq}) ($F^{k}$) for $k\geq k^*$ is such that each solution $x^{k} = x^{k}(\sigma^{k}, \gamma^{k}, F^{k})$ with initial condition $(\sigma^{k}, \gamma^{k})$ exists on $[\sigma^{k}-r,b]$ and $x^{k} \to x^{0}$ uniformly on $[\sigma^{k}-r,b]$. 
Since all $x^{k}$ may not be defined on $[\sigma^{k}-r,b]$, by $x^{k} \to x^{0}$ uniformly on $[\sigma^{k}-r,b]$, we mean that for any $\varepsilon>0$, there is a $k^*(\varepsilon)$ such that $x^{k}$, $k \geq k^*(\varepsilon)$, is defined on $[\sigma^{0}-r+\varepsilon,b]$, and $x^{k} \to x^{0}$ uniformly on $[\sigma^{0}-r+\varepsilon,b]$.
\end{thm}

Next we proceed to check that our $F: \R\times C([-r,0],\R^3) \to \R^3$ in (\ref{our:F1})-(\ref{our:F3}) is continuous in both variables. This is obvious for $F_1$ and $F_3$, since $F_1$ is the projection of the second component and $F_3$ is just the zero constant function. In the case of $F_2$, the first term, $h(\phi(\xi))$, is continuous because it is a composition of continuous functions. The integral term is continuous since it maps $C_b(\R)$ into $C_b(\R)$ and, finally, the last term, $W(\xi)$ is continuous because of the regularity of $\phi$ in $\xi \in (-\infty,0)$ and finiteness is obtained as is explained above.

To study the existence of solutions for (\ref{RFDE}) starting at $\sigma=0$, we only need to prescribe the values for $(\phi,\psi,\nu)(\xi)$ at $-r<\xi<0$, since the history of $\psi(\xi)$ for $\xi<0$ is incorporated in $W(\xi)$ which we treat as a given function.
To study the continuous dependence of solutions on $\tau$ (or equivalently $\nu$) in a neighbourhood of $\tau_0$, we consider the following initial data
\[
 \sigma^k \equiv 0, \quad
 \phi^k \to \phi^0, \quad
  \psi^k \to \psi^0, \quad
 \nu^k \to \tfrac{1}{\tau_0} \quad \mbox{as} \quad k\to \infty.
\]
Note that the delicate point is that changing $\tau$ (or $\nu$) influences the profile $\phi(\xi)$, $\psi(\xi) =\phi'(\xi)$, for $\xi<0$, hence, also $F$ through its dependence on $W(\xi)$. Therefore, we have to use continuous dependence of 'local' solutions with respect to $\tau$, to justify the assumptions on $\phi^k(\xi)$, $\psi^k(\xi)$ and $F^k$. But this results follows from \cite[Lemma~2]{ACH}. Let us apply this here.

Considering the behaviour of the travelling wave solution and its derivative at $-\infty$ we know that for all $\nu_k =1/\tau_k>0$
\[
\lim_{\xi \to -\infty} \phi^k(\xi) = \lim_{\xi \to -\infty} \phi_{\tau_k}(\xi)
= \phi_- \quad \mbox{and} \quad \lim_{\xi \to -\infty} \psi^k(\xi)
=\lim_{\xi \to -\infty} \phi'_{\tau_k}(\xi) = 0.
\]
Therefore, by \cite[Lemma~2]{ACH} and for all fixed $k>0$ natural number, there exists some $\xi_k= \log(1/k)/\lambda_k$ such that
\[
|\phi^k(\xi) - \phi_-|<\frac{1}{k},
\quad |\psi^k(\xi)|<\frac{1}{k}, \ \ \forall \xi < \xi_k. 
\]
Since it is known that
\[
\lim_{\xi \to -\infty} \phi^0(\xi) = \lim_{\xi \to -\infty} \phi_{\tau_0}(\xi)
= \phi_- \quad \text{and} \quad \lim_{\xi \to -\infty} \psi^0(\xi)
=\lim_{\xi \to -\infty} \phi'_{\tau_0}(\xi) = 0,
\]
then by the triangle inequality we get that 
\[
|\phi^k(\xi) - \phi^0(\xi)|<\frac{2}{k}, \quad |\psi^k(\xi)- \psi^0(\xi)|<\frac{2}{k}, \ \ \forall \xi < \xi_k. 
\]
Now, for each fixed $1/k>0$, we can apply continuous dependence on $\tau$ in the interval $[\xi_k, 0]$ taking as initial condition an arbitrary sequence of $\nu_k = 1/\tau_k$ that converges to $1/\tau_0$ as $k \to \infty$, $\sigma^k = \xi_k$, $F^k =F$, $\phi^k = \phi_{\tau_k}$ and $\psi^k= \phi'_{\tau_k}$. Therefore, by the continuous dependence result we yield that 
\[
\begin{split}
\forall \e_k=\frac{1}{k}>0, \, \exists k_0 > 0,
\, k>k_0: \ &|\phi^k(\xi) - \phi^0(\xi)|<\frac{1}{k} \\
%\quad 
\mbox{and}\quad &|\psi^k(\xi)-\psi^0(\xi)|<\frac{1}{k},
\quad \forall \xi \in [\xi_k, 0].
\end{split}
\]
Note that since $\tau_k \to \tau_0$ as $k\to \infty$ then for all $\delta_k>0$ small, there exists some $k_0 >0$ such that for all $k>k_0$, $\tau_k \in (\tau_0- \delta_k , \tau_0 + \delta_k)$. If we define a new subsequence taking the values $\nu_k$ for $k>k_0$ and rename this subsequence again as $\{\nu_k\}_{k\in \N}$, therefore, for this new sequence and taking $\sigma^k\equiv 0$ we conclude that
\[
\phi^k(\xi) \to \phi^0(\xi), \quad \psi^k(\xi) \to \psi^0(\xi),
\ \ \forall \xi \leq 0,
\]
which is sufficient to apply Theorem~\ref{CDTh} of continuous dependence on $\tau$ for the system (\ref{RFDE}) in an arbitrary bounded interval $[0,b]$ on the interval of existence.%for any $0<b<\xi_{exist}$.

%HEREHERE... Write it as a corolary... to apply directly

%------------------------------------------

\section{The characteristic equations of the linearised problems}
\label{appendix:roots}
Let us recall some results about the zeros of the functions
\begin{equation}\label{gen:left}
\tau z^2 + b z^\alpha - a \quad \mbox{for} \quad a\,,b>0\,, \quad \alpha\in(0,1),
\end{equation}
(compare with \ref{pol:left}) and
\begin{equation}\label{gen:right}
\tau z^2 + b z^\alpha + a \quad \mbox{for} \quad a\,,b>0\,, \quad \alpha\in(0,1),
\end{equation}
we can give the following result:
\begin{lem}\label{roots} 
For $\alpha \in (0,1)$, consider the principal branch of 
$z^\alpha$ ($-\pi<\mbox{arg}(z)<\pi$). Then (\ref{gen:left}) has exactly one positive real root and two complex conjugate roots with negative real part, and (\ref{gen:right}) has exactly two complex conjugate roots with negative real part on the principal branch of $z^\alpha$.
\end{lem}
The statement about \eqref{gen:left} and \eqref{gen:right} are proven in \cite{ACH} and \cite{BK}, respectively, using variants of Rouche's theorem.

For later use, we give the expansion of the zeros of
(\ref{gen:right}) provided that $a$ and $b$ are of order $1$ as $\tau \to 0^+$ (see \cite{ACH}):
\begin{equation}\label{complex:tau:small:right}
z= b^{\frac{1}{\alpha-2}} e^{\pm i\pi\frac{1}{\alpha-2}}\frac{1}{\tau^{\frac{1}{2-\alpha}}}
 -\frac{a}{2b^{\frac{1}{\alpha-2}} e^{\pm i\pi\frac{1}{\alpha-2}}  
 + b^{\frac{\alpha-1}{\alpha-2}} \alpha e^{\pm i\pi\frac{\alpha-1}{\alpha-2}}}\frac{1}{\tau^{\frac{1-\alpha}{2-\alpha}}} + O\left(\frac{1}{\tau^{\frac{1-2\alpha}{2-\alpha}}}\right) %\quad\mbox{as}\quad \tau\to 0^+\,.
\end{equation}
as $\tau \to 0^+$.

\section{The linearised equation}\label{appendix:v}
In this appendix we consider the linear inhomogeneous equation:
\begin{equation}\label{lin:eq}
\tau \psi''+\DD^\alpha_{0}[ \psi ]+ a \psi  
= Q(\eta) \,,
\quad\mbox{with}\quad a>0,
\end{equation}
(here $'= \frac{d}{d\eta}$) with initial conditions 
\begin{equation}\label{ics}
\psi(0^+)=C_0\,, \quad \psi'(0^+)=C_1\,.
\end{equation}
We recall the derivation of a solution via the Laplace transform see e.g. \cite{BK}. 
%\cite{GM2}, we follow the latter.
Applying the Laplace transform $\LL$ to (\ref{lin:eq})--(\ref{ics}) yields 
\begin{equation}\label{laplace:transform}
\LL(\psi)(s) =\frac{1}{\tau s^2 + s^\alpha+a} \left(\LL(Q)(s) 
+ (\tau s + s^{\alpha-1}) \psi(0^+) + \tau \psi'(0^+)\right)\,,
\end{equation}
where $\LL(f)(s) =\int_{0}^\infty e^{-s\eta}f(\eta) d\eta$. 
Using $\LL(f\ast g)(s)= \LL(f)(s) \, \LL(g)(s)$, we deduce
\[
\begin{split}
\psi =\psi(0^+)\LL^{-1}\left(\frac{\tau s + s^{\alpha-1}}{\tau s^2 + s^\alpha + a}\right) 
&+\tau\psi'(0^+) \LL^{-1}\left(\frac{1}{\tau s^2 + s^\alpha + a}\right) \\
&+ \LL^{-1}\left(\frac{1}{\tau s^2 + s^\alpha + a}\right)\ast Q \,.
\end{split}
\]
Define
\begin{equation}\label{v:coeff}
v(\eta) := \LL^{-1}\left(\frac{\tau s + s^{\alpha-1}}{\tau s^2 + s^\alpha + a} \right) (\eta) \quad 
\mbox{and} \quad \tilde{v}(s) :=\frac{\tau s + s^{\alpha-1}}{\tau s^2 + s^\alpha + a}\, .
\end{equation}
Observing that $\lim_{\eta\to 0^+} v(\eta) = \lim_{s\to \infty} s\tilde{v}(s)=1$ and
\[
\frac{1}{\tau s^2 + s^\alpha + a}
= \frac{1}{a}(1-s\tilde{v}(s)) 
= -\frac{1}{a} \big(s\LL(v)(s) -v(0^+) \big) 
\]
implies 
\[
\LL^{-1}\left(\frac{1}{\tau s^2 + s^\alpha + a}\right)(\eta)= -\frac{1}{a} v'(\eta)\,.
\] 
Consequently, 
\begin{equation}\label{v:at:zero}
\lim_{\eta\to 0^+} v'(\eta)=0\,.
\end{equation}
Writing the expression for $\psi$ in terms of $v$ reads
\begin{equation}\label{var:consts1}
\psi(\eta) =\psi(0^+) v(\eta)-\frac{\tau}{a} \psi'(0^+) v'(\eta)
-\frac{1}{a}\int_0^\eta v'(y) Q(\eta-y)\, \,dy\,.
\end{equation}

For $a>0$, let us sketch the computation of $v(\eta)$. We recall that since this is the inverse 
Laplace transform of $\tilde{v}(s)$, we have to compute:
\begin{equation}\label{inv:L:vtilde}
v(\eta)= \frac{1}{2\pi i} \int_{Br}e^{s\eta} \frac{\tau s +s^{\alpha-1}}{\tau s^2 +s^{\alpha}+a} \,ds
\end{equation} 
where $Br\subset \mathbb{C}$ is a Bromwich contour:
\begin{equation}\label{Br:cont}
Br:=
\{s: \ \mbox{Re}(s)=\sigma\geq 1 \ \& \ \mbox{Im}(s)\in (-\infty,\infty)\}.
\end{equation}
Moreover, we restrict to the principal representation of $s$, namely, here $\mbox{arg}(s)\in(-\pi,\pi]$. Following the approach in \cite{GM2} and \cite{BK}
and denoting by $s_{1}$ and $s_{2}=\overline{s_{1}}$ the zeros of (\ref{gen:right}) with $b=1$, which are the poles of the integrand in (\ref{inv:L:vtilde}).
  The contribution to the integral of these poles can be computed 
away from the Riemann surface cut (since $\alpha\in (0,1)$) that is the negative part of the real line. One can then split the integral as 
follows:  
\begin{equation}\label{v:expression}
v(\eta) =\frac{a\sin(\alpha \pi)}{\pi}\int_0^{\infty}e^{-\eta r} K(r)\, dr
+ 2\mbox{Re}\left( e^{s_1 \eta} \frac{\tau s_1 +s_1^{\alpha-1}}{2\tau s_1 +\alpha s_1^{\alpha-1}} \right)\,,
\end{equation}
where
\begin{equation}\label{Ks}
K(r)= r^{\alpha-1}\tilde{K}(r)
\quad \mbox{with} \quad \tilde{K}(r)= \frac{ 1 }{ (\tau r^2 + a)^2 
+ 2(\tau r^2 + a)r^\alpha \cos(\alpha\pi)+r^{2\alpha}}\,.
\end{equation} 
The integral term is bounded since $K\in L^1((0,\infty))$.
% see also the generalized Riemann-Lebesgue theorem \cite[Theorem 222]{Sir71}.
The asymptotic behaviour of the integral term for $\eta\to\infty$ can be studied by a refined Watson's Lemma in \cite[p. 65]{Sir71} and \cite[\S 4]{BleHan}. We note that the function $\tilde{K}(r)$ is not differentiable at $r=0$, but we have that for a small $\e>0$, the properties $K\in L^1((0,\infty))$ and $K =o(r^{\alpha-1+\e})$ for $r\to 0$ imply that $\LL(K)(\eta) =o(\eta^{-\alpha-\e})$ for $\eta \to \infty$. Using a Puiseux series expansion of $K(r)$ for $r\to 0$, allows to deduce for $\eta \to \infty$,
\begin{equation}\label{v:infty}
 \int_0^{\infty}e^{-\eta r} K(r) \, dr
 =\frac{\Gamma(\alpha)}{a^2} \frac{1}{\eta^{\alpha}}+ O(\eta^{-2\alpha}).
\end{equation}

Finally, we provide the next result of continuity of $v(\eta)=v(\eta;\tau)$ with respect to $\tau>0$ and the asymptotic behaviour at the origin.
\begin{lem}\label{v:cont:behav}
For $a>0$ and $\tau>0$, let $v(\eta;\tau)$ be the function defined by (\ref{v:expression})-(\ref{Ks}) for $\eta>0$. Then $v(\eta;\tau)$ as a function of $\tau$ is continuous for $\tau>0$. We also have that $v$ is uniformly bounded with respect to $\eta>0$ and have the following behaviour for $\eta \to 0$,
\[
v(\eta) \sim 1- \frac{a}{2\tau} \eta^2 \, , \ \ 
v'(\eta) \sim -\frac{a}{\tau} \eta\, , \ \ \ \text{as} \ \ \eta \to 0\, .
\]
\end{lem}
\begin{proof}
The continuity is obtained by the Dominated Convergence Theorem and taking into consideration that the integrand of (\ref{inv:L:vtilde}) is continuous for $\tau>0$ since the denominator has two complex conjugate roots with negative real part (see Lemma~\ref{roots}).

While the asymptotic behaviour at $\eta \to 0$ is derived applying the Initial Value Theorem and computing the limits in the Laplace transform variable (see e.g. \cite[Chapter 2]{ABHN} for more information). For $\tau>0$, we have:
\[%\begin{equation}\label{v:laplace}
\lim_{\eta\to 0} v(\eta)= \lim_{s\to +\infty} s \LL(v)(s)
= \lim_{s\to +\infty} s \frac{\tau s + s^{\alpha-1}}{\tau s^2 +s^{\alpha} +a} = 1\,,
\]%\end{equation}
\[%\begin{equation}\label{v:prime:laplace}
\lim_{\eta\to 0} v'(\eta)%= \lim_{s\to +\infty} s \LL(v')(s)
= \lim_{s\to +\infty} s \left( s \LL(v)(s) -v(0)\right)
=\lim_{s\to +\infty} \frac{-a\,s}{\tau s^2 +s^{\alpha} +a}  = 0\,,
\]%\end{equation}
and
\begin{equation}\label{v:2prime:laplace}
\lim_{\eta\to 0} v''(\eta)%= \lim_{s\to +\infty} s \LL(v'')(s)
= \lim_{s\to +\infty} s \left( s^2 \LL(v)(s) -sv(0)-v'(0)\right)
=\lim_{s\to +\infty} \frac{-a\,s^2}{\tau s^2 +s^{\alpha} +a}  = -\frac{a}{\tau} \,.
\end{equation}
Therefore, we get the desired asymptotic behaviour by the Taylor series expansion formula at the origin.
\end{proof}

%------------------------------------------------
\subsection{Monotonicity of $v$ for small values of $\tau$}\label{v:monotone}
In this section we study the behaviour of $v$, $v'$ and $v''$. 
The main idea here is that one can absorb the non-monotone part of each function into the monotone part for $\tau>0$ sufficiently small. 
From \cite[Lemma 13 (iii)]{ACH}, we know that the three functions are uniformly bounded on $[0,\infty)$, the first one by a constant independent of $\tau$ and the other two by a constant dependent of $\tau$ which gets unbounded as $\tau \to 0^+$.

\begin{lem}\label{v:bhv:tau:small}
For $a>0$ and $\tau>0$, let $v(\eta)$ be the function defined by (\ref{v:expression})-(\ref{Ks}) for $\eta>0$. Then for $\tau>0$ sufficiently small, $0<v(\eta)<1$, $v'(\eta)<0$ for all $\eta > 0$. Moreover, there exists some $\eta_{inflex} \sim \tau^{1/(2-\alpha)}$ as $\tau \to 0^+$, such that 
\begin{equation}\label{v'':sign}
v''(\eta) < 0 \quad \mbox{for} \quad 0<\eta < \eta_{inflex}
\quad \mbox{and} \quad v''(\eta) > 0 \quad  \mbox{for}
\quad  \eta > \eta_{inflex}.
\end{equation}
%Regarding the asymptotic behaviour, we have:
Also, there exists some $\eta_0 \sim \tau^{\frac{1}{2-\alpha}}$ as $\tau \to 0^+$ such that
\begin{equation}\label{bh:pass:inflex}
v(\eta)\sim \frac{K(\tau)}{\eta^\alpha}, \quad v'(\eta) \sim -\frac{K'(\tau)}{\eta^{\alpha+1}}%, \quad v''(\eta) \sim \frac{1}{\eta^{\alpha+2}}
\quad \mbox{for all} \quad \eta>\eta_0
\end{equation}
with $K(\tau)$, $K'(\tau)\sim\tau^{\frac{2\alpha}{2-\alpha}}$ as $\tau\to 0^+$, and
\begin{equation}\label{bh:pass:inflex:vpp}
  \lim_{\eta \to +\infty }v''(\eta)= 0.
\end{equation}

Finally, for $\tau\ll 1$ and $\eta\to 0^+$, valid in a layer of $\eta$ of order $\tau^{\frac{1}{2-\alpha}}$, we have
\begin{equation}\label{v:eta0:corr}
v(\eta) \sim 1 - \frac{a}{2\tau}\eta^2 + \frac{1}{(4-\alpha)(3-\alpha)(2-\alpha)} \frac{a}{\tau^2}\eta^{4-\alpha}\quad \mbox{as}\quad \eta\to 0^+\,,
\end{equation}
\begin{equation}\label{v:prime:eta0:corr}
v'(\eta)\sim- \frac{a}{\tau}\eta + \frac{1}{(3-\alpha)(2-\alpha)} \frac{a}{\tau^2}\eta^{3-\alpha}
\quad \mbox{as}\quad \eta\to 0^+
\end{equation}
and 
\begin{equation}\label{v:2prime:eta0:corr}
v''(\eta) \sim  - \frac{a}{\tau}+ \frac{1}{2-\alpha} \frac{a}{\tau^2}\eta^{2-\alpha}\quad \mbox{as}\quad \eta\to 0^+\,.
\end{equation}
\end{lem}
\begin{proof}
Considering the expression (\ref{v:expression}) of $v(\eta)$, one can get the following upper and lower bounds for the integral term:
\begin{equation}\label{v:int:bd:up}
\begin{split}
  &\int_0^{\infty} e^{-\eta r} \frac{r^{\alpha-1}}{(\tau r^2+a)^2+2(\tau r^2+a) r^\alpha \cos(\alpha\pi) + r^{2\alpha} }  \,dr \\ \leq &
  \int_{0}^{\infty}e^{-\eta r} \frac{r^{\alpha-1}}{(\tau r^2 +a)^2\sin^2(\alpha\pi)} \, dr 
\leq \frac{1}{a^2\sin^2(\alpha\pi)} \Gamma(\alpha) \frac{1}{\eta^\alpha}.
\end{split}
\end{equation}
In order to get the first inequality, we rewrite the denominator as 
\[
\begin{split}
&(\tau r^2+a)^2+2(\tau r^2+a) r^\alpha \cos(\alpha\pi) + r^{2\alpha} \\ =& \left((\tau r^2 + a)\cos(\alpha\pi) + r^\alpha\right)^2 + (\tau r^2 +a)^2\sin^2(\alpha\pi),
\end{split}
\]
while the last one is obtained computing the minimum of the denominator which is attained at zero and applying the change of variable $\eta r=r'$. On the other hand, taking into account that the integrand is non-negative and proceeding in the same way, one gets this lower bound, for any $0\leq A<B$
\begin{equation}\label{v:int:bd:low}
  \int_0^{\infty} e^{-\eta r} r^{\alpha-1}\tilde{K}(r) dr
  \geq \int_{A}^{B} e^{-\eta r} \frac{r^{\alpha-1}}{(\tau r^2 +a + r^\alpha)^2} \, dr
  \geq  \frac{e^{-\eta B} (B^\alpha-A^\alpha)}{\alpha(\tau B^2 +a + B^\alpha)^2 }\,.
\end{equation}

We rewrite the second term in (\ref{v:expression}) as follows:
\begin{equation}\label{Real:part:kernel}
\mbox{Re}\left(e^{s_1\eta}\frac{\tau s_1 + s_1^{\alpha-1}}{2\tau s_1 + \alpha s_1^{\alpha-1}}\right)
= e^{p\eta} \left(C_1 \cos(q\eta) + C_2 \sin(q\eta) \right), 
\end{equation}
thus $p=\mbox{Re}(s_1)<0$ and $q=\mbox{Im}(s_1)$,
and
\[
C_1 =\mbox{Re} \left(\frac{\tau s_1 + s_1^{\alpha-1}}{2\tau s_1 + \alpha s_1^{\alpha-1}}\right) ,\quad C_2 =-\mbox{Im} \left(\frac{\tau s_1 + s_1^{\alpha-1}}{2\tau s_1 + \alpha s_1^{\alpha-1}}\right).
\]

%UPPER BOUND v:
With this notation, we apply the upper bound (\ref{v:int:bd:up}) in
(\ref{v:expression}), to get (observe that $\sin\alpha\pi>0$): 
\[
v(\eta ) \leq C_r(\alpha) \frac{1}{\eta^\alpha}+ 2 e^{p\eta} C(\tau)\,,
\]
with constants
\[
C_r(\alpha) = \frac{\Gamma(\alpha)}{\pi a\sin(\alpha\pi)} , \ C(\tau)=|C_1|+|C_2|\,.
\]
Observe that the maximum of the function
$C_r(\alpha) + 2\eta^{\alpha} e^{p\eta} C(\tau)$ is attained at
$\eta_{max}=-\alpha/p>0$, thus
\begin{equation}\label{upp:v:1}
  v(\eta ) \leq \frac{1}{\eta^\alpha} \left(
  C_r(\alpha) + 2\left(-\frac{\alpha}{p}\right)^{\alpha}e^{-\alpha} C(\tau)
  \right)\,.
 \end{equation}

We observe that for $\tau\ll 1$ the constant $C(\tau)$ is of order
$\tau^{\frac{\alpha}{2-\alpha}}$. We deduce this fact by applying
(\ref{complex:tau:small:right}) of the
Appendix~\ref{appendix:roots}\footnote{
Taking into account that $\tau s_1^2+s_1^\alpha=-a$, then, as $\tau\to 0^+$,
\[
C_1=
\mbox{Re} \left(\frac{-a}{-2a-(2-\alpha)s_1^\alpha}
\right) = \mbox{Re} \left(\frac{a}{2a+(2-\alpha)s_1^\alpha}
\right) =
\frac{2a^2+a(2-\alpha)\mbox{Re}(s_1^\alpha)}{|2a+(2-\alpha)s_1^\alpha|^2}
= O(\tau^{\frac{\alpha}{2-\alpha}})\,.
\]
To leading order, the sign of $\mbox{Re}(s_1^\alpha)\sim \cos(\alpha \pi/(\alpha-2))\tau^{-\alpha/(2-\alpha)}$ as $\tau\to 0^+$ depends on $\alpha$: positive when $\alpha<2/3$, negative when $\alpha>2/3$, zero when $\alpha=2/3$.

We also have, as $\tau\to 0^+$,
\[
C_2 = - \mbox{Im} \left(\frac{-a}{-2a-(2-\alpha)s_1^\alpha}
\right) = -\mbox{Im} \left(\frac{a}{2a+(2-\alpha)s_1^\alpha}
\right) =  \frac{a(2-\alpha)\mbox{Im}(s_1^\alpha)}{|2a+(2-\alpha)s_1^\alpha|^2}
= O(\tau^{\frac{\alpha}{2-\alpha}})
\]
which is negative to leading order as $\tau\to 0^+$, since $\sin(\alpha\pi/(\alpha-2)) <0$
(see the expansion of $s_1$ with $b=1$ in (\ref{complex:tau:small:right}) and that $\mbox{Re}( \overline{s_1} s_1^{\alpha-1}) = \mbox{Re}( s_1\overline{s_1^{\alpha-1}})\sim \tau^{-\alpha/(2-\alpha)} (\cos(\pi/(\alpha-2))\cos((\alpha-1)\pi/(\alpha-2))+\sin(\pi/(\alpha-2))\sin((\alpha-1)\pi/(\alpha-2)))+\dots=\tau^{-\alpha/(2-\alpha)}\cos(\pi/(\alpha-2)-(\alpha-1)\pi/(\alpha-2))=-\tau^{-\alpha/(2-\alpha)}$).
}.

This is a good estimate for sufficiently large $\eta$. For small and large
values of $\eta$ we have a uniform, in $\tau\in [0,1]$, upper bound.
Indeed, for $\alpha\in(0,1/2)$, $\cos(\alpha\pi)>0$ and we obtain
\[%\label{vint:upper:bnd:2}
\begin{split}
  \int_0^{\infty}  e^{-\eta r} r^{\alpha-1} \tilde{K}(r)\,dr &\leq
  \frac{1}{\sin^2(\alpha \pi)} \int_0^1  \frac{r^{\alpha-1}}{(\tau r^\alpha +a)^2} \,dr +
  \int_1^{\infty} \frac{r^{\alpha-1}}{a^2 + r^{2\alpha}}\,dr \\ 
&\leq  \frac{1}{\alpha a^2 \sin^2(\alpha \pi)} + \frac{1}{\alpha}
\end{split}
\]
For $\alpha\in[1/2,1)$, $\cos(\alpha\pi)\leq 0$ and we obtain
\begin{align*}%\label{vint:upper:bnd:2}
  \int_0^{\infty}  e^{-\eta r} r^{\alpha-1} \tilde{K}(r)\,dr &\leq
  \frac{1}{\sin^2(\alpha \pi)} \int_0^R  \frac{r^{\alpha-1}}{(\tau r^\alpha +a)^2} \,dr +
  \int_R^{\infty} \frac{r^{\alpha-1}}{(\tau r^2 +a - r^{\alpha})^2}\,dr  \\
%  \frac{R^\alpha}{\alpha a^2 \sin^2(\alpha \pi)} + &\sup_{r\in(R,\infty)}\left\{\frac{r^{2\alpha}}{(\tau r^2 + a -r^\alpha)^2}\right\}\int_R^\infty \frac{dr}{r^{\alpha+1}}=\\
 \leq \frac{R^\alpha}{\alpha a^2 \sin^2(\alpha \pi)} + &\sup_{r\in(R,\infty)}\left\{\frac{r^{2\alpha}}{(\tau r^2 + a -r^\alpha)^2}\right\}  \frac{R^{-\alpha}}{\alpha}%\int_R^\infty \frac{dr}{r^{\alpha+1}}
\end{align*}
%NOTE=the max is achieve at alpha/2, the so if R>alpha/2 the sup is at r=R
where $R$ is larger that the positive root of $\tau r^2 +a -  r^{\alpha}$.
Then, for all $\alpha\in(0,1)$ and $\eta>0$ with $\tau\ll 1$, 
 \begin{equation}\label{upp:v:2}
 v(\eta)\leq  C +  C(\tau) + O(\tau)\,.
 \end{equation}
 for an order one constant $C$, and $C(\tau)=O(\tau^{\alpha/(2-\alpha)})$.
 The same bounds, clearly hold replacing $v(\eta)$ by $|v(\eta)|$. Let us see that, indeed $v(\eta)>0$ for all $\eta$ if $\tau$ is sufficiently small. 

% LOWER BOUND v:
 We also observe that although $p<0$, in the limit $\tau\to 0^+$ we have the following behaviours:
%\noindent
  %{\bf CASE I.}
%If $\eta\ll \tau^{\frac{1}{2-\alpha}}$, then,
\[
\eta \mbox{Re}(s_1)\to 0, \quad  \mbox{if}\quad \eta\ll \tau^{\frac{1}{2-\alpha}},  
\]
%\noindent
%{\bf  CASE II.}
%if, $\eta\sim \tau^{\frac{1}{2-\alpha}}$, then: 
\[
\eta \mbox{Re}(s_1) \to -C,  \quad  \mbox{if}\quad\eta\sim \tau^{\frac{1}{2-\alpha}}
\]
and %if
%\noindent
%{\bf CASE III.}
%$\eta> \tau^{\frac{1}{2-\alpha}}$, then,
\[
\eta \mbox{Re}(s_1) \to -\infty \quad  \mbox{if}\quad\eta> \tau^{\frac{1}{2-\alpha}}.
\]
This follows from Lemma~\ref{roots} in Appendix~\ref{appendix:roots}.

In the third case we then have as a lower bound for $v$, using (\ref{v:int:bd:low}) with $A=0$ and $B=1/\eta$
\[
\frac{1}{\eta^\alpha}\left(\frac{a\sin(\alpha \pi)}{e\alpha \pi}
  \frac{ \eta^4}{(\tau  +a\eta^2 + \eta^{2-\alpha})^2 }
- 2 e^{p\eta}\eta^\alpha C(\tau)\right) \leq
v(\eta )\,. 
\]
Since the function $\frac{ \eta^4}{(\tau  +a\eta^2 + \eta^{2-\alpha})^2 }$ is
increasing and the minimum of the second term is attained at
$\eta_{max}=-\alpha/p=O(\tau^{1/(2- \alpha)})$ we have that there exists
$\eta_0 >\eta_{max}$ with $\lim_{\tau\to 0^+} \eta_0/\eta_{max}=\infty$, such that 
\[
0< \frac{1}{\eta^\alpha} \left(\frac{a\sin(\alpha \pi)}{e\alpha \pi}
\frac{ \eta_0^4}{(\tau  +a\eta_0^2 + \eta_0^{2-\alpha})^2 }
-  2 e^{p\eta_0}\eta_0^\alpha C(\tau) \right) 
\leq v(\eta ) \quad \mbox{for all} \quad \eta \geq \eta_0.
\]
We can improve this for $\eta_0=K\eta_{max}$ with $K>1$ sufficiently large,
with the estimate
\[
0< \frac{1}{\eta^\alpha} \left(\frac{a\sin(\alpha \pi)}{e\alpha \pi}
\frac{ \eta_0^4}{(\tau  +a\eta_0^2 + \eta_0^{2-\alpha})^2 }
-  2 e^{p\eta_{max}}\eta_{max}^\alpha C(\tau) \right) 
\leq v(\eta ) \quad \mbox{for all} \ \eta \geq \eta_0
\]
since both terms are of order $\tau^{2\alpha/(2-\alpha)}$ in that case,
but the first has the freedom for $K$ which can be made large\footnote{
\[
\frac{ \eta_0^4}{(\tau  +a\eta_0^2 + \eta_0^{2-\alpha})^2 }
\sim \tau^{2\alpha/(2-\alpha)}\frac{K^4}{(1+a\alpha^2C^2 K^2\tau^{\alpha/(2-\alpha)}
  + K^{2-\alpha} \alpha^{2-\alpha}C^{2-\alpha} )^2}, \quad \tau\to 0^+\,.
\]} but independent of $\tau$. This, (\ref{upp:v:1}) and (\ref{upp:v:2}) imply the first behaviour given in (\ref{bh:pass:inflex}).

More generally, from (\ref{v:int:bd:low}) with $A=0$, take $0<B< |p|$, then,
\[
 e^{-B\eta}\left( \frac{a\sin(\alpha \pi)}{\pi}
  \frac{B^\alpha}{\alpha(\tau B^2 +a + B^\alpha)^2 }
- 2 e^{\eta (-|p|+B)}C(\tau)\right) \leq
v(\eta ), 
\]
so, for all $\eta\in(0,\eta_0)$\[
 e^{-B\eta}\left( \frac{a\sin(\alpha \pi)}{\pi}
  \frac{B^\alpha}{\alpha(\tau B^2 +a + B^\alpha)^2 }
- 2 C(\tau)\right) \leq
v(\eta ), 
\]
then we can choose $B=O(1)$ as $\tau\to 0^+$ to guarantee that the
right-hand side is positive. This shows that $v(\eta)>0$ for all $\eta>0$ if $\tau$ is small enough. 
%CHECK AGAIN, knowing now that $C(\tau) \to 0$

Let us now get the behaviour and bounds on $v'(\eta)$. Using (\ref{v:expression}) and (\ref{Real:part:kernel}) we derive
\[
\begin{split}
  v'(\eta) =&
  - \frac{a\sin(\alpha\pi)}{\pi}\int_0^\infty e^{-r\eta} r^\alpha \tilde{K}(r)dr \\
  &+ 2 e^{p \eta}\left( (p\cos(q\eta)-q\sin(q\eta))C_1
  + (p\sin(q\eta)+q\cos(q\eta)C_2 \right) \\
  =&-\frac{a\sin(\alpha\pi)}{\pi}\int_0^\infty e^{-r\eta} r^\alpha \tilde{K}(r)dr
  + \mbox{Re}\left(e^{iq\eta} s_1
  \frac{\tau s_1 + s_1^{\alpha-1}}{2\tau s_1 +\alpha s_1^{\alpha-1}}\right)e^{p\eta}\\
  =&-\frac{a\sin(\alpha\pi)}{\pi}\int_0^\infty e^{-r\eta} r^\alpha \tilde{K}(r)dr
  -\mbox{Re}\left(e^{iq\eta} \frac{a}{2\tau s_1 + \alpha s_1^{\alpha-1}}\right)e^{p\eta}\,.
\end{split}
\]
Then, we have the lower bounds for $v'$ (that are obtained in a similar way as the upper bounds on $v$ (\ref{upp:v:1}) and (\ref{upp:v:2})):
\begin{equation}\label{vprime:lower:1}
  v'(\eta) \geq
  - \frac{1}{\eta^{\alpha+1}} \left(\frac{\Gamma(\alpha+1)}{a\pi\sin(\alpha\pi)}
 + \left(\frac{\alpha+1}{|p|}\right)^{\alpha+1} e^{-(\alpha+1)}C'(\tau)\right)
\end{equation}
and
\begin{equation}\label{vprime:lower:2}
  v'(\eta)\geq
  - \frac{\tau^{-1}}{2(1-\alpha)\sin(\alpha \pi)}
  - \frac{1}{(\alpha+1)\pi a\sin(\alpha \pi)}   - 2 C'(\tau) e^{p \eta} \,,
\end{equation}
where $C'(\tau) =2 (|C_1| +|C_2|)(|p|+|q|)= O(\tau^{-\frac{1-\alpha}{2-\alpha}})$
as $\tau\to 0^+$ (we can even improve the first term\footnote{
Here we use for $\tau\ll 1$
  \begin{align*}%\label{vint:upper:bnd:2}
  \int_0^{\infty}  e^{-\eta r} r^{\alpha} \tilde{K}(r)\,dr &\leq
  \frac{1}{a^2\sin^2(\alpha \pi)} \int_0^1  r^{\alpha} \,dr +
  \int_1^{\infty} \frac{r^\alpha}{(\tau r^2+a)^2\sin^2(\alpha \pi)}\,dr \\
  \leq \frac{1}{(\alpha+1) a^2 \sin^2(\alpha \pi)}
  &+ \frac{1}{2 \tau a\sin^2(\alpha \pi)}\int_1^{\infty} r^{\alpha-2}\,dr\,.
  \end{align*}
We can also use:
 \begin{align*}%\label{vint:upper:bnd:2}
  \int_0^{\infty}  e^{-\eta r} r^{\alpha} \tilde{K}(r)\,dr &\leq
  \frac{1}{2a\tau\sin^2(\alpha \pi)} \int_0^\infty  e^{-\eta r} r^{\alpha-2} \,dr
  \leq C\frac{\eta^{1-\alpha}}{\tau}
 \end{align*}
or
 \begin{align*}%\label{vint:upper:bnd:2}
  \int_0^{\infty}  e^{-\eta r} r^{\alpha} \tilde{K}(r)\,dr &\leq
  \frac{1}{a^2\tau^2\sin^2(\alpha \pi)} \int_0^\infty  e^{-\eta r} r^{\alpha-4} \,dr
  \leq C\frac{\eta^{3-\alpha}}{\tau^2}\,.
 \end{align*}
}).

%UPPER BOUNDS v'

Then, we use that for any $0\leq A<B$
\begin{equation}\label{vprime:int:bd:low}
  \int_0^{\infty} e^{-\eta r} r^{\alpha}\tilde{K}(r) dr
  \geq \int_{A}^{B} e^{-\eta r} \frac{r^{\alpha}}{(\tau r^2 +a + r^\alpha)^2} \, dr
  \geq  \frac{e^{-\eta B} (B^{\alpha+1}-A^{\alpha+1})}{(\alpha+1)(\tau B^2 +a + B^\alpha)^2 }.
\end{equation}
Taking $A=0$ and $B=1/\eta$, we get that there exists
$\eta_0'>\eta_{max}'=-(\alpha+1)/p$ such that, for all $\eta\geq\eta_0'$,  
\[
v'(\eta)\leq - \frac{1}{\eta^{\alpha+1}}\left(
\frac{a\sin(\alpha \pi)}{e(\alpha+1) \pi}
\frac{ (\eta_0')^4}{(\tau  +a(\eta_0')^2 + (\eta_0')^{2-\alpha})^2 }
- 2 e^{p\eta_0'}(\eta_0')^{\alpha+1}C'(\tau)
\right)<0\,.
\]
Which can again be improved, as before, for $\tau$ sufficiently small,
for $\eta_0'=K'\eta_{max}'$ for some $K'>1$ sufficiently large, but independent of $\tau$, so that for all
$\eta\geq \eta_0'$, 
\[
v'(\eta)\leq - \frac{1}{\eta^{\alpha+1}}\left(
\frac{a\sin(\alpha \pi)}{e(\alpha+1) \pi}
\frac{ (\eta_0')^4}{(\tau  +a(\eta_0')^2 + (\eta_0')^{2-\alpha})^2 }
- 2 e^{p\eta_{max}'}(\eta_{max}')^{\alpha+1}C'(\tau)
\right)<0. 
\]
In this case the term in brackets is also of the order $\tau^{\frac{2\alpha}{2-\alpha}}$ as $\tau \to 0^+$. This implies the second estimate in (\ref{bh:pass:inflex}). 

The behaviour of $v''(\eta)$ for $\eta>\eta_0$ can be deduced similarly, but it is enough for our purposes to just get (\ref{bh:pass:inflex:vpp}). This can be done by using similar estimates as in (\ref{v:int:bd:up}) and (\ref{Real:part:kernel}).

Let us now show that $v'(\eta)<0$ for $\eta\leq \eta_0$. For very small values of $\eta$, we can take $B=K|p|$ and choose $K>1$
large enough, but independent of $\tau$, so that for $\tau$ sufficiently small, we get that
\[
\left( \frac{a\sin(\alpha \pi)}{e(\alpha+1)\pi}
  \frac{ B^{\alpha+1}}{(\tau B^2 +a + B^\alpha)^2 }
- 2 C'(\tau)\right) >0 
\]
observe that both terms are of the same order as $\tau\to 0^+$, but making $K$
large makes the first larger. Thus, there exists $K>1$ large enough so that
for all $0<\eta\leq (K|p|)^{-1}$ and for $\tau$ small enough, we have
$v'(\eta)\leq 0$, which together with the previous estimate implies that $v$ decreases for all $\eta>0$.

A next order correction of the behaviours given in Lemma~\ref{v:cont:behav} and of that implied by (\ref{v:2prime:laplace}), can be obtained from a similar computation, as in the proof of this Lemma, of the third derivative. Putting all together we get the expansion (\ref{v:2prime:eta0:corr}) for $v''$. Thus there is a value $\eta_{inflex}\ll 1$ if $\tau$ is sufficiently small such that $v''(\eta_{inflex})=0$ and has
\begin{equation}\label{eta:inflex:bh}
\eta_{inflex} \sim (2-\alpha)^{\frac{1}{2-\alpha}} \tau^{\frac{1}{2-\alpha}} \quad\mbox{as}\quad \tau\to 0^+.
\end{equation}
Then, from this and (\ref{v:2prime:eta0:corr}) we deduce (\ref{v:prime:eta0:corr}), and then (\ref{v:eta0:corr}). These limits are valid as long as $\eta\leq \eta_{inflex}$ for $\tau$ small enough but positive.

From the linear equation satisfied by $v$, which is $\tau v'' + \DD_0^\alpha[v] + a v=0$ with $v(\eta)>0$ and $v'(\eta)<0$ for all $\eta>0$ if $\tau$ is sufficiently small, we deduce that $\tau v'' + a v>0$ for all $\eta>0$ if $\tau$ is sufficiently small. If initially $v''<0$ and on the other hand $v>0$ decreases for all $\eta$, then $v''$ must change sign. The regularity of $v$ and the fact that $\eta_{max}$ and $\eta_{max}'$ are of order $\tau^{\frac{1}{2-\alpha}}$ as $\tau\to 0^+$, imply that for $\tau$ small enough this change of sign of $v''$ occurs only once, and that might be given by (\ref{eta:inflex:bh}).

We observe that there is a boundary layer of size $O(\tau^{1/(2-\alpha)})$ as $\tau\to 0^+$. In particular the behaviours obtained above are consistent for $\eta\sim \eta_{inflex}$ with the behaviour of the corresponding solution $v_0$ of the linear problem when $\tau=0$ as $\eta\to 0$, since (see e.g. \cite{GM2})
\[
\lim_{\eta\to 0} v_0(\eta)= 1
\quad
\mbox{and}\quad
v_0'(\eta) \sim - a\eta^{\alpha-1} \quad \mbox{as}\quad \eta \to 0^+.
\]

\end{proof}

%\bibliography{frac}

\end{document}